\documentclass[11pt, a4paper]{article}

\newcommand{\blind}{0}
\newcommand{\arxiv}{1}

\if0\arxiv
\newcommand\spacevar{1.45}
\fi

\if1\arxiv
\usepackage{a4wide}
\newcommand\spacevar{1}
\fi

\usepackage[T1]{fontenc}
\usepackage[english]{babel}
\usepackage[utf8]{inputenc}
\usepackage{amsmath,amssymb,amsfonts,mathrsfs}
\usepackage{mathtools}
\usepackage[amsmath,thmmarks]{ntheorem}
\usepackage{dsfont}
\usepackage{paralist}
\usepackage{enumerate}
\usepackage{paralist}
\usepackage{multicol}
\usepackage{color}
\usepackage[hidelinks]{hyperref}
\usepackage{tikz}
\usepackage{tkz-graph}
\usetikzlibrary{shapes.geometric}
\usetikzlibrary{backgrounds}
\usepackage{tabularx}
\usepackage{ragged2e}
\newcolumntype{Y}{>{\RaggedRight\arraybackslash}X}
\usepackage{booktabs}
\usepackage[framemethod=TikZ]{mdframed}
\usepackage[ruled]{algorithm2e}
\usepackage[toc,page,header]{appendix}

\newcommand{\R}{\mathbb{R}}
\newcommand{\N}{\mathbb{N}}

\DeclareMathOperator*{\argmax}{arg\,max}

\DeclarePairedDelimiterX{\norm}[1]{\lVert}{\rVert}{#1}
\DeclarePairedDelimiterX{\abs}[1]{\lvert}{\rvert}{#1}

\renewcommand{\epsilon}{\varepsilon}

\renewcommand{\P}{\mathbb{P}}
\newcommand{\E}{\mathbb{E}}
\newcommand{\iid}{\overset{\text{\tiny iid}}{\sim}}
\newcommand{\independent}{\perp\!\!\!\perp}

\newcommand{\var}{\operatorname{Var}}

\newcommand{\vI}{\mathbf{Id}}
\newcommand{\vNull}{\mathbf{0}}
\newcommand{\landauO}{\mathcal{O}}
\newcommand{\landauo}{o}
\newcommand{\landauOp}{\mathcal{O}_{\P}}
\newcommand{\landauop}{o_{\P}}
\newcommand{\landauOLp}{\mathcal{O}_{\operatorname{L}^p}}

\newcommand{\vD}{\mathbf{D}}

\newcommand{\vP}{\mathbf{P}}

\newcommand{\vV}{\mathbf{V}}
\newcommand{\vX}{\mathbf{X}}
\newcommand{\vY}{\mathbf{Y}}
\newcommand{\vZ}{\mathbf{Z}}
\newcommand{\vx}{\mathbf{x}}

\newcommand{\vz}{\mathbf{z}}
\newcommand{\vepsilon}{\boldsymbol{\epsilon}}

\newcommand{\CPt}{\operatorname{CP}}

\newcommand{\hbeta}{\hat{\beta}}

\newcommand{\hgamma}{\hat{\gamma}}
\newcommand{\resid}{\mathbf{R}}
\newcommand{\sresid}{\widetilde{\mathbf{R}}}
\newcommand{\noise}{\boldsymbol{\epsilon}}
\newcommand{\snoise}{\tilde{\boldsymbol{\epsilon}}}

\newcommand{\HOS}{H_{0,S}}

\newcommand{\HAS}{H_{A,S}}
\newcommand{\HO}{H_{0}}
\newcommand{\HA}{H_{A}}
\newcommand{\tildeHOS}{\widetilde{H}_{0,S,p}}

\newcommand{\pa}{\operatorname{PA}}

\numberwithin{equation}{section}

\theoremstyle{break}
\newtheorem{theorem}{Theorem}[section]
\newtheorem{example}[theorem]{Example}
\newtheorem{corollary}[theorem]{Corollary}
\newtheorem{definition}[theorem]{Definition}
\newtheorem{lemma}[theorem]{Lemma}
\newtheorem{proposition}[theorem]{Proposition}
\newtheorem{assumption}{Assumption}

\theoremstyle{plain}
\newtheorem{remark}[theorem]{Remark}

\theoremstyle{nonumberplain}
\theorembodyfont{\normalfont}
\theoremsymbol{\ensuremath{\square}}
\newtheorem{proof}{Proof}

\usepackage{natbib}
\if0\arxiv
\usepackage{xr-hyper}
\fi

\begin{document}

\def\spacingset#1{\renewcommand{\baselinestretch}%
{#1}\small\normalsize} \spacingset{1}

\if0\arxiv
\if0\blind
{
  \title{\bf Invariant Causal Prediction for Sequential Data}
  \author{Niklas Pfister\thanks{Partially funded by Swiss National
      Science Foundation (SNSF) No. 200021\_153504}\hspace{.2cm}, Peter Bühlmann\\
    Seminar for Statistics, ETH Zürich\\
    and \\
    Jonas Peters \\
    Department of Mathematical Sciences, University of Copenhagen}
  \maketitle
} \fi

\if1\blind
{
  \bigskip
  \bigskip
  \bigskip
  \begin{center}
    {\LARGE\bf Invariant Causal Prediction for Sequential Data}
  \end{center}
  \medskip
} \fi

\bigskip
\fi

\if1\arxiv
\author{Niklas Pfister, Peter B\"uhlmann and Jonas Peters}
\title{Invariant Causal Prediction for Sequential Data}
\maketitle
\fi

\begin{abstract}
  We investigate the problem of inferring the causal predictors of a
  response $Y$ from a set of $d$ explanatory variables
  $(X^1,\dots,X^d)$.  Classical ordinary least squares regression
  includes all predictors that reduce the variance of $Y$.  Using only
  the causal predictors instead leads to models that have the
  advantage of remaining invariant under interventions; loosely
  speaking they lead to invariance across different ``environments''
  or ``heterogeneity patterns''. More precisely, the conditional
  distribution of $Y$ given its causal predictors remains invariant
  for all observations. Recent work exploits such a stability to infer
  causal relations from data with different but known environments. We
  show that even without having knowledge of the environments or
  heterogeneity pattern, inferring causal relations is possible for
  time-ordered (or any other type of sequentially ordered) data. In
  particular, this allows detecting instantaneous causal relations in
  multivariate linear time series which is usually not the case for
  Granger causality. Besides novel methodology, we provide statistical
  confidence bounds and asymptotic detection results for inferring
  causal predictors, and present an application to monetary policy
  in macroeconomics.
\end{abstract}

\if0\arxiv
\noindent%
 \vfill
\fi

\if1\arxiv
\noindent%
\fi

\spacingset{\spacevar}
\section{Introduction}

Detecting causal relations is a core problem in many scientific
fields. Performing controlled randomized intervention experiments can
be considered the gold standard for inferring causal relations
\citep[e.g.,][]{Peirce1883, Pearl2009,
  Imbens2015, Peters2017book}. %
In many situations however, randomization and interventions are
unethical, physically impossible or too costly. In addition, many
datasets nowadays come from non-designed experiments: the question is
then whether one can still infer causal relations. Assuming additional
structure, this is indeed possible.

The field of causal structure learning attempts to infer causal
relations from data. Many procedures in this field only use
observational data
\citep[e.g.,][]{
Spirtes2000, 
Chickering2002, 
Shimizu2006, 
Janzing2012,
Peters2013anm,
vandegeer2013, 
Buehlmann2014cam}.
These methods are either based on strong assumptions (that are in
particular violated for heterogeneous data) or they do not output a
single causal graph estimate but a set of so-called Markov equivalent
graphs. The latter occurs because of severe identifiability problems
in general. Having access to interventional data improves
idenitifiability for inferring causal relations and
several methods have been proposed that exploit both observational and
interventional data \citep[e.g.,][]{Eaton2007, hauser2012,
  peters2016}. Causal discovery, however, is an ambitious task, and
most of the above methods only provide point estimates and lack
statistical confidence guarantees.

The problem of identifying causal directions is greatly reduced in the
time series setting, where the concept of Granger causality
\citep{Granger1969} plays a prominent role; see also structural vector
auto-regressive models (SVAR) that are popular in econometrics
\citep[e.g.,][]{lutkepohl2005}. When excluding instantaneous effects,
the time order allows applying regression techniques to infer causal
relations between variables. In this paper we consider the more
general problem of inferring also the time-instantaneous effects.

We consider a target variable $Y$ and covariates
$X^1, \ldots, X^d$.  Instead of reconstructing
the %
full causal structure, we thus try to infer the set
$S^* \subseteq \{1,\ldots ,d\}$ of causal predictors\footnote{In the
  context of causal graphical models, it is more common to use the
  term ``causal parents'' instead of ``causal predictors''. Here, we
  use the latter in order to emphasize the regression setting.} of
$Y$ (where the indices in $\{1,\ldots ,d\}$ refer to the indices of
the variables $X^1,\ldots ,X^d$). Our approach comes with the
following two advantages; (1) Most importantly, it provides a
statistical confidence guarantee: the method outputs an estimate
$\hat{S} \subseteq \{1, \ldots, d\}$ of the set of causal predictors
such that $\hat{S} \subseteq S^*$ with controllable (large) coverage
probability $1-\alpha$; (2) The method does not need to model
interdependence of the predictor variables $X^1, \ldots, X^d$, but
rather only the dependence of the causal variables from $S^*$ on the
target $Y$.

Our approach uses sequential data that are assumed to arise from a mix
of observational and interventional settings.
Given that type of data, we propose to look for invariant structures, i.e.,
conditionals that do not change over time.  To do so, we do not need
to know the nature or location of the intervention regimes.
A framework that connects stability (or invariance) to causality has
been recently formulated by \citet{peters2016} under the name {\it
  invariant causal prediction} which we will summarize next.
(Although the underlying principles coincide, \citet{peters2016} do
not consider sequential data, but assume knowledge of the location of
the different regimes.)

Invariant causal prediction considers the situation where one observes
the response and covariates $(Y_e, X^1_e, \ldots, X^d_e)$ in different
environments $e \in \mathcal{E}$, that is, in each of these
environments, we have an i.i.d.\ data set.  The crucial assumption is
that the conditional distribution of the response given the variables
from $S^*$ is the same in all environments: more formally, we have for
all $e, f \in \mathcal{E}$ and all $x$ that
\begin{equation}
  \label{eq:conditional_invariance}
  Y_e \,|\, \big(X^{S^*}_e = x\big) 
  \;\; \overset{d}{=} \; \;
  Y_f \,|\, \big(X^{S^*}_f = x\big).
\end{equation}
This assumption is satisfied, for example, if the environments
correspond to different intervention settings, that do not contain
a direct intervention on the target variable $Y$
\citep[][Proposition~1]{peters2016}. Here, we develop invariant causal
prediction in an environment-free way, i.e., without knowing the
different environments.

Summarizing the above comments, our method is applicable in the
following situation. There is a target variable $Y_t$ and a set
${S^*}$ of causal predictors that satisfies the following property:
for all $t$, $Y_t = X_t^{S^*}\beta + \epsilon_t$ for some Gaussian
i.i.d.\ sequence $\epsilon_t$ with $\epsilon_t \independent X_t^{S^*}$
(see Assumption~\ref{assumption:invariance} in
Section~\ref{sec:structural_invariance} for details). Our method aims
to estimate ${S^*}$. Furthermore, we do not need to assume that the
predictors $X_t$ are i.i.d.\ over $t$ as our method utilizes any
changes in distribution.

\subsection{Contribution and relation to other work}\label{subsec.relwork}

\citet{peters2016} assume that the environments are known in order to
exploit the invariance in \eqref{eq:conditional_invariance}.  Without
knowledge of the environments the task becomes more difficult.
Naively estimating the environments from data and subsequently
applying the existing methodology may lead to a loss of the method's
coverage guarantee or yield less powerful results. In particular, this
is the case for recovering the environments by clustering (see Remark
B.2) or by using change point detection methods (see discussion in
Section~\ref{sec:test_stat_cp} and Remark B.3). In contrast, our
procedure does not estimate the environments but instead utilizes the
existing non-invariances directly. It can thus be seen as a highly
non-trivial generalization of invariant causal prediction when the
environments are unknown.

From a technical perspective, we provide a new asymptotic analysis for
the Chow test \citep{chow1960} for simultaneously testing equality of
regression coefficients and homoscedasticity of the residuals. In
particular, we show that it has a non-optimal rate for detecting
differences of regression coefficients. As an alternative with better
rates, we propose using a decoupled version that individually tests
regression coefficients and residual variances, and combines them with
a Bonferroni correction. Finally, we employ a bootstrap procedure due
to \citet{shah2015} which allows for efficient multiple testing over
many smooth or block-wise time segments of the data.

The proposed causal inference methodology can be directly used for
multivariate auto-regressive time series, allowing for detection of
\emph{instantaneous} causal predictors. The notion of ``different
environments'' then translates to non-stationarity; in fact, it is the
non-stationary nature of the system which allows for detection of
instantaneous causal predictors, whereas a stationary process would
not provide the required ``perturbations'' to identify causality. Our
work is therefore different from the celebrated concept of Granger
causality for non-instantaneous causal relations
\citep[e.g.,][]{Granger1969, lutkepohl2005}. Other methods that
are able to identify instantaneous effects \citep[e.g.,][]{Chu2008,
  Hyvarinen2008, Peters2013nips} often require nonlinearities or
non-Gaussianity. Additionally, they need to model the full network and
do not come with any notion of causal significance.
Another line of work starts from high-frequency data sampled from a
Granger model without instantaneous effects and tries to infer the
instantaneous effects appearing in low-frequency sub-sampled data
\citep[e.g.,][]{gong15, tank17}. A further related area of research
extends non-instantaneous effects models by allowing for
time-dependent parameters \citep[e.g.,][]{talih05, siracusa09}. These
methods usually work with stationary data, model the full causal
system rather than one target variable, and do not come with
significance statements on their causal findings.

The paper is structured as follows. Section~\ref{sec:icp} introduces
invariant causal prediction in an environment-free way. Our method
and the confidence guarantees for detecting causal predictors are
described in Section~\ref{sec:scaled_residuals}. We establish
consistency and detection rates of this method in
Section~\ref{sec:consistency}.  Algorithmic details are given in
Section~\ref{sec:implementation}, with programming code available
online as an \texttt{R}-package. In Section~\ref{sec:timeseries}, we
extend the framework to multivariate time series data, and
Section~\ref{sec:experiments} reports on numerical experiments and an
application in macroeconomics for monetary policy.

\section{Invariant causal prediction}\label{sec:icp}

Throughout this work, we assume that we are given data from a sequence
$(Y_t,X_t)_{t\in\{1,\dots,n\}}$, where $X_t\in\R^{1\times d}$ contains
predictor variables and $Y_t\in\R$ is a target variable of
interest. Moreover, let
$\vY\coloneqq(Y_1,\dots,Y_n)^{\top}\in\R^{n\times 1}$ and
$\vX\coloneqq(X_1^{\top},\dots,X_n^{\top})^{\top}\in\R^{n\times d}$
denote the corresponding matrix quantities. We are interested in
settings where the experimental conditions are allowed to change over
time, as long as the structural dependence (predictor set and
regression parameters) of $Y_t$ on $X_t$ remains
fixed, which is the corresponding environment-free version of the
invariance assumption given in
\eqref{eq:conditional_invariance}. Ideally, we would like to make
direct use of this assumption for structural inference. However, in
order to have a reasonable amount of power to test such an assumption
based on a finite sample it is useful to describe the dependence of
$Y$ on its parents by a parametric function class. In this paper we
focus on linear Gaussian models but our ideas potentially also extend
to more complicated models.

\subsection{Structural invariance}\label{sec:structural_invariance}

In this subsection we formalize the fixed structural dependence
(predictor set and regression parameters) of $Y_t$ on $X_t$ to be
linear Gaussian.  We denote by
$(\vY,\vX)=(Y_t,X_t)_{t\in\{1,\dots,n\}}\in\R^{n\times (d+1)}$ the
random vectors corresponding to the entire data and for any set
$S\subseteq\{1,\dots,d\}$, the vector $X^S\in\R^{1\times\abs{S}}$
contains only the variables $\{X^k; k\in S\}$. We make the following
definition.
\begin{definition}[invariant set $S$]
  \label{def:invariantset}
  A set $S\subseteq\{1,\dots,d\}$ is called invariant with respect to
  $(\vY,\vX)$ if there exist parameters $\mu\in\R$,
  $\beta\in(\R\setminus\{0\})^{\abs{S}\times 1}$ and $\sigma\in\R_{>0}$ such that
  \begin{compactitem}
  \item[(a)] $\forall t\in\{1,\dots,n\}:\quad
    Y_t=\mu+X_t^S\beta+\epsilon_t\text{ and }\epsilon_t \independent
    X_t^{S}$, 
  \item[(b)] $\epsilon_1,\dots,\epsilon_n\iid \mathcal{N}(0,\sigma^2)$.
  \end{compactitem}  
\end{definition}
Throughout the paper, the symbol $\independent$ denotes independence
and we neglect the intercept term $\mu$ as it can be added without
loss of generality by including a constant term in $X$. For an
invariant set $S$ we thus have that conditionally, $Y_t\mid X_t^S$
($t=1,\dots,n$) are i.i.d. Gaussian random variables. It is crucial to
observe that Definition~\ref{def:invariantset} makes no restrictions
on the distribution of the process $(X_t)_{t\in\{1,\dots,n\}}$, and
also the distribution of $(\vY,\vX)$ can be quite general. In
particular, this allows for time dependencies and arbitrary changes in
the distribution of $X_t$.  In Remark~\ref{rmk:nonGaussianNoise} we
discuss a potential extension that allows to weaken the Gaussian
linear assumption.

Based on Definition~\ref{def:invariantset}, we can formalize
an invariance assumption similar to \citet[Assumption 1]{peters2016},
by requiring the existence of an invariant set $S^*$.
\begin{assumption}[structural invariance]
  \label{assumption:invariance}
  There exists a set $S^*\subseteq\{1,\dots,d\}$ which is invariant
  with respect to $(\vY,\vX)$.
\end{assumption}
The set $S^*$ can be seen as a set of predictor variables which
shields off $Y$ from any interventions on the system other than
interventions on $Y$ itself. In the setting of heterogeneous data, the
set $S^*$ then corresponds to the set of predictors that can be safely
included into a prediction model which works at all time points
$t\in\{1,\dots,n\}$. 

\subsection{Invariant prediction and coverage}

In this section we recall some definitions and results related
to invariant causal prediction from
\citet{peters2016}. Assumption~\ref{assumption:invariance} enforces the
existence of a set $S^*$ such that the structural dependence
(predictor set and regression parameters) between $Y_t$ and
$X_t^{S^*}$ remains fixed. Our goal is to estimate the set $S^*$ based
on the observed data $(\vY,\vX)$. We build this estimate by taking the
intersection of all sets $S\subseteq\{1,\dots,d\}$ which are invariant
with respect to $(\vY,\vX)$, i.e., we consider all such sets $S$ and
test the following null hypothesis
\begin{equation}
  \label{eq:basic_nullhypothesis}
  \HOS:\text{ S is an invariant set with respect to }(\vY,\vX).
\end{equation}
Based on this null hypothesis we define the %
set of plausible causal predictors
\begin{equation}
  \label{eq:Shat}
  \tilde{S}\coloneqq\bigcap_{\overset{S\subseteq\{1,\dots,d\}:}{\HOS\text{
    is true}}}S\subseteq S^*,
\end{equation}
where we define the intersection over an empty index set as the empty
set. The name plausible causal predictors is motivated by the underlying
causal interpretation explained in Section~\ref{sec:causalinterp}. The
property that $\tilde{S}$ is contained in $S^*$ follows immediately
from Assumption~\ref{assumption:invariance}. In general, this
containment is strict since there could be several sets other than
$S^*$ which satisfy the invariance condition across the considered
interventions. Hence, if we change the interventions in such a way
that the number of invariant sets decreases this leads to an increase
in the size of the set $\tilde{S}$.  Intuitively, an increase in
interventions results in an increase in the set $\tilde{S}$.

Empirically, we can use this to construct an estimator based on an
arbitrary family of hypothesis tests
$\varphi=(\varphi_S)_{S\subseteq\{1,\dots,d\}}$, where $\varphi_S$ is
the decision rule that either rejects $\HOS$ ($\varphi_S=1$) or
accepts $\HOS$ ($\varphi_S=0$). We then estimate
$\tilde{S}$ using the family of tests $\varphi$ by
\begin{equation}
 \label{eq:Shat2}
  \hat{S}(\varphi)\coloneqq\bigcap_{\overset{S\subseteq\{1,\dots,d\}:}{\,\varphi_S \text{ accepts }\HOS}}S.
\end{equation}
It is obvious that this estimator has the following coverage property,
given that the hypothesis tests achieve correct level.
\begin{proposition}[coverage property {\citep[Theorem~1]{peters2016}}]
  \label{thm:empirical_coverage}
  Assume Assumption~\ref{assumption:invariance} and let
  $\varphi=(\varphi_S)_{S\subseteq\{1,\dots,d\}}$ be a family of
  hypothesis tests for the null hypotheses
  $(\HOS)_{S\subseteq\{1,\dots,d\}}$ 
  which achieve level
  $\alpha\in(0,1)$.  Then, %
 $ \P\left(\hat{S}(\varphi)\subseteq
    S^*\right)\geq 1-\alpha$.
\end{proposition}

In the following paragraph we compare our framework with that of
\citet{peters2016}. While most parts are similar, the main
difference is that we have phrased our framework without the use of
environments.  \citet{peters2016} is then contained as a special case,
more precisely their Equation~(10) is contained within our null
hypothesis~\eqref{eq:basic_nullhypothesis}.  Our more general
formulation comes with three benefits: (1) It allows a mathematically
rigorous treatment of data that are generated by systems which change
between every observation (while satisfying Assumption~1). (2) It
allows constructing tests for a wider class of alternative hypotheses
(e.g., smoothly changing systems), while at the same time justifying
any test based on environments \citep[e.g.,][Method I and Method
II]{peters2016}. (3) In particular, it allows for a more
straightforward justification of procedures for pooling environments,
see \citet[discussion by Richardson and Robins on page
1003]{peters2016}.

\subsection{Relation to causality and discussion of assumptions}\label{sec:causalinterp}

An insightful interpretation of the set $S^*$ is given in the context of
causality. Under certain assumptions, the set $S^*$ corresponds to the
set of direct causes (or parents) of the target variable $Y$. This is
best understood in the framework of structural causal models (SCMs)
\citep[e.g.,][]{Bollen1989,Pearl2009}.
\begin{definition}[structural causal models]
  \label{def:SCM}
  A vector of variables $(X^0,X^1,\dots,X^d)\in\R^{d+1}$ ($X^0$ plays
  later the role of the target variable $Y$) is said to satisfy a
  structural causal model (SCM) if for all $j\in\{0,1,\dots,d\}$ there
  exist functions $f^j:\R^{\abs{\pa(j)}+1}\rightarrow\R$ and
  jointly independent noise variables $\epsilon^j$ satisfying
  \begin{equation*}
    X^j\leftarrow f^j(X^{\pa(j)},\epsilon^j),
  \end{equation*}
  where the sets $\pa(j)$ denote the parents of the variable $X^j$ in
  the directed (possibly cyclic) graph corresponding to the structure
  of the SCM.
\end{definition}
Recall that our procedure can be applied to any data generating
process $(\vY, \vX)$ which satisfies the structural invariance
(Assumption~\ref{assumption:invariance}). In the following example, we
define a class of causal models, which satisfies this
invariance and use it as an illustration of the types of
assumptions necessary to fit to our framework.

\begin{example}[SCM with linear Gaussian target]
  Assume that the data $(\vY,\vX)$ is generated in the following way.
  For all $t\in\{1,\dots,n\}$ the variables
  $(Y_t,X_t)=(Y_t=X^0_t,X^1_t,\dots,X^d_t)$ are generated by
  potentially different SCMs such that the structural assignment of
  $Y$ is linear Gaussian, does not depend on $Y$, i.e., there is no
  direct feedback loop, and is fixed across all time points. That is,
  there exists $\sigma\in(0,\infty)$ and $\beta \in \R^{|\pa(0)|}$
  such that for all $t\in\{1,\dots,n\}$ it holds that
  $Y_t\leftarrow X^{\pa(0)}_t\beta+\epsilon^0_t$ with
  $\epsilon^0_t\sim\mathcal{N}(0,\sigma^2)$, and
  $0\notin\operatorname{AN}(0)$, where $\operatorname{AN}(0)$ denotes
  the ancestors and $\pa(0)$ the parents of $Y$. Furthermore, assume that for all
  $t\in\{1,\dots,n\}$ it holds that
  \begin{equation}
    \label{eq:independence_assumption}
    \epsilon^0_t\independent\left\{\epsilon^j_s\mid\forall
    j\in\operatorname{AN}(0), s\in\{1,\dots,n\}\right\}.
  \end{equation}
  Then, similar to \citet[Proposition~1]{peters2016},
  Assumption~\ref{assumption:invariance} is satisfied for the parents
  of $Y$, namely $S^*=\pa(0)$. In particular, this motivates the term
  plausible causal predictors used in \eqref{eq:Shat}.
\end{example}

This causal model allows for any type of intervention on
the predictor variables $(X^1,\dots,X^d)$ at any time. In particular,
this means we do not need to worry about what or when changes occur on
the predictors. In contrast, the restriction that the structural
assignment as well as the noise of $Y$ are not allowed to change
across time implies that no direct interventions on $Y$ are
permitted. This is a reasonable assumption, for example, if $Y$ is a phenotype and the predictors are
gene activities, measured in a time-course experiment, or if $Y$ is a
macro economic indicator and the set of predictors contain all
possible variables which could be used to intervene on $Y$ (see our
monetary policy example in Section~\ref{sec:expreal}). 

A feature of the invariant causal prediction procedure is that we
expect it to be conservative with respect to violations of its
assumptions. For example, if there was a direct intervention on $Y$,
it is expected that all sets are rejected, which results in the empty
set. Conversely, the procedure also remains conservative in the
absence of interventions on the predictors, as the empty set remains
invariant in that case. The resulting output would therefore be
correct (although uninformative), as expected by the coverage property in
Proposition~\ref{thm:empirical_coverage}.

In causal network discovery, the goal is often to infer the entire
causal graph. Our framework is different in this respect, as it aims
to only infer the parents of a single target variable. This comes with
the advantage that we only need to locally infer the dependence from
$Y$ (one node in the graph) on its causal predictors. It further
allows us to use heterogeneous data without worrying about the types
of interventions it may contain, except that they are assumed to not
directly affect the target variable $Y$. The latter is the reason why
we cannot simply apply the methodology to the full network: we need
interventions to obtain informative answers, but these interventions
are not allowed to act directly on the target variable $Y$. Thus, we
cannot use each variable in the graph as a target.  However, when
having the additional information on the time of the interventions and
which variables they directly affect, our method can be iteratively
applied to each variable (i.e., node) in the graph separately by
removing all observations belonging to interventions on that specific
variable; hence allowing us to recover the entire causal structure.

\paragraph{Hidden variables.}
The invariant causal prediction framework used here is also robust to
the presence of hidden variables. For example, any hidden variables
that are not direct parents of the target $Y$ are permitted. More
generally, it can be shown that for settings with arbitrary
constellations of hidden variables (allowing for hidden confounding
between $Y$ and the predictors) and given suitable assumptions
(including a faithfulness assumption on the underlying causal graph)
the plausible causal set estimator still satisfies the following
(slightly weaker) coverage property
\begin{equation*}
  \tilde{S} \subseteq \operatorname{AN}(Y),
\end{equation*}
where $\operatorname{AN}(Y)$ are the ancestors of $Y$. The precise
result, is taken from Proposition~5 in \citet{peters2016}.
Selected ancestor variables can be interpreted as a true rather than a
false positive. Thus, this result establishes a useful robustness
property: the price to be paid for allowing hidden variables is a loss
of detection power.

\section{Tests for $\HOS$ based on scaled residuals}\label{sec:scaled_residuals}

In this section, we construct a general class of tests for $\HOS$,
based on an exact resampling procedure. These tests rely on the linear
Gaussian model that is assumed to exist for the invariant set $S^*$
given in Assumption~\ref{assumption:invariance}.

\subsection{Scaled residual tests}\label{sec:scaledresidualtest}

Consider a fixed invariant set $S\subseteq\{1,\dots,d\}$. As a first
step, observe that by reformulating Definition~\ref{def:invariantset}
in matrix notation it holds that
\begin{equation}
  \label{eq:HOS_vec}
  \HOS:
  \begin{cases}
    \quad\exists \beta\in(\R\setminus\{0\})^{\abs{S}},\,\sigma\in(0,\infty):\\
    \quad\vY=\vX^S\beta+\vepsilon,\text{ with }
    \vepsilon\independent\vX^S\text{ and }
    \vepsilon\sim\mathcal{N}(\mathbf{0},\sigma^2\vI),
  \end{cases}
\end{equation}
where $\vepsilon\coloneqq(\epsilon_1,\dots,\epsilon_n)$. Whenever the
set $S$ is not invariant, the dependence
of $\vY$ on $\vX^{S}$ is not given by the same linear function across
all time points. We can therefore construct a test for $\HOS$ by
performing a goodness of fit test of the linear Gaussian model
\begin{equation}
  \label{eq:linGaussmodel}
  \vY=\vX^S\beta+\boldsymbol{\epsilon},\text{ with }
  \boldsymbol{\epsilon}\sim\mathcal{N}(\mathbf{0},\sigma^2\vI).
\end{equation}
This motivates a two-step procedure; (1) use linear regression to fit
a linear Gaussian model, (2) test whether the residuals are i.i.d.
Gaussian distributed. \citet{shah2015} give a general methodology for
dealing with such tests. We adapt their method to our setting and
notation and consider several specific choices of tests which apply to
our problem. Define the projection matrix
$\vP^S_{\vX}\coloneqq
\vX^S\left((\vX^S)^{\top}\vX^S\right)^{-1}(\vX^S)^{\top}$.  Then the
residuals resulting from an OLS-fit of the model
\eqref{eq:linGaussmodel} are given by
  $\resid^S\coloneqq(\vI-\vP^S_{\vX})\vY$. 
Furthermore, assuming model \eqref{eq:linGaussmodel} is true, i.e.,
$\HOS$ holds, the scaled residuals
$\sresid^S\coloneqq\resid^S/\norm{\resid^S}_2$ can be expressed as
\begin{equation}
  \label{eq:scaled_resid}
  \sresid^S
  \coloneqq\frac{(\vI-\vP^S_{\vX})\vY}{\norm{(\vI-\vP^S_{\vX})\vY}_2}
  =\frac{(\vI-\vP^S_{\vX})\noise}{\norm{(\vI-\vP^S_{\vX})\noise}_2}
  =\frac{(\vI-\vP^S_{\vX})\snoise}{\norm{(\vI-\vP^S_{\vX})\snoise}_2},
\end{equation}
where $\snoise\coloneqq\noise/\norm{\noise}_2$ is the scaled noise.
Given that model \eqref{eq:linGaussmodel} is true, one can thus sample
from the distribution of $\sresid^S\mid\vX=\vx$ by a resampling
procedure using that $\snoise\sim\mathcal{N}(\mathbf{0},\vI)$. More
formally, assume we are given a measurable function
$T:\R^{n}\rightarrow\R$ and a let $B\in\N$ be the number of
simulations. Then, we can use a resampling approach to construct a
sequence of cut-off functions $c_{T,B}:\R^{n\times d}\rightarrow\R$
such that the sequence of hypothesis tests
$(\varphi^S_{T,B})_{B\in\N}$ defined for all $B\in\N$ by
\begin{equation}
  \label{eq:hyptest}
  \varphi^S_{T,B}(\vY,\vX)\coloneqq\mathds{1}_{\{\abs{T(\sresid^S)}>c_{T,B}(\vX)\}},
\end{equation}
achieves correct asymptotic level as $B$ goes to infinity. To see this, fix a
significance level $\alpha\in(0,1)$, let
$\snoise_1,\snoise_2,\dots\iid\mathcal{N}(\mathbf{0},\vI)$ and for all
$\vx\in\R^{n\times d}$ and for all $b\in\N$ define the random
variables
$\sresid_b^{S,\vx}\coloneqq {(\vI-\vP^S_{\vx})\snoise_b}/{\norm{(\vI-\vP^S_{\vx})\snoise_b}_2}$, 
which are i.i.d. copies of $\sresid^S\mid\vX=\vx$. Moreover, for all
$\vx\in\R^{n\times d}$ define
\begin{equation}
  \label{eq:cutoff_fun}
  c_{T,B}(\vx)\coloneqq\lceil B(1-\alpha)\rceil\text{-largest value
    of } \abs{T(\sresid_1^{S,\vx})},\dots,\abs{T(\sresid_B^{S,\vx})}.
\end{equation}
Based on the convergence of the empirical quantiles to the population
quantiles, we get that the test in
\eqref{eq:hyptest} has the following level guarantee.
\begin{proposition}[level of the scaled residual test]
  \label{thm:srtest_level}
  For all measurable $T:\R^n\rightarrow\R$, based on the
  scaled residuals $\sresid^S$, and for all $(\vY,\vX)$ with
  $\P^{(\vY,\vX)}\in\HOS$, the hypothesis 
  test
  $\varphi^S_{T,B}$ defined in \eqref{eq:hyptest} achieves %
  level $\alpha$ as
  $B \rightarrow \infty$, i.e.,
  \begin{equation*}
    \lim_{B\rightarrow\infty}\P\left(\varphi^S_{T,B}(\vY,\vX)=1\right)=\alpha.
  \end{equation*}
\end{proposition}
More details on the proof can be found in
Appendix~\ref{sec:tslevelproof}, where we prove a more general result
including time lags. For any measurable test statistic $T$ we have,
therefore, constructed a hypothesis test which achieves correct
asymptotic level for testing the null hypothesis $\HOS$. It is clear
that the power properties of such a test depend on the alternative and
the form of the function $T$. In the next section, we give specific
choices of $T$ that allow us to detect alternatives for which $S$ is
not an invariant set.

\subsection{Choosing test statistics}\label{sec:diff_teststat}

Recall that the invariance of a set $S$ (see
Definition~\ref{def:invariantset}) corresponds to a time invariance of
the conditional distribution, i.e.
\begin{equation}
  \label{eq:conditionaldist_time}
  \forall t,s\in\{1,\dots,n\}:\quad\P^{Y_t|X_t^S}=\P^{Y_s|X_s^S}.
\end{equation}
Violations to this invariance include conditional distributions that
change at some, many or even at every time point, see
Figure~\ref{fig:alternatives} for three examples.
\begin{figure}[h!]
  \spacingset{1}
  \centering
  \input{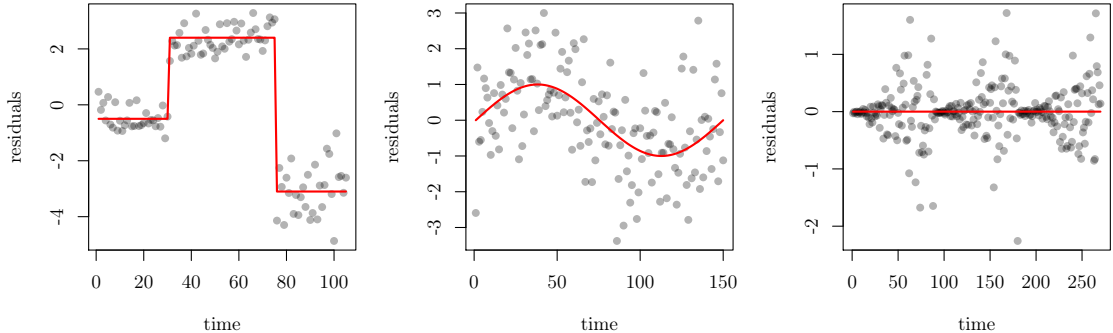}
  \caption{Plots illustrating the residuals of a pooled regression vs
    time for three different alternatives. In all cases $\HOS$ is
    violated such that we obtain residuals that are not i.i.d.\ over
    time.  This is due to block-wise shifts in mean and variance (left
    plot), a gradually shifting mean (center plot), varying higher
    moments (right plot).  The red lines represent the means as
    functions of time.  }
  \label{fig:alternatives}
  \spacingset{\spacevar}
\end{figure}
These can, e.g., be generated by noise interventions in SCMs, i.e.,
interventions that change the mean or variance of the noise for the
predictor variables. However, not all types of interventions are
necessarily captured as a time dependence in the residual distribution
alone.

\begin{example}[non-detectable structure changes in residual
  distribution vs time]
  \label{ex:non_detectable_structurechanges}
  Assume we are given data from the following generative model
  \begin{equation*}
    Y_t=\beta_tX_t+\epsilon_t,\quad t\in\{1,\dots,200\}
  \end{equation*}
  with $X_t,\epsilon_t\iid\mathcal{N}(0,1)$, $\beta_t=1$ for
  $t\in\{1,\dots,100\}$ and $\beta_t=-1$ for $t\in\{101,\dots,200\}$.
  Then, the regression parameter resulting from a pooled ordinary
  least squares regression is
  given by
 $   \hat{\beta}_{\operatorname{OLS}}\approx 0$. 
  In particular, this implies that it is impossible to detect the structure
  change in a residuals versus time plot. Instead, one can group the
  data into the environments $e_1=\{1,\dots,100\}$ and
  $e_2=\{101,\dots,200\}$ and then consider the residuals versus the
  predictors on each environment individually. The result is that on
  $e_1$ the residuals are increasing with slope one and on $e_2$ the
  residuals are decreasing with slope one, which clearly contradicts
  the invariance assumption. This shows, that some violations are only
  detectable in the pooled residuals if we additionally use
  information contained in the ordering of the predictors rather than
  only using the time ordering.
\end{example}
The example illustrates that certain types of interventions (in
particular if they change the structure) can lead to alternatives that
are hard (or even impossible) to detect by looking only for changes in
the distribution of the residuals from a pooled regression across
time. This implies certain types of violations are only detected by
also considering information from predictors, for example, by checking
that the regression coefficients remain constant. In the following two
subsections we consider specific types of alternative hypotheses and
discuss which types of test statistics can be used to detect them. The
choice of the test statistic only affects the power of our method,
meaning that any of the following test statistics will result in tests
which control the type I error as described in
Section~\ref{sec:scaledresidualtest}.

\subsubsection{Change point alternatives}\label{sec:test_stat_cp}

Throughout this subsection, we want to construct test statistics that
focus power on detecting deviations from invariance where the
interventions occur in a block-wise manner, i.e., non-invariances
occur at specific change points. As described in the methodology in
the previous sections, we are not interested in estimating the change
points from data.
To be
more precise, we now introduce some notation related to change point
models. Consider tuples of the form
$\CPt=(t_1,\dots,t_{L})\in\{1,\dots,n-1\}^{L}$ satisfying $t_i<t_j$
for all $i<j$. For every such tuple $\CPt$ define for all
$i\in\{1,\dots,L+1\}$ the following block-wise environments
\begin{equation*}
  e_i(\CPt)\coloneqq
  \begin{cases}
    \{1,\dots,t_1\} \quad &\text{if }i=1,\\
    \{t_{i-1}+1,\dots,t_i\} \quad &\text{if }1<i\leq L,\\
    \{t_{L}+1,\dots,n\} \quad &\text{if }i=L+1.\\
  \end{cases}
\end{equation*}
Moreover, denote by
$\mathcal{E}(\CPt)\coloneqq\{e_1(\CPt),\dots,e_{L+1}(\CPt)\}$ the
collection of the $L+1$ environments. We will drop the tuple $\CPt$ in
the notation whenever it is clear from the context. We consider models
described by a fixed set of change points at which changes in the
experimental conditions can occur. The underlying change point model
can then be specified by the existence of a fixed (unknown) tuple of change points
$\CPt^*=(t_1^*,\dots,t_{L^*}^*)$ such that for all environments
$e\in\mathcal{E}(\CPt^*)$ it holds that
\begin{equation*}
  (Y_t,X_t)_{t\in e}\iid F_e,
\end{equation*}
where $F_e$ are fixed distributions depending only on the
environments.

Given the true collection of change points $\CPt^*$, this reduces to
the original ideas of invariant causal prediction introduced by
\citet{peters2016}. Here, we are interested in the case when the
change points are unknown and we no longer have the correct
environments.
A naive approach would be to use an existing change point detection
method and plug-in the estimated segments into the invariant causal
prediction (ICP) method from \citet{peters2016}. As discussed in
Remark~\ref{rem:clustering}, however, a change point detection method
is only allowed to be used on data from the response variable $Y$,
since otherwise the coverage property of the procedure could be
destroyed. This implies a major restriction: an example in
Remark~\ref{rem:changepoint} shows that changes in the covariates $X$
might be non-detectable in the response variable $Y$ and thus, any
change point method applied on the response variable $Y$ might brake
down. Here, we instead propose a procedure which exploits changing
structures among all the variables and directly optimizes the power to
detect (non-)invariances. This is done by simultaneously testing for
(non-)invariances over all potential environments based on a grid of
potential change points, and encoding this multiple testing problem in
the test statistic. Our resampling approach
(see Section~\ref{sec:scaled_residuals}) ensures that the generally strong
dependencies between these tests are taken into account and one only
pays a small price for the somewhat higher degree of multiple testing
adjustment.

Our goal is to construct test statistics $T=T(\sresid^S)$, based on
scaled residuals from a regression on the covariates $S$ which
are capable of capturing potential violations in model
\eqref{eq:linGaussmodel} that can occur due to the underlying change
points $\CPt^*$. Essentially, this means that $\abs{T(\sresid^S)}$
should be small whenever the model \eqref{eq:linGaussmodel} is true
and large whenever it is false. Violations in the invariance occur
due to differences in the structural form of model
\eqref{eq:linGaussmodel} between two different environments
$e,f\in\mathcal{E}(\CPt^*)$. Therefore, the idea is to take a
collection of environments
$\mathcal{E}\subseteq\mathcal{P}(\{1,\dots,n\})$ that makes use of
the block-wise structure of the data and then combine all pairwise
comparisons between these environments. To be more precise, for all
$e,f\in\mathcal{E}$ we construct several test statistics
$T_{e,f}^i$ which detect differences between the environments
$e$ and $f$. We combine them to single test statistics
either by
\begin{equation}
  \label{eq:teststat_multi1}
  T^{\max,\mathcal{F}}_i(\sresid^S)\coloneqq\max_{(e,f)\in\mathcal{F}}\abs[\big]{T_{e,f}^i(\sresid^S)},
\qquad \text{ 
or by } \qquad
  T^{\operatorname{sum},\mathcal{F}}_i(\sresid^S)\coloneqq\sum_{(e,f)\in\mathcal{F}}\abs[\big]{T_{e,f}^i(\sresid^S)},
\end{equation}
where $\mathcal{F}\subseteq\mathcal{E}\times\mathcal{E}$. Details on
how to construct the collections of environments
$\mathcal{E}\subseteq\mathcal{P}(\{1,\dots,n\})$ and the corresponding
collection $\mathcal{F}$ are given in Section~\ref{sec:choiceofE}. For
the theory part we consider only
\begin{equation}
  \label{eq:F1}
  \mathcal{F}^{1}\coloneqq\{(e,f)\in\mathcal{E}\times\mathcal{E}\mid
e\cap f=\varnothing\}.
\end{equation}
The test statistics $T_{e,f}$ should be capable of detecting
differences between two environments, in the following we consider two
options: (1) Test statistics that perform a regression step in order
to incorporate information from the predictors, which are then capable
of (at least in the large sample limit) detecting any violation of the
invariance. (2) Test statistics that only check for changes in the
pooled residual distribution, which have the advantage of being
computationally faster but are not capable of detecting all violations
(see Example~\ref{ex:non_detectable_structurechanges}).

\paragraph{Detecting block-wise shifts in the regression of the scaled
  residuals on predictors}\mbox{}\\
In the following we construct test statistics which are capable of
detecting the following two types of violations of model
\eqref{eq:linGaussmodel} that can arise from an underlying change
point model:
\begin{compactitem}
\item[(i)] difference in the regression coefficients: $\beta_{e,S}\neq\beta_{f,S}$
\item[(ii)] difference in the noise variance: $\sigma_{e,S}^2\neq\sigma_{f,S}^2$
\end{compactitem}
where $\beta_{e,S}$, $\beta_{f,S}$, $\sigma_{e,S}$ and $\sigma_{f,S}$
are population least squares regression coefficients and residual
variances on the environments $e,f\in\mathcal{E}(\CPt^*)$ when
regressing $\vY$ on $\vX^S$ restricted to environment $e$ and
$f$ respectively. Both of these violations can be detected by regressing the scaled
residuals $\sresid^S$ on $\vX^S$ for each of the two environments $e$
and $f$. To this end, define for all possible environments
$h\subseteq\{1,\dots,n\}$ the regression coefficient and biased sample
variance of the scaled residuals regressed on $\vX^S_h$ by
\begin{equation*}
  \hgamma_{h,S}\coloneqq((\vX_h^S)^{\top}\vX^S_h)^{-1}(\vX_h^S)^{\top}\sresid^S_h
  \quad\text{ and }\quad
  \hat{s}^2_{h,S}\coloneqq\frac{(\sresid^S_h-\vX^S_h\hgamma_{h,S})^{\top}(\sresid^S_h-\vX^S_h\hgamma_{h,S})}{\abs{h}},
\end{equation*}
respectively, where $\sresid^S_h$ is the restriction of $\sresid^S$ to
environment $h$. The idea is that both of the above violations (i) and
(ii) lead to differences between the regressions of at least two
environments $e,f\in\mathcal{E}(\CPt^*)$. It is possible to test
for either of the two violations individually using the test
statistics
\begin{align}
  T_{e,f}^1(\sresid^S)&\coloneqq\norm{\hgamma_{e,S}-\hgamma_{f,S}}_2,   \label{eq:teststat_beta}
\\
  T_{e,f}^2(\sresid^S)&\coloneqq\frac{\hat{s}^2_{e,S}}{\hat{s}^2_{f,S}}-1,   \label{eq:teststat_sigma}
\end{align}
for differences in the regression coefficients and for differences in the variance of the noise, respectively. The two resulting
hypothesis tests can then be combined with a Bonferroni correction,
which we refer to as decoupled test throughout this paper. A further
option is to test for both potential violations simultaneously by
using a test statistic similar to the Chow test \citep{chow1960} given
by
\begin{equation}
  \label{eq:teststat_single}
  T_{e,f}^3(\sresid^S)\coloneqq\frac{(\sresid^S_{e}-\vX^S_{e}\hgamma_{f,S})^{\top}(\sresid^S_{e}-\vX^S_{e}\hgamma_{f,S})}{\hat{s}^2_{f,S}\abs{e}}-1.
\end{equation}
For the remainder of this paper we denote the test based on
$T_{e,f}^3$ as the combined test. Unlike the Chow test we do not
normalize the denominator, which means that $T_{e,f}^3$ in
particular does not follow an F-distribution. Since we use an exact
resampling approach this will however also not be necessary here.

\paragraph{Detecting block-wise shifts in the scaled residuals}\mbox{}\\
As illustrated by Example~\ref{ex:non_detectable_structurechanges} we
are not capable of detecting all types of violations of the invariance
by checking for shifts in the distribution of the pooled residuals
across time. Nevertheless, many violations are in fact detectable in
this fashion. An example in which the underlying model has two change
points, leading to a block-wise time dependence of the (scaled)
residuals, is illustrated in the left plot of
Figure~\ref{fig:alternatives}. Such block-wise shifts in mean and
variance between two (true) environments
$e,f\in\mathcal{E}(\CPt^*)$ can for example be detected using the
following two test statistics
\begin{align}
  T_{e,f}^4(\sresid^S)&\coloneqq\frac{1}{\abs{e}}\sum_{i\in e}\sresid^S_i-\frac{1}{\abs{f}}\sum_{i\in f}\sresid^S_i
\qquad \text{ and}
  \label{eq:teststat_blockmeantime}\\
  T_{e,f}^5(\sresid^S)&\coloneqq\frac{(\sresid^S_{e})^{\top}\sresid^S_{e}}{(\sresid^S_{f})^{\top}\sresid^S_{f}}-1,   \label{eq:teststat_blockvartime}
\end{align}
where \eqref{eq:teststat_blockmeantime} detects shifts in mean and
\eqref{eq:teststat_blockvartime} detects shifts in variances. The main
advantage of these two test statistics is that they do not require an
extra regression step of the residuals on the predictors and are thus
computationally faster.

\subsubsection{Further alternatives}\label{sec:gradual_shifts}

The test statistics constructed in Section~\ref{sec:test_stat_cp} are
tuned to detect alternatives arising from an underlying change point
model. Depending on the setting, more natural alternatives might be
gradual mean shifts (see center plot in Figure~\ref{fig:alternatives})
or even more complicated shifts in the higher moments (see right plot
of Figure~\ref{fig:alternatives}). In the following, we give two
potential choices of test statistics which focus power on these two
latter alternatives. As discussed in
Section~\ref{sec:scaledresidualtest} the level properties hold , even
for finite samples, for any test statistic, allowing us to
choose arbitrary test statistics and plug them into our methods
described above.

\paragraph{Detecting gradual shifts in the scaled residuals}\mbox{}\\
Assume that we want to detect gradual mean shifts across time as
illustrated in the center plot in Figure~\ref{fig:alternatives}. A
natural idea is to use a non-linear (smooth) regression procedure to
regress the scaled residuals $\sresid$ given in
\eqref{eq:scaled_resid} on $time$. This results in an estimator of the
mean function $\mu_t=\E(\sresid_t)$, which satisfies $\mu_t\equiv 0$
under the null hypothesis $\HOS$ and captures the gradual shifts in
the alternative.  Essentially, the idea is to use a smoothing
procedure that best approximates the expected gradual shifts in the
alternative. For example, for very smooth shifts one could use
generalized additive models (GAM) \citep{gam}, implemented in the
\texttt{R}-package \texttt{mgcv}, to get the non-linear smoothing fit
and then consider a measure of how far the smoother deviates from the
horizontal line at $0$. Possible measures include the area under the
smoother or the p-value corresponding to the hypothesis test which
tests whether all coefficients are simultaneously zero. Along the same
lines one can also detect shifts in second moment by smoothing the
squared scaled residuals $\sresid^2$ across time.

\paragraph{Detecting more complicated shifts in the scaled residuals}\mbox{}\\
In case the alternatives we are interested in include nonlinear
changes of higher moments or other more complicated variations across
time, e.g., right plot in Figure~\ref{fig:alternatives}, one option is
to use the test statistic of a non-parametric independence test. For
example, we could use the Hilbert-Schmidt independence criterion
(HSIC) introduced by \citet{gretton2007} and consider the test statistic
\begin{equation*}
  T^{\operatorname{HSIC}}(\sresid^S)\coloneqq \widehat{\operatorname{HSIC}}(\sresid^S,\operatorname{time}),
\end{equation*}
where $\widehat{\operatorname{HSIC}}$ is the empirical version of
HSIC. The use of HSIC is motivated by the property that it allows to
construct independence tests which are capable of capturing any
type of dependence between random variables. An implementation of the
Hilbert-Schmidt independence criterion is given in the
\texttt{R}-package \texttt{dHSIC} \citep{pfister2016}.

\section{Detection rates} \label{sec:consistency}

While the assumption that a set $S$ is invariant in the sense of
Definition~\ref{def:invariantset} is sufficient for the scaled
residual test to achieve correct level for arbitrary test statistics (see
Proposition~\ref{thm:srtest_level}), we require additional constraints
on the underlying model in order to phrase and prove
results about the power. Additionally, any type of power analysis will
rely on the form of the test statistic. In this section, we restrict
ourselves to showing that the tests based on the statistics
\eqref{eq:teststat_beta}, \eqref{eq:teststat_sigma} and
\eqref{eq:teststat_single} are able to detect a large class of
alternatives resulting from an underlying change point model. In
particular, we show that they are consistent in the sense that they
have asymptotic power equal to one in the large sample limit, with
additional results on the rate of convergence.

\subsection{Asymptotic change point model}

Since we are interested in analyzing the large sample behavior of our
methods we need to formalize what a growing sample size means in our
change point model. We restrict ourselves to the case of a fixed
number of change points where additional data points are added in such
a way that the relative positions of the change points are
conserved. To this end, assume we are given data from a triangular
array $((Y_{n,t},X_{n,t})_{t\in\{1,\dots,n\}})_{n\in\N}$, which
satisfies the following assumption.
\begin{assumption}[asymptotic change point model]
  \label{assumption:asymptotic_cpmodel}
  There exists a fixed (unknown) collection of relative change points
  $\alpha^*_1,\dots,\alpha^*_L\in(0,1)$ satisfying for
  $i\in\{1,\dots,L\}$ that
  $\lim_{n\rightarrow\infty}t^*_{n,i}/n=\alpha^*_i$,
  where $\CPt^*_n\coloneqq(t^*_{n,1},\dots,t^*_{n,L})$ is the true set
  of change points for $n$ data points. Moreover, for all
  $i\in\{1,\dots,L+1\}$ it holds that
  \begin{equation*}
    (Y_{n,t},X_{n,t})_{t\in e_{i}(\CPt^*_n)}\iid F_{n,i},
  \end{equation*}
  for some fixed distributions $F_{n,i}$.
\end{assumption}
This, in particular, implies that each environment grows linearly as
the sample size increases, i.e., for all $i\in\{1,\dots,L+1\}$ it
holds that $\abs{e_{i}(\CPt^*_n)}=\landauO(n)$ as
$n\rightarrow\infty$. We assume a finite number of asymptotic change
points, and for any finite sample size, the position of these change
points is unconstrained. Moreover, our results can be extended to
settings where the number of change points increases with $n$, as long
as the size of the individual environments grows
polynomially. Finally, we require one further assumption.
\begin{assumption}[Multivariate normality]
  \label{assumption:normal}
  For all $n\in\N$ and for all $e\in\mathcal{E}(\CPt^*_n)$ the random
  variable $(Y_t,X_t)_{t\in e}$ has a multivariate normal distribution.
\end{assumption}
This assumption together with the i.i.d. assumption for
$(Y_t,X_t)_{t\in e}$ for any environment $e\in\mathcal{E}(\CPt^*_n)$
ensures that for any fixed set $S\subseteq \{1,\dots,d\}$ and for
every $e\in\mathcal{E}(\CPt^*_n)$ there exist unique parameters
$\beta_{e,S}$, $\mu_{e,S}$ and $\sigma_{e,S}$ such that for all
$t\in e$ it holds that
\begin{equation}
  \label{eq:lineardecomp_env}
  Y_t=\mu_{e,S}+X^S_t\beta_{e,S}+\epsilon_t, \text{ with
  }\epsilon_t\sim\mathcal{N}(0,\sigma_{e,S}^2)\text{ and }X^S_t\independent\epsilon_t.
\end{equation}
The important part is the independence between $X^S$ and the noise
$\epsilon$, which is no longer true if Assumption~\ref{assumption:normal} is
dropped.

\subsection{Asymptotic results}

Throughout this section, we assume that $(\vY_{n},\vX_{n})_{n\in\N}$
satisfies Assumption~\ref{assumption:asymptotic_cpmodel} and
Assumption~\ref{assumption:normal}. We show that for an
appropriate choice of environments it is possible to prove consistency
of our test, against the following alternatives,
\begin{align*}
  \HAS^n(a,b)\coloneqq\left\{P\,\right\rvert\,
  &(\vY_n,\vX_n)\sim P\text{ such that }\exists\, i,j\in\{1,\dots,L+1\}:\\
  &\left.a=\abs[\big]{\sigma^2_{e_i(\CPt^*_n),S}-\sigma^2_{e_j(\CPt^*_n),S}}>0
    \text{ and }
    b=\norm[\big]{\beta_{e_i(\CPt^*_n),S}-\beta_{e_j(\CPt^*_n),S}}_{2}>0\right\}.
\end{align*}
For all $n\in\N$, let
$\mathcal{E}_n\subseteq\mathcal{P}(\{1,\dots,n\})$ be a collection of
pairwise disjoint non-empty environments. In order to obtain a
consistency result we are interested in sequence of such collections
$(\mathcal{E}_n)_{n\in\N}$ satisfying the following 3 conditions,
\begin{enumerate}[widest=0]
\item[(C1)\phantom{,k}] there exists a sequence $(r_n)_{n\in\N}$ such
  that
  $$r_n=\min_{e_n\in\mathcal{E}_n}\abs{e_n}\quad\text{and}\quad
  \lim_{n\rightarrow\infty}r_n=\infty,$$\label{item:c1}
\item[(C2)\phantom{,k}] for all $i\in\{1,\dots,L+1\}$ there exists a sequence
  $(f_n)_{n\in\N}$ with $f_n\in\mathcal{E}_n$ and a constant $N\in\N$ such that for all
  $n\geq N$ it holds that $f_n\subseteq
  e_i(\CPt^*_n)$ and such that the sequences
  $(\sigma_{f_n}^2)_{n\in\N}$ and $(\beta_{f_n})_{n\in\N}$ are
  convergent and $\lim_{n\rightarrow\infty}\sigma_{f_n}^2>0$. \label{item:c2}
\item[(C3,k)] for all $e_n\in\mathcal{E}_n$ the matrix
  $1/\abs{e_n}\cdot\vX_{e_n}^{\top}\vX_{e_n}$ is
  $\P$-a.s. invertible and there exist $c,C\in\R$ and $k\in\N$ such that
  for all $n\in\N$ it holds that
  $$0<c\leq\E\left(\lambda_{\min}\left(\tfrac{1}{e_n}\vX_{e_n}^{\top}\vX_{e_n}\right)^{2k}\right)\leq\E\left(\lambda_{\max}\left(\tfrac{1}{e_n}\vX_{e_n}^{\top}\vX_{e_n}\right)^{2k}\right)\leq
  C<\infty$$ \label{item:c3}
\end{enumerate}
Conditions \hyperref[item:c1]{(C1)} and \hyperref[item:c2]{(C2)} are
in particular satisfied for $(\mathcal{E}^{G_n})_{n\in\N}$ where the
environments are constructed using a grid as defined
in \eqref{eq:unknownCPenvs} and given that the sequence of grids
$(G_n)_{n\in\N}$ becomes finer sufficiently fast as $n$
grows. Moreover, the moment condition \hyperref[item:c3]{(C3,k)} is satisfied
for any sequence of collections $(\mathcal{E}_n)_{n\in\N}$ and any
$k\in\N$ due to Assumption~\ref{assumption:normal}.

Based on these conditions we can prove consistency rates for the
tests based on fixed sets $S$, which results in a consistency of
the estimation of the set $\tilde{S}$.

\subsubsection{Rate consistency of tests for fixed sets $S$}

Consider a fixed non-invariant set $S$, then the following theorems
show that we are capable of detecting the non-invariance with a rate
depending on the type of test we use. We begin with the result for the
decoupled test. Recall, that the decoupled test
$\varphi_{\operatorname{decoupled},B}$ combines the test statistics in
\eqref{eq:teststat_beta} and \eqref{eq:teststat_sigma} and adjusts the
level with a Bonferroni correction, i.e.,
$\varphi_{\operatorname{decoupled},B}$ rejects the null hypothesis at
level $\alpha$ if and only if at least one of the tests
$\varphi_{T_1^{\max,\mathcal{F}^1(\mathcal{E}_n)},B}$ or
$\varphi_{T_2^{\max,\mathcal{F}^1(\mathcal{E}_n)},B}$ reject the null
hypothesis at level $\alpha/2$. The following theorem shows that it is
consistent.
\begin{theorem}[rate consistency of decoupled test]
  \label{thm:consistent_test_bonferroni}
  Assume Assumption~\ref{assumption:asymptotic_cpmodel}
  and~\ref{assumption:normal}, let $S\subseteq\{1,\dots,d\}$ and let
  $(\mathcal{E}_n)_{n\in\N}$ be a sequence of collections of pairwise
  disjoint non-empty environments with the properties
  \hyperref[item:c1]{(C1)}, \hyperref[item:c2]{(C2)} and
  \hyperref[item:c3]{(C3,k)} and assume that for all $n\in\N$ it holds
  that $(\vY_n,\vX_n)\sim P_n\in\HAS^n(a_n,b_n)$, where $a_n$ and $b_n$
  satisfy the following condition
  \begin{equation*}
    \frac{\abs{\mathcal{E}_n}^{\tfrac{1}{k}}}{\sqrt{r_n}}=\landauo(a_n)
    \quad\text{or}\quad
    \frac{\abs{\mathcal{E}_n}^{\tfrac{1}{k}}}{\sqrt{r_n}}=\landauo(b_n).
  \end{equation*}
  Then it holds that
  \begin{equation*}
    \lim_{n\rightarrow\infty}\lim_{B\rightarrow\infty}\P_{P_n}\left(\varphi_{\operatorname{decoupled},B}(\vY_n,\vX_n^S)=1\right)=1.
  \end{equation*}
\end{theorem}
A proof of this result is given in
Appendix~\ref{sec:proof_consistent2}. A different option is to use the
combined test based on the test statistic in
\eqref{eq:teststat_single} which tests for both shifts in regression
coefficients and shifts in variance simultaneously. Surprisingly, this
leads to a worse rate of detecting shifts in the regression
coefficients than for the decoupled test.
\begin{theorem}[rate consistency of combined test]
  \label{thm:consistent_test}
  Assume Assumption~\ref{assumption:asymptotic_cpmodel}
  and~\ref{assumption:normal}, let $S\subseteq\{1,\dots,d\}$ and let
  $(\mathcal{E}_n)_{n\in\N}$ be a sequence of collections of pairwise
  disjoint non-empty environments with the properties
  \hyperref[item:c1]{(C1)}, \hyperref[item:c2]{(C2)} and
  \hyperref[item:c3]{(C3,k)} and assume that for all $n\in\N$ it holds
  that $(\vY_n,\vX_n)\sim P_n\in\HAS^n(a_n,b_n)$, where $a_n$ and $b_n$
  satisfy the following condition
  \begin{equation*}
    \frac{\abs{\mathcal{E}_n}^{\tfrac{1}{k}}}{\sqrt{r_n}}=\landauo(a_n)
    \quad\text{or}\quad
    \frac{\abs{\mathcal{E}_n}^{\tfrac{1}{k}}}{\sqrt{r_n}}=\landauo(b_n^2).
  \end{equation*}
  Then it holds that
  \begin{equation*}
    \lim_{n\rightarrow\infty}\lim_{B\rightarrow\infty}\P_{P_n}\left(\varphi_{T_3^{\max,\mathcal{F}^1(\mathcal{E}_n)},B}(\vY_n,\vX_n^S)=1\right)=1.
  \end{equation*}
\end{theorem}
A proof of this result is given in
Appendix~\ref{sec:proof_consistent}.

\begin{remark}[uniform consistency]\label{rmk:uniform_consistency}
  The results in Theorems~\ref{thm:consistent_test_bonferroni}
  and~\ref{thm:consistent_test} can be extended to be
  uniform across the following alternatives
  \begin{equation*}
    \bar{H}^{n}_{A,S}(\bar{a},\bar{b})\coloneqq  \left\{P \in
      H^n_{A,S}(a,b) \,\rvert\, a\geq \bar{a},\, b\geq \bar{b} \right\},
  \end{equation*}
  i.e., across all alternatives with a fixed minimum signal. Then,
  given the same rates for $\bar{a}_n$ and $\bar{b}_n$ in the
  Theorems~\ref{thm:consistent_test_bonferroni}
  and~\ref{thm:consistent_test} we get the result
  \begin{equation*}
    \lim_{n\rightarrow\infty}\lim_{B\rightarrow\infty}\inf_{P_n\in\bar{H}^{n}_{A,S}(\bar{a}_n,\bar{b}_n)}\P_{P_n}\left(\varphi_{B}(\vY_n,\vX_n^S)=1\right)=1,
  \end{equation*}
  where $\varphi$ is either the combined or the decoupled test. The
  precise statement is given in
  Theorem~\ref{thm:uniform_consistent_test} in
  Appendix~\ref{sec:uniform_consistency}. In order to extend the
  proofs we additionally assume that the condition
  \hyperref[item:c2]{(C2)} is assumed to be uniform across
  $\bar{H}^{n}_{A,S}(\bar{a}_n,\bar{b}_n)$. Further details on this
  extension are given in Appendix~\ref{sec:uniform_consistency}.
\end{remark}

\subsubsection{Rate consistency of estimator $\hat{S}$}

We can also show that the estimator for the plausible causal
predictors $\hat{S}$ given in \eqref{eq:Shat2} converges to the set
$\tilde{S}$ in \eqref{eq:Shat} with the same rates as in the previous
section.

\begin{corollary}[rate consistency of estimator $\hat{S}$ (decoupled
  test)]
  \label{thm:consistent_method}
  Assume Assumption~\ref{assumption:asymptotic_cpmodel}
  and~\ref{assumption:normal}, let $(\mathcal{E}_n)_{n\in\N}$ a
  sequence of collections of pairwise disjoint non-empty environments
  with the properties \hyperref[item:c1]{(C1)},
  \hyperref[item:c2]{(C2)} and
  \hyperref[item:c3]{(C3,k)}. Additionally assume that there exists
  positive sequences $(a_n)_{n\in\N}$ and $(b_n)_{n\in\N}$ satisfying
  for all $n\in\N$ and for all $S\subseteq\{1,\dots,d\}$ with $\HOS$
  false that $(\vY_n,\vX_n)\sim P_n\in\HAS^n(a_n,b_n)$ and
  \begin{equation*}
    \frac{\abs{\mathcal{E}_n}^{\tfrac{1}{k}}}{\sqrt{r_n}}=\landauo(a_n)
    \quad\text{or}\quad
    \frac{\abs{\mathcal{E}_n}^{\tfrac{1}{k}}}{\sqrt{r_n}}=\landauo(b_n).
  \end{equation*}
  Moreover, for all fixed sets $S\subseteq\{1,\dots,d\}$ denote by
  $\varphi^S_{n,B}$ the hypothesis test given by
  $\varphi_{\operatorname{decoupled},B}$ and define the family of
  tests $\varphi_{n,B}=(\varphi^S_{n,B})_{S\subseteq\{1,\dots,d\}}$. Then it
  holds that
  \begin{equation*}
    \lim_{n\rightarrow\infty}\lim_{B\rightarrow\infty}\P_{P_n}\left(\hat{S}(\varphi_{n,B})=\tilde{S}\right)\geq1-\alpha.
  \end{equation*}
\end{corollary}
A proof is given in Appendix~\ref{sec:proof_corollary}. Similar to
Theorem~\ref{thm:consistent_test} we get an equivalent result (with
worse detection rate) for the combined test. As explained in
Section~\ref{sec:causalinterp}, the set $\tilde{S}$ is a subset of the
parents of $Y$ (or in the case of hidden variables a subset of the
ancestors $Y$). Hence, this theorem shows that (under sufficient
interventions on the predictors) we are able to recover the correct
parents (or ancestors) with a known detection rate.

\section{Implementation} \label{sec:implementation}

Our methods are implemented in the \texttt{R}-package \texttt{seqICP}
available on CRAN. The package in particular includes all the test
statistics introduced in Section~\ref{sec:diff_teststat}. In this
section we discuss some additional details on the practical
implementation of our methods. A rough outline of our block-wise
procedures is given in Appendix~\ref{sec:support} in
Algorithm~\ref{alg:seqICP}. In contrast to the block-wise procedure,
our methods based on smoothing or general independence tests (see
Section~\ref{sec:gradual_shifts}) do not require a separation into
blocks of environments.%

\subsection{Choosing environments and comparison set}\label{sec:choiceofE}

There are many reasonable ways in which the set of comparisons
$\mathcal{F}$ can be chosen. The choice affects both empirical power
properties and computational complexity. This in particular leads to a
trade-off between the number of comparisons and the size of the
environments. As shown in Section~\ref{sec:consistency} this trade-off
can be chosen in such a way that our methods become consistent.

We consider two options of choosing the comparison set $\mathcal{F}$ which
work well in practice. The first option, which we already introduced
in \eqref{eq:F1}, is to use the choice from the
theoretical results where we compare all pairs of non-intersecting
environments, i.e.,
\begin{equation*}
  \mathcal{F}^1(\mathcal{E})=\left\{(e,f)\in\mathcal{E}\times\mathcal{E}\,\mid\,
  e\cap f=\varnothing \right\}.
\end{equation*}
A second computationally more efficient option is to not
compare all environments pairwise but to rather compare each
environment against its complement, i.e.,
\begin{equation}
  \label{eq:F2}
  \mathcal{F}^2(\mathcal{E})=\left\{(e,f)\in\mathcal{P}(\{1,\dots,n\})^2\,\mid\,e\in\mathcal{E}\text{
    and }
  f=\{1,\dots,n\}\setminus e \right\}.
\end{equation}

For each type of comparison we need to additionally choose the
collection of environments $\mathcal{E}$. A reasonable option is to
pick a grid $G=(g_1,\dots,g_m)$ on $\{1,\dots,n\}$ (where
$0<g_1<\cdots<g_m<n$) and then use
\begin{equation}
  \label{eq:unknownCPenvs}
  \mathcal{E}^{G}\coloneqq\bigcup_{\overset{k,\ell\in\{0,g_1,\dots,g_m,n\}}{k<\ell}}\left\{k+1,k+2,\dots,\ell\right\}.
\end{equation}
This collection of environments is in particular larger than the one
introduced in Section~\ref{sec:test_stat_cp},
i.e. $\mathcal{E}\left(\CPt\right)\subseteq\mathcal{E}^{\CPt}$, where
$\CPt$ is the set of change points. Given that the set of change
points is unknown one can simply take an equally spaced grid on
$\{1,\dots,n\}$. However, it is also possible to include some prior
knowledge about the approximate locations of change points into the
grid $G$.

In order to achieve the consistency rates from the theory
(Section~\ref{sec:consistency}) we could choose the size of the grid
such that $m$ is of the order $\log(n)$ and the size of each of the
$m+1$ environments tends to infinity as $n$ gets large. In particular,
the comparison sets would then satisfy condition
\hyperref[item:c1]{(C1)}. For example, we could choose an equidistant
grid with $\log(n)$ grid points, then using the notation of
Theorems~\ref{thm:consistent_test_bonferroni}
and~\ref{thm:consistent_test} it holds that
$\abs{\mathcal{E}_n}=\landauO(\log(n)^2)$ and $r_n=c\cdot n/\log(n)$,
for some $c>0$. Hence, given that condition \hyperref[item:c3]{(C3,k)}
is satisfied for all $k\in\N$ (this is the case for Gaussian noise),
we detect shifts in either the variance or the regression coefficients
with a rate of $\landauo((\log(n)/n)^{1/2})$ for the decoupled
test. Whereas for the combined test the detection rate for shifts in
the regression coefficients would only be
$\landauo((\log(n)/n)^{1/4})$ and for shifts in variance it would be
$\landauo((\log(n)/n)^{1/2})$.

\subsection{Computational complexity}

In order to analyze the computational complexity of the procedure it
is helpful to distinguish between the complexity of performing an
invariance test for a single set $S$ and the complete procedure, which iterates
over all such potential invariant sets.

The complexity of a single test based scaled residual resampling
introduced in Section~\ref{sec:diff_teststat} requires one step of
ordinary least squares to compute the residuals and $B$ evaluations of
the test statistic to approximate the null
distribution. The computational cost of one evaluation of the
various test statistics is given in Table~\ref{table:complexity}.
\begin{table}[h!]
  \spacingset{1}
  \centering
  \begin{tabular}{lccc}
    \toprule
    \textbf{test statistic}
    & $T^{*,\mathcal{F}}_i$ with $i\in\{1,2,3\}$
    & $T^{*,\mathcal{F}}_i$ with $i\in\{1,2,3\}$
    & $T^{\operatorname{HSIC}}$ \\
    \midrule
    \textbf{complexity}
    & $\landauO\left(\abs{\mathcal{F}}\cdot\abs{S}^2\cdot(n+\abs{S})\right)$  
    & $\landauO\left(\abs{\mathcal{F}}\cdot n\right)$
    & $\landauO\left(n^2\right)$ \\
    \bottomrule
  \end{tabular}
  \caption{Computational cost of a single evaluation of each test
    statistic. The symbol $*$ either stands for $\operatorname{max}$ or $\operatorname{sum}$.}
  \label{table:complexity}
  \spacingset{\spacevar}
\end{table}
For the comparison sets from the previous sections we have
$\abs{\mathcal{F}^1(\mathcal{E})}=\landauO(\abs{\mathcal{E}^2})$ and
$\abs{\mathcal{F}^2(\mathcal{E})}=\landauO(\abs{\mathcal{E}})$, which
implies that if we choose $\mathcal{E}$ to contain of order $\log(n)$
sets the complete complexities of our change point based tests
$T^{*,\mathcal{F}}_i$ are $\landauO(B\cdot n\log(n))$ in the low dimensional
setting.

Depending on the number of potential predictor variables an exhaustive
search over all subsets can quickly become unfeasible. For such
settings we would suggest reducing the number of potential predictor
variables by using an appropriate pre-selection, for example the Lasso
as also described in \citet{peters2016}. Additionally, there is often
no need to compute all subsets due to the fact that the intersection
in \eqref{eq:Shat2} is computed. For details see \citet{peters2016}.

\section{Instantaneous causal effects in multivariate time series}\label{sec:timeseries}
We have mentioned in Section~\ref{sec:icp} that the structural invariance
defined in Definition~\ref{def:invariantset} does not restrict 
the dependence structure between the predictor variables. This implies that dependence structures as Model~A in Figure~\ref{fig:time_dependence} is
included in our framework. Our
framework can be adapted to also include time dependencies as the ones in Model~B and Model~C in
Figure~\ref{fig:time_dependence}. In this section, we show that this
is possible whenever the dependence of $Y$ on $X$ and the past of
$(Y,X)$ is linear Gaussian and higher order Markovian.
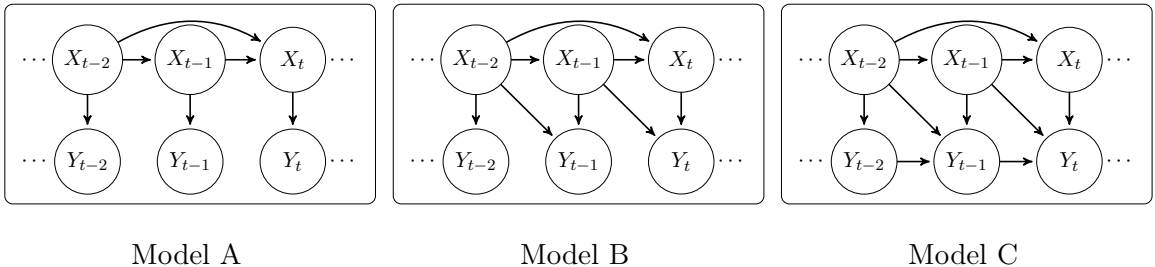
\begin{figure}[h!]
\spacingset{1}
\begin{minipage}{0.32\textwidth}
  \scalebox{0.75}{
  \begin{tikzpicture}[scale=3,framed,background rectangle/.style={draw=black,rounded corners}]
    \draw (-0.9,0.6) node {$\cdots$};
    \draw (-0.9,0) node {$\cdots$};
    \draw (0.9,0.6) node {$\cdots$};
    \draw (0.9,0) node {$\cdots$};
    \tikzstyle{VertexStyle} = [shape = circle, minimum width = 3em,draw]
    \SetGraphUnit{2}
    \Vertex[Math,L=X_{t-2},x=-0.6,y=0.6]{X1}
    \Vertex[Math,L=X_{t-1},x=0,y=0.6]{X2}
    \Vertex[Math,L=X_t,x=0.6,y=0.6]{X3}
    \Vertex[Math,L=Y_{t-2},x=-0.6,y=0]{Y1}
    \Vertex[Math,L=Y_{t-1},x=0,y=0]{Y2}
    \Vertex[Math,L=Y_t,x=0.6,y=0]{Y3}
    \tikzstyle{EdgeStyle} = [->,>=stealth',shorten > = 2pt]
    \Edge(X1)(X2)
    \Edge(X2)(X3)
    \Edge(X1)(Y1)
    \Edge(X2)(Y2)
    \Edge(X3)(Y3)
    \tikzset{EdgeStyle/.append style = {->, bend left}}
    \Edge(X1)(X3)
  \end{tikzpicture}}
\center{Model A}
\end{minipage}
\begin{minipage}{0.32\textwidth}
  \scalebox{0.75}{
  \begin{tikzpicture}[scale=3,framed,background rectangle/.style={draw=black,rounded corners}]
    \draw (-0.9,0.6) node {$\cdots$};
    \draw (-0.9,0) node {$\cdots$};
    \draw (0.9,0.6) node {$\cdots$};
    \draw (0.9,0) node {$\cdots$};
    \tikzstyle{VertexStyle} = [shape = circle, minimum width = 3em,draw]
    \SetGraphUnit{2}
    \Vertex[Math,L=X_{t-2},x=-0.6,y=0.6]{X1}
    \Vertex[Math,L=X_{t-1},x=0,y=0.6]{X2}
    \Vertex[Math,L=X_t,x=0.6,y=0.6]{X3}
    \Vertex[Math,L=Y_{t-2},x=-0.6,y=0]{Y1}
    \Vertex[Math,L=Y_{t-1},x=0,y=0]{Y2}
    \Vertex[Math,L=Y_t,x=0.6,y=0]{Y3}
    \tikzstyle{EdgeStyle} = [->,>=stealth',shorten > = 2pt]
    \Edge(X1)(X2)
    \Edge(X2)(X3)
    \Edge(X1)(Y1)
    \Edge(X2)(Y2)
    \Edge(X3)(Y3)
    \Edge(X1)(Y2)
    \Edge(X2)(Y3)
    \tikzset{EdgeStyle/.append style = {->, bend left}}
    \Edge(X1)(X3)
  \end{tikzpicture}}
\center{Model B}
\end{minipage}
\begin{minipage}{0.32\textwidth}
  \scalebox{0.75}{
   \begin{tikzpicture}[scale=3,framed,background rectangle/.style={draw=black,rounded corners}]
     \draw (-0.9,0.6) node {$\cdots$};
     \draw (-0.9,0) node {$\cdots$};
     \draw (0.9,0.6) node {$\cdots$};
     \draw (0.9,0) node {$\cdots$};
     \tikzstyle{VertexStyle} = [shape = circle, minimum width = 3em,draw]
     \SetGraphUnit{2}
     \Vertex[Math,L=X_{t-2},x=-0.6,y=0.6]{X1}
     \Vertex[Math,L=X_{t-1},x=0,y=0.6]{X2}
     \Vertex[Math,L=X_t,x=0.6,y=0.6]{X3}
     \Vertex[Math,L=Y_{t-2},x=-0.6,y=0]{Y1}
     \Vertex[Math,L=Y_{t-1},x=0,y=0]{Y2}
     \Vertex[Math,L=Y_t,x=0.6,y=0]{Y3}
     \tikzstyle{EdgeStyle} = [->,>=stealth',shorten > = 2pt]
     \Edge(X1)(X2)
     \Edge(X2)(X3)
     \Edge(Y1)(Y2)
     \Edge(Y2)(Y3)
     \Edge(X1)(Y1)
     \Edge(X2)(Y2)
     \Edge(X3)(Y3)
     \Edge(X1)(Y2)
     \Edge(X2)(Y3)
     \tikzset{EdgeStyle/.append style = {->, bend left}}
     \Edge(X1)(X3)
   \end{tikzpicture}}
\center{Model C}
\end{minipage}
\caption{Illustration of three potential time dependence
  structures. Model A can be directly applied to our framework
  described in Section~\ref{sec:icp}, Model B and C require a slight
  modification which accounts for the dependence of $Y_t$ on its
  past.}
\label{fig:time_dependence}
\spacingset{\spacevar}
\end{figure}

Consider a sequence $(Y_t,X_t)_{t\in\{1,\dots,n\}}$ as
described at the beginning of Section~\ref{sec:icp} for which 
there
exists $p\in\{0,\dots,n-2\}$, $S^*\subseteq\{1,\dots,d\}$,
$\beta=(\beta_1,\dots,\beta_{\abs{S^*}})^{\top}\in(\R\setminus\{0\})^{\abs{S^*}\times
  1}$ and $B_k\in\R^{(d+1)\times 1}$ for $k\in\{1,\dots,p\}$, satisfying for all
$t\in\{p+1,\dots,n\}$ that
\begin{equation}
  \label{eq:armodel}
  Y_t=X_t^{S^*}\beta+\sum_{k=1}^p(Y_{t-k},X_{t-k})B_k+\epsilon_t,
\end{equation}
where $\epsilon_{p+1},\dots,\epsilon_n\iid\mathcal{N}(0,\sigma^2)$ are
independent noise variables. Such a condition is for example satisfied
if $(Y_t,X_t)$ is a structural vectorized auto-regressive process
(SVAR) \citep[see e.g.,][Chapter 9]{lutkepohl2005}. However, this
framework also allows for more complicated (e.g., non-linear)
structures between the predictor variables. We can adapt our
methodology from the previous sections to estimate the set of
\emph{instantaneous} effects $S^*$ for which the model
\eqref{eq:armodel} remains invariant. The central idea is to include
the past $p$ lags of all variables into each regression step. To make
this more precise, for each potential set of instantaneous effects
$S\subseteq\{1,\dots,d\}$ we do not regress $Y_t$ 
on $X_t^S$, as in the previous sections,
but on
\begin{equation*}
  Z_t^{S,p}\coloneqq
  (X_t^S,Y_{t-1},X_{t-1},\dots,Y_{t-p},X_{t-p})
\end{equation*}
instead. We denote
the corresponding scaled residuals by
\begin{equation}
  \label{eq:scaled_residZ}
  \sresid^{S,p}
  \coloneqq {(\vI-\vP^{S,p}_{\vZ})\vY}\,/\,{\norm{(\vI-\vP^{S,p}_{\vZ})\vY}_2},
\end{equation}
where $\vP^{S,p}_{\vZ}$ is the projection operator onto the linear
span of $\vZ^{S,p} = (Z_{p+1}^{S,p},\ldots ,Z_n^{S,p})$ and with a
slight abuse of notation $\vY=(Y_{p+1},\dots,Y_n)$.  Equivalently to
\eqref{eq:HOS_vec}, we consider the following null hypothesis
expressed in terms of $Z_t^{S,p}$.
\begin{equation*}
  \tildeHOS:
  \begin{cases}
    \exists\eta\in\R^{(\abs{S}+(d+1)p)},\sigma\in(0,\infty)\text{ such
    that }\forall t\in\{p+1,\dots,n\}:\\
    Y_t=Z_t^{S,p}\eta+\epsilon_t\text{ with }\epsilon_t \independent
    Z_t^{S,p}\text{ and }\epsilon_{p+1},\dots,\epsilon_n\iid \mathcal{N}(0,\sigma^2).
  \end{cases}
\end{equation*}
Then, the same reasoning as in Section~\ref{sec:icp} can be applied,
given that the model in \eqref{eq:armodel} remains invariant across
time. The corresponding result is as follows.
\begin{proposition}[level of the scaled residual test including time lags]
  \label{thm:srtest_level2}
  For all measurable functions $T:\R^{n-p}\rightarrow\R$ based on the
  scaled residuals $\sresid^{S,p}$ in \eqref{eq:scaled_residZ}, and
  for all $(\vY,\vX)\sim P\in\tildeHOS$, it
  holds that the hypothesis test $\varphi^{S,p}_{T,B}$ defined as in
  \eqref{eq:hyptest} (where instead of regressing $Y$ on $X^S$ we
  regress $Y$ on $Z^{S,p}$) achieves correct level as $B$ goes to
  infinity, i.e.,
  \begin{equation*}
    \lim_{B\rightarrow\infty}\P_{P}\left(\varphi^{S,p}_{T,B}(\vY,\vX)=1\right)=\alpha.
  \end{equation*}
\end{proposition}
A proof is given in Appendix~\ref{sec:tslevelproof}. In practical
applications, we usually do not know the number of time lags
$p$. Essentially, there are three ways of dealing with this
issue. Firstly, one can include a sufficiently large number of lags
$p$ which accounts for enough of the existing time dependence. Since
we then need to estimate more parameters this will, however, typically
decrease the power of our invariance procedure. A second option would
be to apply variable selection such as AIC or BIC.  As we aim at
finding invariant models rather than models that predict well, one may
also base the variable selection on a criterion that optimizes this
goal. For example, we could go over all reasonable lags and then
select the $p$ which results in the largest causal set, i.e.,
$p=\argmax_{k}\abs{\hat{S}(k)}$ where $\hat{S}(k)$ is our estimator
resulting from using $k$ lags.\footnote{The nature of this idea is
  similar to the one proposed in \citet{Mooij2009}, in which the
  authors evaluate the goodness of a regression function by the
  independence between residuals and predictors rather than the
  residual variance. The independence of residuals is afterwards
  used for causal structure learning.}
As with any variable selection
procedure we need to be careful when interpreting the
confidence statements due to post-selection issues. A third option
which circumvents any post-selection issues is to use the set
$\hat{S}=\cup_{k\in L}\hat{S}(k)$, for some set of potential lags
$L\subset\N$ and then adjust the level using a Bonferroni adjustment
of size $\abs{L}$ to account for multiple testing.

Proposition~\ref{thm:srtest_level2} establishes a framework for dealing
with instantaneous causal effects. This itself allows going beyond the
concept of Granger causality which excludes instantaneous
effects. Furthermore, the power of invariant causal prediction using a
test as in Proposition~\ref{thm:srtest_level2} hinges on the amount of
non-stationarity present in the multivariate time series to detect deviations
from the null-hypothesis $\tildeHOS$; that is,
non-stationarity, which loosely relates to perturbations, is
potentially beneficial for inferring causal time-instantaneous
structures. Section~\ref{sec:expshocks} illustrates this empirically.

\section{Numerical experiments} \label{sec:experiments} 

We apply our methodology on both artificial data sets (based on SCMs) and real data.  In
Section~\ref{sec:num_sim}, we summarize the findings from the
numerical simulations and in Section~\ref{sec:expreal} we apply our
method to a real world monetary example.

\subsection{Numerical simulations}\label{sec:num_sim}

We empirically verify the theoretical results we have developed in
Section~\ref{sec:consistency}. In particular, we show that detecting
the difference in regression coefficients and residual variance using
the combined test based on \eqref{eq:teststat_single} and the
decoupled test based on \eqref{eq:teststat_beta} and
\eqref{eq:teststat_sigma} yield different convergence rates
(Appendix~\ref{sec:exptheo}).  In Appendix~\ref{sec:exppower}, we
compare the power of different choices of the test statistic, e.g.,
when combining the different environments using a sum or a maximum,
see~\eqref{eq:teststat_multi1}. %
One further loses only little power when the true underlying
environments are unknown compared to the traditional approach that
exploits the precise location of the change points.  In fact, in some
situations and for large sample sizes, it is beneficial to split the
true environments into smaller sets, see
Appendix~\ref{sec:expsplitting}.  Finally, in
Appendix~\ref{sec:expshocks}, we consider the time series setting
discussed in Section~\ref{sec:timeseries}. Due to the time dependence,
it is possible to infer the causal structure even if there is a shock
in the dynamical system at a single time instance (leading to an
environment of size one). Whereas for some practical applications, it
might be difficult to distinguish a shock from an outlier, we show
that in case of an outlier, our method remains conservative.

\subsection{Monetary policy example} \label{sec:expreal}

To illustrate the usefulness of our method for practical applications
we apply it to a real world data set related to the monetary policy of
the Swiss National Bank (SNB) (see Appendix~\ref{sec:dataset} for
details). Our data set consists of monthly data from January 1999 to
January 2017 based on the variables given in
Table~\ref{table:variables}. Our goal is to find the instantaneous
monthly causal predictors that affect the log returns of the Euro -
Swiss Franc exchange rate (variable $Y$). The predictors we selected
can be grouped into two categories. Variables $X^1$ to $X^6$ are all
related to the policies of the SNB whereas variables $X^7$ to $X^9$
describe the economic conditions in Switzerland and the 19 Euro zone
countries. As the SNB cannot directly set the exchange rate it is
reasonable to assume that any active influence on the exchange rate
either occurs through one of the SNB variables or due to changes in
the economic conditions. Since we expect a time dependence in the
target variable $Y$, we apply our method by including lagged variables
as described in Section~\ref{sec:timeseries}. After regressing the
target variable $Y$ on the past of $(Y,X)$, the mean as well as the
regression coefficients remain fairly stable. Hence, any of our tests
testing for shifts in either of these quantities are not able to
reject the empty set. In contrast, the residual variance is unstable
and tests testing for these changes are indeed able to reject that the
empty set is invariant. In this example, we therefore only apply tests
capable of detecting instabilities of the second moment.
\begin{table}[h]
  \spacingset{1}
  \small
  \centering
  \begin{tabularx}{\textwidth}{@{} l Y @{}}  
    \toprule
    & \textbf{description} \\
    \midrule
    $Y$   & log returns of end of month exchange rate Euro to Swiss
    Francs\\
    $X^1$ & change in average call money rate (no log transform as part of the values are negative)\\
    $X^2$ & log returns of end of month proportion of foreign currency investments from total
            assets on the balance sheet of the SNB\\
    $X^3$ & log returns of end of month proportion of reserve positions at
            International Monetary Fund (IMF) from total assets on the balance
            sheet of the SNB\\
    $X^4$ & log returns of end of month proportion of monetary assistance loans from
            total assets on the balance sheet of the SNB\\
    $X^5$ & log returns of end of month proportion of Swiss Franc securities
            from total assets on the balance sheet of the SNB\\
    $X^6$ & log returns of end of month proportion of remaining assets from total assets on the
            balance sheet of the SNB\\
    $X^7$ & log returns of Swiss GDP (in Euro) resulting from
            interpolation of quarterly (seasonally adjusted) data and adjusted using the
            monthly average exchange rate\\
    $X^8$ & log returns of Euro zone GDP resulting from an
            interpolation of quarterly (seasonally adjusted) GDP data\\
    $X^9$ & inflation rate for Switzerland computed from the
            monthly consumer price index (CPI)\\
    \bottomrule
  \end{tabularx}
  \caption{Description of each variable in the data set.}
  \label{table:variables}
  \spacingset{\spacevar}
\end{table}
In Figure~\ref{fig:difflags}, we plot the
$p$-values for different lags resulting from the block-wise variance
test, the block-wise decoupled test, the block-wise combined test, the
smoother based variance test and the HSIC based test, all of which are
introduced in Sections~\ref{sec:test_stat_cp}
and~\ref{sec:gradual_shifts}. For comparison purposes, we also apply
the following non-causal method: Fit a linear model including all
instantaneous effects as well as all lagged effects and compute the
$p$-values from the standard t-test.
\begin{figure}[h!]
  \spacingset{1}
  \centering
  \scalebox{1}{
  \input{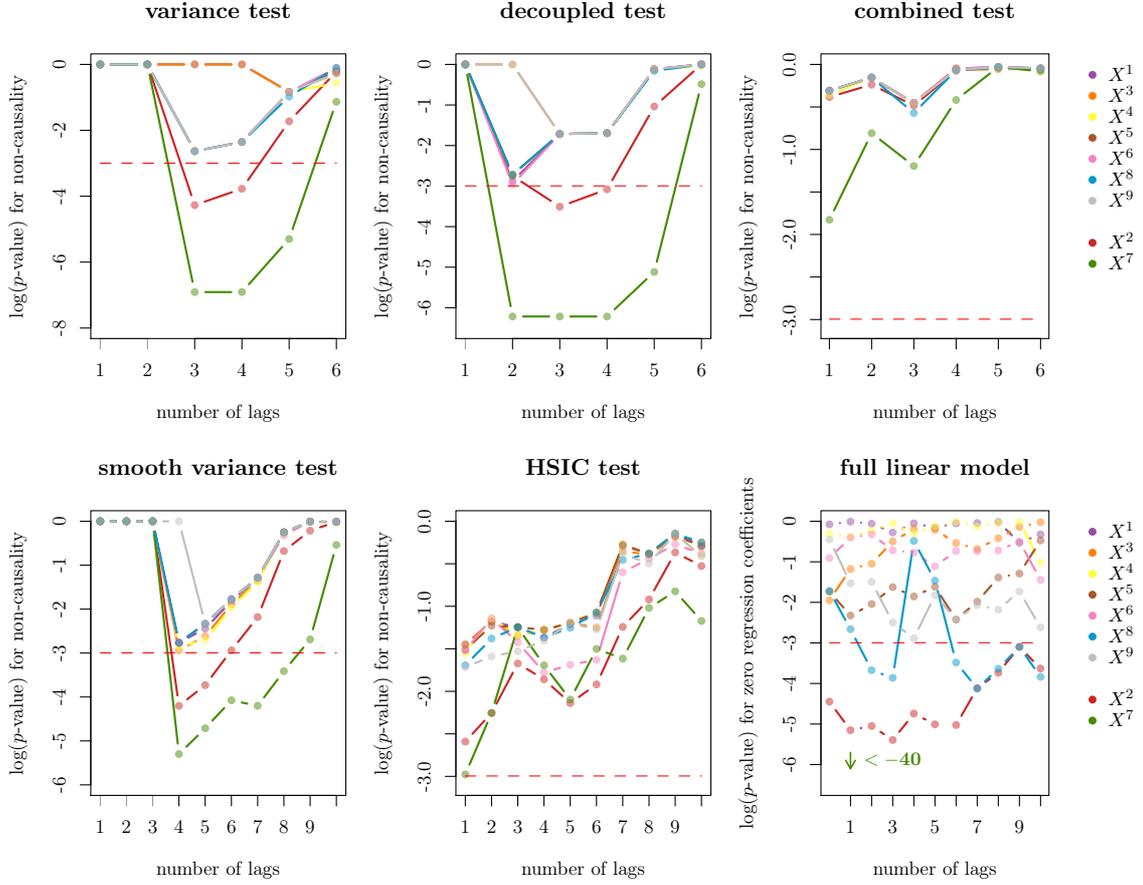}}
\caption{Monetary policy example. The block-wise variance test, the
  decoupled test and the smooth variance test are able to reject
  non-causality of the variables $X^2$ and $X^7$ at a $5\%$-level
  (dashed red line) for $4$ lags. In contrast, the combined test and
  the HSIC based test do not have sufficient power in this
  example. For all methods the ordering of the $p$-values among the
  variables seem robust with respect to the choice of lags $p$. The
  number of lags for which all sets are rejected are marked light gray
  on the $x$-axis. As the number of lags increases the $p$-values
  increase, which makes sense as the model begins to over-fit to the
  data. In contrast, the $p$-values for zero regression coefficients
  resulting from fitting the full linear model do not show this
  behavior.}
  \label{fig:difflags}
  \spacingset{\spacevar}
\end{figure}
The results show that the predictors $X^2$ and $X^7$ appear to be
causally significant for most methods, or at least they consistently
lead to the lowest $p$-values. From an economic viewpoint this also
makes sense: Variable $X^2$ represents the foreign currency
investments which are a known tool of the SNB to reduce the value of
the Swiss Franc. Variable $X^7$ is the Swiss GDP, an important
economic indicator; it seems plausible that this is a causal predictor
as well. Additionally, the plots show that the $p$-values tend to
increase when adding more lags. This happens because including too
many lags leads to models that heavily over-fit, in which case our
tests lose power. Moreover, the results show that both the combined
test and the HSIC based test have less power here. In particular, they
are not able to reject the model that includes only one lag.

This example is only an illustration of potential
applications to real world data sets. In practice, one might benefit from a 
more in-depth analysis and a careful a priori selection of the predictors that are included in the
model. In our example, one could argue that, instead of taking GDP, it
might be useful to use more specific indicators such as the purchasing
managers index (PMI) or other economic measures that might be more
directly linked to the exchange rate. However, due to the fact that we
obtain economically plausible results which post-hoc validate our
methodology at least to a certain extent, we do believe that the
example illustrates the potential of our approach for practical
applications.

\section{Summary}

We introduce a framework for inferring causal predictors of a target
variable $Y$ from sequentially ordered data. In contrast to classical
invariant prediction \citep{peters2016} we do not need the
knowledge of different environments. Nonetheless, we are able to
ensure exact type I error control (Propositions~\ref{thm:srtest_level}
and~\ref{thm:srtest_level2}).
Given that the data are generated by a change point model, we
additionally prove rates of consistency of our block-wise procedures
(Theorems~\ref{thm:consistent_test_bonferroni}
and~\ref{thm:consistent_test}); more precisely, they can detect any
violation of invariance with a rate which is essentially as fast as
$1/\sqrt{n}$. We furthermore show that our framework can be extended
to include linear time dependencies
(Section~\ref{sec:timeseries}). This opens the door to go beyond the
concept of Granger causality and also allows for instantaneous causal
effects. From this perspective, our methods make use of
non-stationarity (induced by interventions occurring throughout time)
in multivariate time series and use it to infer instantaneous causal
effects. The empirical performance of our methods are illustrated by
simulations. Notably, we verify the convergence rates empirically and
show that our methods have comparable power properties to classical
invariant causal prediction without requiring knowledge about
environments. In the case of time series data our methods are able to
detect causal directions even from a single shock intervention at a
specific point in time. Finally, we illustrate an application to a
real data set about the monetary policy of the Swiss National Bank.

\if0\blind
\section*{Acknowledgements}
The authors thank Nicolai Meinshausen and anonymous reviewers for
helpful discussions and constructive comments. NP was supported by a
research grant (200021\_153504) from the Swiss National Science
Foundation (SNSF) and JP was supported by a research grant (18968)
from VILLUM FONDEN. \fi

\bibliography{refs}

\if1\arxiv
\newpage
\renewcommand\appendixpagename{Supplementary material}
\renewcommand\appendixtocname{Supplementary material}
\appendix
\appendixpage
The supplementary material consists of the following five appendices.
\begin{enumerate}
\item[\textbf{A}] \hyperref[sec:num_sim_detailed]{\textbf{Detailed
      numerical simulations}}
\item[\textbf{B}] \hyperref[sec:support]{\textbf{Supporting material}}
\item[\textbf{C}] \hyperref[sec:proofs]{\textbf{Proofs}}
\item[\textbf{D}] \hyperref[sec:auxiliary]{\textbf{Auxiliary results}}
\item[\textbf{E}] \hyperref[sec:dataset]{\textbf{Monetary policy data set}}
\end{enumerate}

\if0\arxiv
\pagebreak
\fi
\section{Detailed numerical simulations}\label{sec:num_sim_detailed}

This section consists of a detailed presentation of the numerical
simulations. We begin in Section~\ref{sec:exptheo} with an experiment
to provide empirical evidence for the consistency results we have
developed in Section~\ref{sec:consistency}. Section~\ref{sec:exppower}
compares the power of different choices of the test statistic, e.g.\
when combining the different environments using a sum or a maximum,
see~\eqref{eq:teststat_multi1}.
Finally, Section~\ref{sec:expshocks}
shows an experiment for the time series setting discussed in
Section~\ref{sec:timeseries}.

\subsection{Comparison of combined and decoupled test statistic} \label{sec:exptheo}

In this section we empirically verify the convergence rates proved in
Theorem~\ref{thm:consistent_test_bonferroni} and
Theorem~\ref{thm:consistent_test}. For our simulations we
use various (even) sample sizes $n$ and simulate data from a linear
Gaussian model of the form
\begin{equation*}
  Y_i=\beta_i X_i + \epsilon_i, \quad\text{with }\epsilon_i\sim\mathcal{N}(0,\sigma_i^2).
\end{equation*}
To verify the convergence rates we consider alternatives with one
change point at $\frac{n}{2}$, leading to two environments
$e_n\coloneqq\{1,\dots,\frac{n}{2}\}$ and
$f_n\coloneqq\{\frac{n}{2}+1,\dots,n\}$ on which the data is
i.i.d. with fixed parameters $\beta_{e_n},\beta_{f_n},\sigma_{e_n}$
and $\sigma_{f_n}$. In this simulation we consider the three
alternatives specified in Table~\ref{table:threealt}.
\begin{table}[h!]
  \spacingset{1}
  \centering
  \begin{tabular}{lccc}
    \toprule
    & alternative 1 & alternative 2 & alternative 3\\
    \midrule
    $\abs{\beta_{e_n}-\beta_{f_n}}$ & $\frac{\log(n)}{20\cdot n^{\frac{1}{2}}}$ & $\frac{\log(n)}{20\cdot n^{\frac{1}{4}}}$ & $0$\\\addlinespace
    $\abs{\sigma_{e_n}^2-\sigma_{f_n}^2}$ & $0$ & $0$ & $\frac{\log(n)}{n^{\frac{1}{2}}}$\\
    \bottomrule
  \end{tabular}
  \caption{The three alternatives used in the simulations.}
  \label{table:threealt}
  \spacingset{\spacevar}
\end{table}
The resulting plots are given in Figure~\ref{fig:chowbon}. For the
alternatives 1 and 2 they show that the combined test based on
$T_{e_n,f_n}^3$ given in \eqref{eq:teststat_single} is only able to
detect changes in the regression coefficients with a rate of
$n^{-1/4}$, while the decoupled test based on $T_{e_n,f_n}^1$ and
$T_{e_n,f_n}^2$ given in \eqref{eq:teststat_beta} and
\eqref{eq:teststat_sigma} has a rate of $n^{-1/2}$. On the other hand
alternative~3 shows that both tests are able to detect changes in the
noise variance with a rate of $n^{-1/2}$. This corresponds with what
has been proved in Theorems~\ref{thm:consistent_test_bonferroni}
and~\ref{thm:consistent_test}. In particular, the
simulations illustrate that the decoupled test appears (at least in
these examples) to be more powerful even for finite sample sizes. This
indicates both from a theoretical and an empirical point of view that it
is preferable to use the decoupled test rather than the combined test.

\spacingset{1}
\begin{figure}[h!]
  \centering
  \input{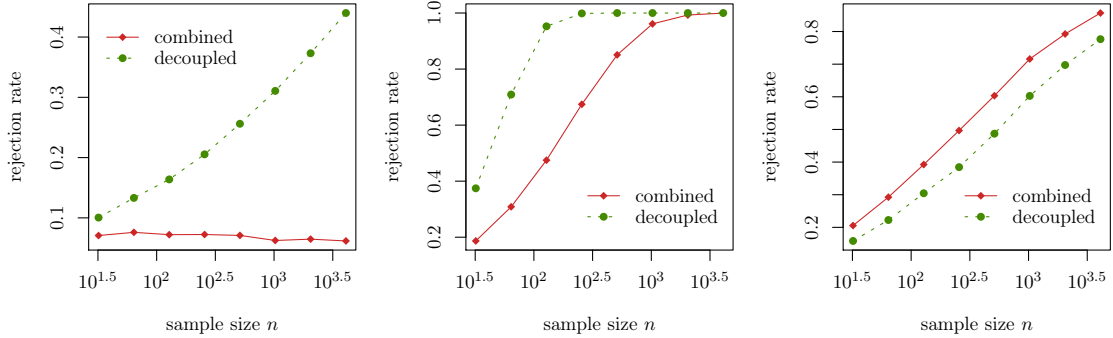}
  \caption{Power for alternatives, depending on $n$, from
    Table~\ref{table:threealt}. From left to right; alternative 1,
    alternative 2 and alternative 3. Comparing the plots for
    alternative 1 and 2 it can be seen that while the decoupled test
    is capable of detecting both alternatives as $n$ increases the
    combined test statistic is only able to reject alternative 2 implying
    a slower consistency rate for detecting changes in regression
    coefficients. The plot for alternative 3 indicates that both tests
    are able to detect shifts in variance with the faster rate of
    $n^{-\frac{1}{2}}$.}
  \label{fig:chowbon}
\end{figure}
\spacingset{\spacevar}

\subsection{Power comparison on simulated data} \label{sec:exppower}
We now apply our methods to simulated data.  As data generating
process we use the linear Gaussian model given in
Figure~\ref{fig:simDAG}.
\begin{figure}[h!]
  \spacingset{1}
  \hspace{0.1 \textwidth}
  \begin{minipage}{0.3\textwidth}
    \centering
    \begin{mdframed}[roundcorner=5pt,backgroundcolor=blue!10]
      \begin{align*}
        X^1&\leftarrow N^1\\
        X^2&\leftarrow \beta^1 X^1 + N^2\\
        Y\phantom{^1}&\leftarrow\beta^2 X^1 + \beta^3 X^2 + N^3\\
        X^3&\leftarrow\beta^4 Y + N^4\\
        \text{with}&\quad N^j\iid\mathcal{N}(\mu_j,\sigma^2_j)
      \end{align*}
    \end{mdframed}
  \end{minipage}
  \hfill
  \begin{minipage}{0.5\textwidth}
    \centering
    \begin{tikzpicture}[scale=1]
      \tikzstyle{VertexStyle} = [shape = circle, minimum width=3em, draw]
      \SetGraphUnit{2}
      \Vertex[Math,L=X^1,x=-2,y=-1]{X1}
      \Vertex[Math,L=X^2,x=-2,y=1]{X2}
      \Vertex[Math,L=Y,x=0,y=0]{Y}
      \Vertex[Math,L=X^3,x=2,y=0]{X3}
      \tikzstyle{EdgeStyle} = [->,>=stealth',shorten > = 2pt]
      \Edge(Y)(X3)
      \Edge(X1)(Y)
      \Edge(X2)(Y)
      \Edge(X1)(X2)
    \end{tikzpicture}
  \end{minipage}
  \caption{(left) structural causal model (SCM) of the observational
    setting with corresponding DAG (right)}
  \label{fig:simDAG}
  \spacingset{\spacevar} 
\end{figure}
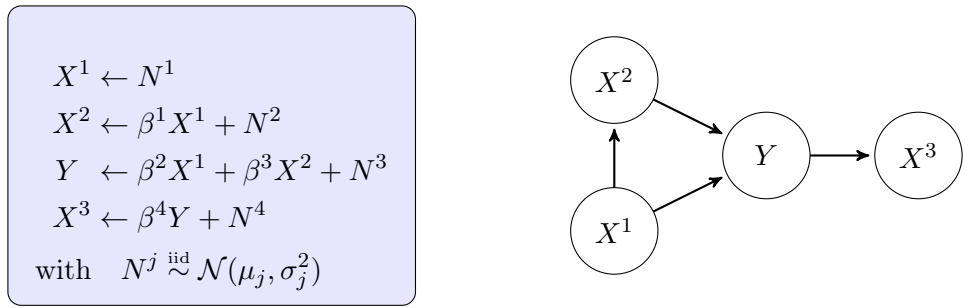
We perform our simulations for the sample sizes
$n\in\{100,200,300,400,500\}$ and for each sample size we generate
$1000$ data sets.  For each of these repetitions, we randomly draw
parameters of the structural causal model 
according to the following distributions
$\beta^j\iid\operatorname{Uniform}([0.5,1.5])$,
$\sigma_j^2\iid\operatorname{Uniform}([0.1,0.3])$ and
$\mu_j\iid\operatorname{Uniform}([0,0.3])$ that are used to
sample from the so-called observational distribution.  In order to
generate different environments, we randomly select two change points
$t_1$ and $t_2$ in $\{1,\dots,n\}$.  This yields the following three
environments.
\begin{compactitem}
\item $e_1=\{1,\dots,t_1\}$: Here, we sample from the observational distribution.
\item $e_2=\{t_1+1,\dots,t_2\}$: Here, we use the model as in the
  first environment but intervene on variable $X^2$, i.e., the
  structural assignment of $X^2$ is replaced by
  $X^2=\beta^1X^1+\tilde{N}^2$, where $\tilde{N}^2$ is a Gaussian
  random variable with mean sampled uniformly between $1$ and $1.5$
  and variance sampled uniformly between $1$ and $1.5$.
\item $e_3=\{t_2+1,\dots,n\}$: Again, we use the same model as in
  environment $e_1$ but this time, we intervene on $X^3$, i.e., the
  structural assignment of $X^3$ 
   is replaced by
  $X^3=\tilde{N}^3$, where $\tilde{N}^3$ is a Gaussian random variable
  with mean sampled uniformly between $-1$ and $-0.5$ and the same
  variance as the noise $N^3$ from the observational setting.
\end{compactitem}
\begin{figure}[h!]
  \spacingset{1}
  \centering
  \input{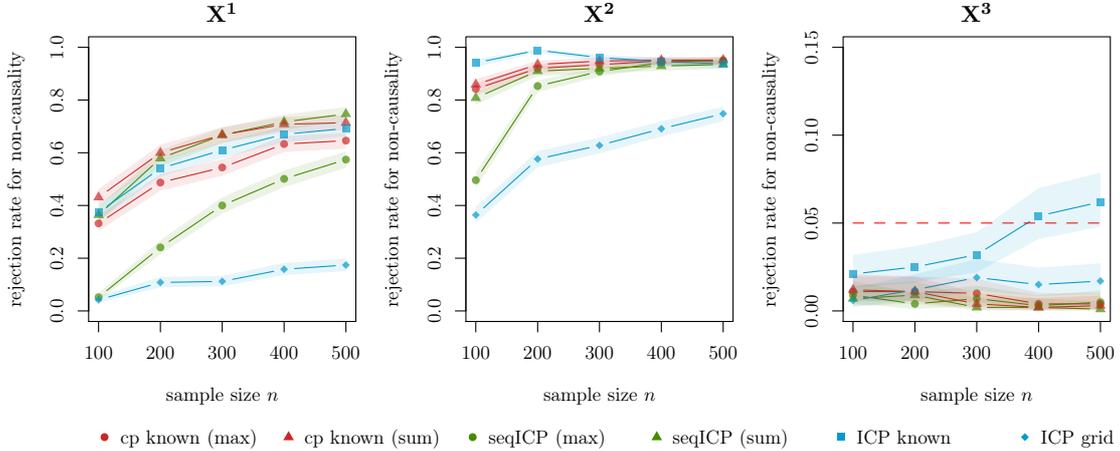}
  \caption{Model in Section~\ref{sec:exppower}. With increasing sample
    size, all methods tend to identify the causal variables $X^1$
    and $X^2$, even if the change points are unknown (green). For the
    sum test statistic (triangle), knowing the environments (red) does
    not improve performance. A possible explanation is that
    sometimes, it might be beneficial to split existing environments,
    see Section~\ref{sec:expsplitting}.}
  \label{fig:powersim1}
  \spacingset{\spacevar}
\end{figure}
The results are shown in Figure~\ref{fig:powersim1}. Here we compare
our method based on unknown change points (seqICP, green) with our
method based on known change points (cp known, red) and the original
version of ICP (ICP, blue) from \citet{peters2016}. Since $X^1$ and
$X^2$ are the true parents of $Y$ in the underlying model, we expect
the methods to reject these two variables (as being non-causal), at
least with increasing sample size.  The figure shows that the sum
(triangle) works slightly better than the maximum (circle),
see~\eqref{eq:teststat_multi1}. 
Furthermore, providing the method with
the true underlying change points and placing the grid points on those
(red) improves over a larger grid (consisting of $10$ grid points)
(green). The difference, however, is not very large and does not seem
significant for the sum estimators. This stability property of the sum
estimator might be due to a phenomenon we investigate in
Section~\ref{sec:expsplitting} below.  The maximum based test
statistic works slightly worse than the sum based statistic if the
sample size in the smallest segments of the grid is too small.  This
may be because the maximum is influenced heavily by the terms that
correspond to such smallest environments, which we expect to have a
large variance.  This effect is not as prominent for the sum, in which
one effectively averages many of such terms.  All the proposed methods
outperform the original version of ICP due to the slightly improved
test statistic.

\subsubsection{Increasing power by splitting environments}\label{sec:expsplitting}
In the example above, for each data set, there are two change points
that have been used in the data generating process.  That is, the data
are i.i.d.\ within each of the three environments $e_1$, $e_2$ and
$e_3$.  Suppose now that we are given the change points and assume
further that the third environment is large compared to the former
two.  We can then run our method using a grid that is placed on the
known change points.  The question arises whether one can benefit (in
terms of power) by splitting the third environment, i.e., by placing
another grid point after the second one.  Intuitively, this should not
be the case for the maximum based test statistic, which is focusing
on the largest difference of distributions between any two
environments that are constructed from the grid.  We observe
empirically, however, that it can be indeed the case for the sum based
test statistic.

As an experiment, we use the same simulation procedure as in
Section~\ref{sec:exppower} above but fix the sample size to be
$n = 200$ and fix the two change points at 15 und 30.  Placing the
grid on those two points yields identification of $X^1$ as a causal
variable in roughly $20\%$ of the repetitions, see
Figure~\ref{fig:powersim2}, red line.  Instead, we can also keep the
location of the first two grid points and split the largest
environment into smaller segments; this is done by introducing
additional grid points between 30 and 200.  Somewhat surprisingly,
this can yield a significant increase of power, see the green line in
Figure~\ref{fig:powersim2}.
\begin{figure}[h!]
  \spacingset{1}
  \centering
  \scalebox{0.7}{
  \input{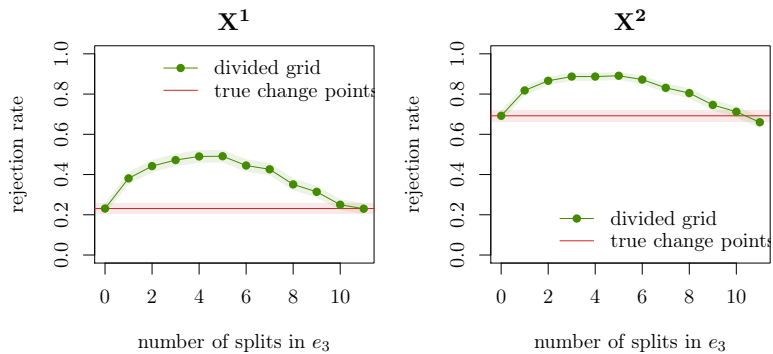}}
\caption{Here, the true change points are assumed to be known. Using
  the sum based estimator with a grid on these true change points
  (red) yields a certain discovery rate of the true causal
  parents. One can gain a significant increase of power, however, by
  splitting the largest environment into smaller segments (green).}
  \label{fig:powersim2}
  \spacingset{\spacevar}
\end{figure}
If one splits an existing segment, one obtains additional terms in the
sum of test statistic, see 
right-hand side of Equation~\eqref{eq:teststat_multi1}.
In the setting above, for example, after splitting environment $e_3$,
we now have tripled the number of terms that measure the difference
between environments $e_1$ and $e_3$.  For rather large environments,
in which the test statistic has relatively small variance even for the
smaller environments, this can be seen as putting more weight on the
corresponding term in the original sum.  This may then ultimately
yield an increase in power of the procedure.  At some point, this
effect levels off, of course.  After introducing too many additional
grid points, the variances of the individual test statistics are so
large that one cannot detect any difference in distributions between
the segments anymore. This is why the green line has to fall below the
red line with increasing number of splits.

\subsection{Shocks in time series} \label{sec:expshocks} 

In this section, we look at a time series example with three variables
$X$, $Y$, and $Z$ with a linear autoregressive structure (with one lag)
given by the DAG in Figure~\ref{fig:timeseries}. More precisely, we use
the following structural time series model
\begin{align*}
  X_t&\leftarrow 0.5X_{t-1}+0.1Y_{t-1}+0.1Z_{t-1}+\epsilon_t^X\\
  Y_t&\leftarrow 0.5X_t+0.1X_{t-1}+0.2Y_{t-1}+0.2Z_{t-1}+\epsilon_t^Y\\
  Z_t&\leftarrow 0.2X_t+0.2Y_t+0.4X_{t-1}+0.4Y_{t-1}+0.2Z_{t-1}+\epsilon_t^Z.
\end{align*}
The choice of parameters is such that all arrows in the DAG in
Figure~\ref{fig:timeseries} correspond to non-zero coefficients.
\spacingset{1}
\begin{figure}[h!]
  \hspace{0.1 \textwidth}
  \begin{minipage}{0.4\textwidth}
    \scalebox{0.75}{
      \input{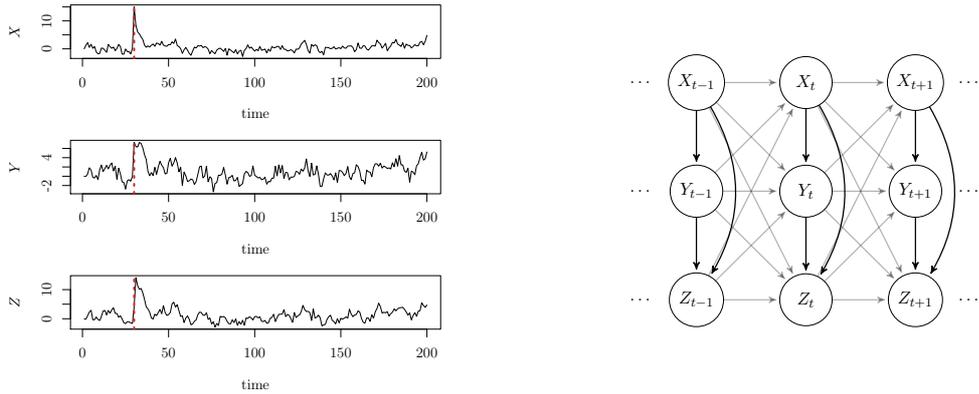}}
  \end{minipage}
  \hfill
  \begin{minipage}{0.55\textwidth}
    \centering
      \scalebox{0.6}{
        \begin{tikzpicture}[scale=1.2]
          \tikzstyle{VertexStyle} = [minimum width=3em, draw=none]
          \Vertex[Math, L=\dots, x=-3,y=2]{past1}
          \Vertex[Math, L=\dots,x=-3,y=0]{past2}
          \Vertex[Math, L=\dots,x=-3,y=-2]{past3}
          \Vertex[Math, L=\dots,x=3,y=2]{future1}
          \Vertex[Math, L=\dots,x=3,y=0]{future2}
          \Vertex[Math, L=\dots,x=3,y=-2]{future3}
          \SetGraphUnit{2}
          \tikzstyle{VertexStyle} = [shape = circle, minimum width=3em, draw]
          \Vertex[Math,L=X_{t-1},x=-2,y=2]{X1}
          \Vertex[Math,L=Y_{t-1},x=-2,y=0]{Y1}
          \Vertex[Math,L=Z_{t-1},x=-2,y=-2]{Z1}
          \Vertex[Math,L=X_{t},x=0,y=2]{X2}
          \Vertex[Math,L=Y_{t},x=0,y=0]{Y2}
          \Vertex[Math,L=Z_{t},x=0,y=-2]{Z2}
          \Vertex[Math,L=X_{t+1},x=2,y=2]{X3}
          \Vertex[Math,L=Y_{t+1},x=2,y=0]{Y3}
          \Vertex[Math,L=Z_{t+1},x=2,y=-2]{Z3}
          \tikzstyle{EdgeStyle} = [->,>=stealth',shorten > = 2pt, opacity=0.3]
          \Edge(X1)(X2)
          \Edge(Y1)(Y2)
          \Edge(Z1)(Z2)
          \Edge(X2)(X3)
          \Edge(Y2)(Y3)
          \Edge(Z2)(Z3)
          \Edge(X1)(Y2)
          \Edge(X1)(Z2)
          \Edge(X2)(Y3)
          \Edge(X2)(Z3)
          \Edge(Y1)(X2)
          \Edge(Y1)(Z2)
          \Edge(Y2)(X3)
          \Edge(Y2)(Z3)
          \Edge(Z1)(X2)
          \Edge(Z1)(Y2)
          \Edge(Z2)(X3)
          \Edge(Z2)(Y3)
          \tikzstyle{EdgeStyle} = [->,>=stealth',shorten > = 2pt]          
          \Edge(X1)(Y1)
          \Edge(X2)(Y2)
          \Edge(X3)(Y3)
          \Edge(Y1)(Z1)
          \Edge(Y2)(Z2)
          \Edge(Y3)(Z3)
          \tikzset{EdgeStyle/.append style = {->, bend left}}
          \Edge(X1)(Z1)
          \Edge(X2)(Z2)
          \Edge(X3)(Z3)
        \end{tikzpicture}}
    \end{minipage}
    \caption{(left) Example of a three variable vector autoregressive
      time series with a shock of size $15$ at time $t=30$ (right) DAG
      of generative time series model.}
    \label{fig:timeseries}
\end{figure}
\spacingset{\spacevar}
For our simulations we use $n=200$. We then draw the time point at
which we intervene uniformly from $\{1,\ldots,n\}$. Our intervention
consists of setting the structural assignment of $X$ at this time
point to the desired shock strength, i.e., the shock intervention
happens only at this one particular instance in time and the
structural assignment of $X$ is changed back to its original form in
the next time step. Due to the time and structural dependence the
shock propagates and spreads to the other variables. An example time
series with a shock intervention of size $15$ at the time point $t=30$
is illustrated in Figure~\ref{fig:timeseries}. In our simulations, we
resample this model $1000$ times for each shock strength in
$\{0,2,\ldots,30\}$ and apply our method using the decoupled test
based on $T_1^{\operatorname{sum},\mathcal{F}^2(\mathcal{E}^G)}$ and
$T_2^{\operatorname{sum},\mathcal{F}^2(\mathcal{E}^G)}$, where
$G=(20,40,\dots,180)$ and $\mathcal{F}^2$ is defined in
\eqref{eq:F2}. The results are illustrated in Figure~\ref{fig:shock}.
\begin{figure}[h!]
  \spacingset{1}
  \centering
  \scalebox{0.8}{
  \input{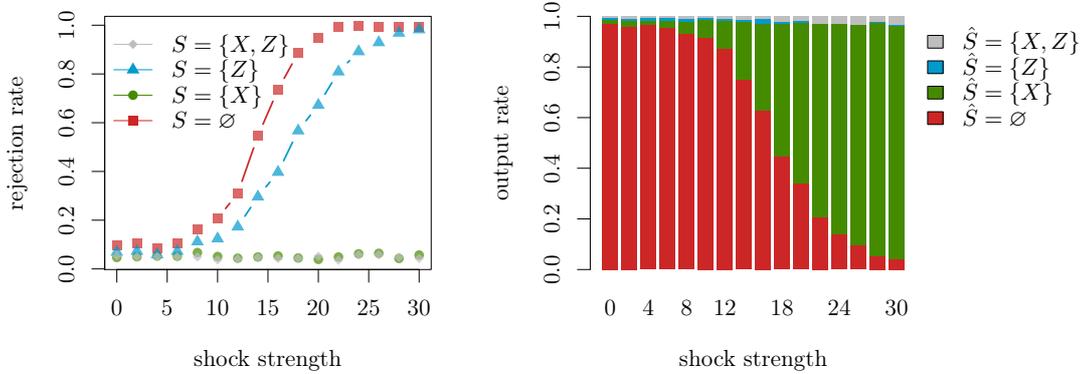}}
\caption{Model given in Section~\ref{sec:expshocks} with shock
  interventions. The left plot shows how the rejection rates of $\tildeHOS$
  for each possible set $S$ depends on the shock strength. For
  sufficiently large shock strengths our method is capable of
  rejecting the empty set (red) and the non-invariant set $\{Z\}$
  (blue). The right plot shows the outcome of our parent set estimator
  $\hat{S}$. The coverage property guarantees that the method returns
  $\{Z\}$ (blue) and $\{X,Z\}$ (gray) in at most $5\%$ of the
  cases.}
  \label{fig:shock}
  \spacingset{\spacevar}
\end{figure}
One might argue that in practical applications it might not be
possible to distinguish between shock interventions and outliers. We
therefore also analyze how our method behaves if instead of a shock
intervention we simply set one value of $Y$ as an outlier, i.e., we
sample the complete time series without any intervention and then set
the value of $Y$ at a random time point to a fixed value. The results
are illustrated in Figure~\ref{fig:outlier}. They show that 
with increasing outlier 
size
one obtains a model misspecification. Our method therefore stays
conservative and outputs the empty set. It does not give an
informative answer but does not output a mistake either.
\begin{figure}[h!]
  \spacingset{1}
  \centering
  \scalebox{0.8}{
  \input{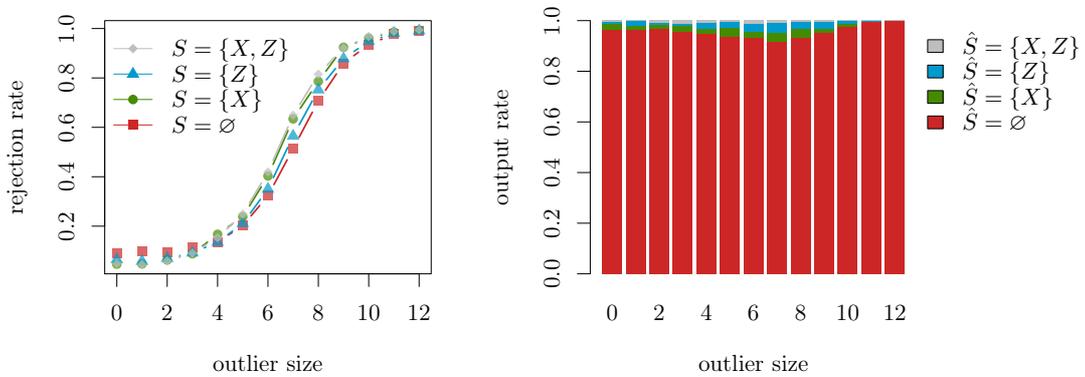}}
\caption{Model given in Section~\ref{sec:expshocks} with outliers. Given an
  outlier on $Y$ instead of a shock intervention the model is no
  longer invariant for any of the sets $S$. This can be seen in the
  left plot, where given that the outlier is large enough the null
  hypothesis $\tildeHOS$ is rejected for all sets. The right plot shows
  that even though the outlier is a model misspecification the
  estimator $\hat{S}$ remains conservative and generally outputs the
  empty set (red).}
  \label{fig:outlier}
  \spacingset{\spacevar}
\end{figure}

\pagebreak
\section{Supporting material}\label{sec:support}

\begin{remark}[violations of the linear Gaussian
  assumption]\label{rmk:nonGaussianNoise}
  The procedure described above relies on the assumption of a linear
  Gaussian model. An interesting extension for practical applications
  would be to allow
  \begin{compactitem}
  \item for non-linear settings, i.e., by replacing the linear dependence
    in Definition~\ref{def:invariantset}~(a) by
    $Y_t=f(X_t^S,\epsilon_t)$ where $f$ is in some general class of
    functions $\mathcal{F}$
  \item and for non-Gaussian noise settings, i.e., by allowing for an
    arbitrary noise distribution $G_{\epsilon}$ in
    Definition~\ref{def:invariantset}~(b).
  \end{compactitem}
  One option is to use a permutation approach as follows; First
  use a general regression procedure to estimate the function $f$ and
  compute residuals $R_1,\dots,R_n$, then in a second step approximate
  the null distribution of a test statistic $T(R_1,\dots,R_n)$ by
  permuting the time index of the residuals. Given, that our estimate
  of $f$ is very close to the true function the residuals should be
  approximately i.i.d. hence (approximately) justifying a permutation
  approach. While we believe this approach is interesting from a
  practical viewpoint, it is only a heuristic. Moreover, it turns out
  to be rather difficult to get precise results about the asymptotic
  level of such a testing procedure.
\end{remark}

\begin{remark}[Obtaining environments by clustering]\label{rem:clustering}
  It is tempting to use $(\vY,\vX)$ to construct environments. For
  example, one could use a clustering procedure on one of the
  variables $X^j$. In general, however, this can break the level
  guarantees of the test. To see this, assume we are given
  observations from the following Gaussian SCM
  \begin{align*}
    Y\leftarrow\epsilon_Y, \quad
    X\leftarrow\operatorname{sign}(Y)+\epsilon_X.
  \end{align*}
  Clearly, $Y$ has no parents implying that the empty set is
  invariant. However, constructing two environments by clustering on
  the sign of $X$ results in a changing distribution of $Y$ across the
  two environments, hence breaking the invariance. A similar counter
  example can be constructed by letting the noise of $Y$ be bi-modal,
  then the same problem occurs even if $X$ depends on $Y$
  linearly. The problem is that the clustering is based on the noise
  of $Y$. One way of avoiding this is to only cluster using the
  ancestors of $Y$. Such a method is proposed in \citet{Heinze2017}.
\end{remark}

\begin{remark}[Comparison to change point methods]\label{rem:changepoint}
  While our proposed method also covers the case of smoothly varying
  shifts we have analyzed its power in a change point model. This
  might lead to the question of how our method relates to two-stage
  procedures which first identify change points and then proceed to
  infer the causal structure based on these environments. Most
  importantly, the difference is that our method directly optimizes
  the (non-)invariance required to infer causal relations, while a
  two-step procedure first solves a change point detection problem,
  which is only indirectly linked to (non-)invariances. A scenario
  which illustrates this is given by a model consisting of very many
  changes for which only very few actually lead to non-invariant
  models. A two-stage procedure will necessarily run into power issues
  due to the many small environments, while our method will not be
  affected in the same way. A second major problem with the two-step
  procedure is that the change point procedure is only allowed to be
  applied to the target variable $Y$ and not to the predictors $X$, as
  one otherwise runs into the same problem discussed in
  Remark~\ref{rem:clustering}. This, however, might lead to a loss in
  power, as the following (rather) artificial example given in
  Figure~\ref{fig:comparison_changepoints} illustrates.
  \begin{figure}[h!]
    \spacingset{1}
    \centering
    \input{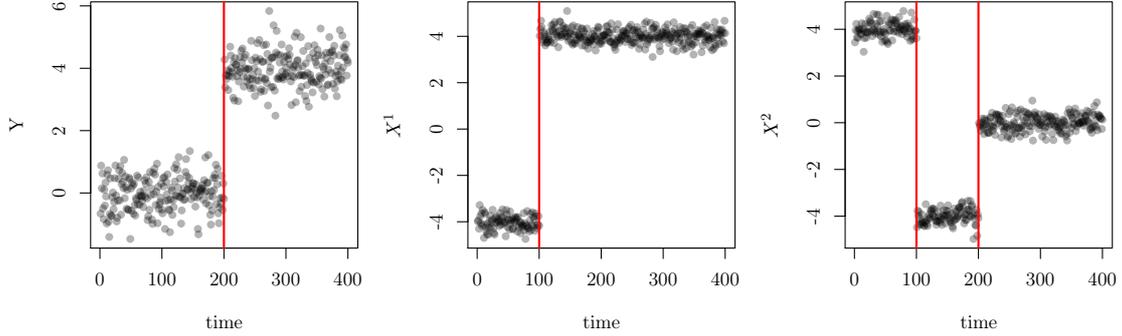}
    \caption{Scatter plots of the variables $(Y, X^1, X^2)$ resulting
      from a linear Gaussian SCM consisting of the two edges
      $X^1\rightarrow Y$ and $X^2\rightarrow Y$. There are two
      interventions (change points) at time points $t_1=100$ and
      $t_2=100$. The first intervention is a mean shift on both $X^1$
      and $X^2$, while the second only shifts $X^2$. The causal
      structure is chosen in such a way that only the second
      intervention is visible in the distribution of $Y$.}
    \label{fig:comparison_changepoints}
    \spacingset{\spacevar}
  \end{figure}
  In this example not all changes are visible in the distribution of
  $Y$ alone. This, in particular, means that any two-stage procedure
  must fail, since there are at most two distinct environments
  ($e_1=\{1,\dots,200\}$ and $e_2=\{201,\dots,400\}$) that can be
  detected in the distribution of $Y$. Applying our procedure to this
  three point grid and using test statistic $T^1$ given in
  \eqref{eq:teststat_beta} (this is essentially the same as standard
  ICP with two environments and a different hypothesis test), we are
  not able to reject the set $\{X^1\}$ ($p$-value of $0.554$). In
  contrast, our procedure (based on a fine grid) is able to exploit
  the differences in distribution in the first half of the
  data. Hence, we are able to reject the set $\{X^1\}$ with a
  $p$-value of $0.039$.
\end{remark}

\spacingset{1}
\begin{algorithm}[H]
\SetAlgoLined
\SetKwInOut{Input}{input}
\SetKwInOut{Output}{output}
\SetKwInOut{Choose}{choose}
\Input{$(Y_1, X_1), \ldots, (Y_n, X_n)$}
\Choose{set of block-wise environments  $\mathcal{E}$ based on a grid;\newline
  comparison set $\mathcal{F}$;\newline
  pairwise test statistic $T_{e,f}$;}
 \For{$S\in\mathcal{P}(\{1,\dots,d\})$}{
   \begin{compactenum}
   \item compute the scaled residuals $\sresid$ by regressing $\vY$ on $\vX^S$
     and normalizing\;
   \item compute pairwise test statistic $T_{e,f}(\sresid)$ for all $(e,f)\in\mathcal{F}$\;
   \item compute combined test statistic $T(\sresid)=\max_{e,f}T_{e,f}(\sresid)$ or $T(\sresid)=\sum_{e,f}T_{e,f}(\sresid)$\;
   \item use simulated null distribution of $T(\sresid)$ to accept or
     reject $\HOS$\;
   \end{compactenum}
 }
 \Output{$\hat{S}=\bigcap_{S:\,\HOS\, \text{accepted}}S$}
 \caption{sequential invariant causal prediction (block-wise
   comparisons)}
 \label{alg:seqICP}
\end{algorithm}
\spacingset{\spacevar}

\pagebreak
\section{Proofs}\label{sec:proofs}

\subsection{Theorem~\ref{thm:consistent_test}}\label{sec:proof_consistent}

In this section we give a proof of
Theorem~\ref{thm:consistent_test}. The key step in the proof is
based on Proposition~\ref{thm:asymptotic_dist_h0_tot} and
Proposition~\ref{thm:asymptotic_dist_ha}. For notational convenience
we drop the set $S$ in the notation, throughout this entire section.

\begin{proof}[Theorem~\ref{thm:consistent_test}]
  The convergence of the Monte-Carlo approximation of the empirical
  distribution is well-established \citep[see e.g.][Example
  11.2.13]{lehmann2005}. It therefore holds $\P$-a.s. that,
  \begin{equation*}
    \lim_{B\rightarrow\infty}c_{T^{\max,\mathcal{F}^1(\mathcal{E}_n)}_3,B}=F_{T^{\max,\mathcal{F}^1(\mathcal{E}_n)}_3}^{-1}(1-\alpha),
  \end{equation*}
  where $c_{T^{\max,\mathcal{F}^1(\mathcal{E}_n)}_3,B}$ is defined in
  \eqref{eq:cutoff_fun}. Define,
  $\varphi^*_{T^{\max,\mathcal{F}^1(\mathcal{E}_n)}_3}\coloneqq\mathds{1}_{\{\abs{T^{\max,\mathcal{F}^1(\mathcal{E}_n)}_3}>F_{T^{\max,\mathcal{F}^1(\mathcal{E}_n)}_3}^{-1}(1-\alpha)\}}$,
  then by the dominated convergence theorem it holds that
  \begin{equation*}
    \lim_{B\rightarrow\infty}\P\left(\varphi_{T^{\max,\mathcal{F}^1(\mathcal{E}_n)}_3,B}(\vY_n,\vX_n)=1\right)=\P\left(\varphi^*_{T^{\max,\mathcal{F}^1(\mathcal{E}_n)}_3}(\vY_n,\vX_n)=1\right).
  \end{equation*}
  It therefore remains to prove that the right-hand side of the above
  equation converges to $1$. Recall, that for a real random variable
  $T$ and constants $a\in\R$ it holds for all $q\in\R$ that
  \begin{equation}
    \label{eq:prop_distfun}
    aF^{-1}_{T}(q)=F^{-1}_{aT}(q),
  \end{equation}
  where $F_{T}^{-1}(q)\coloneqq\inf\{t\in\R\,\rvert\, \P(T\leq t)\geq q\}$ is the generalized inverse
  distribution function.
  In order to simplify the notation we define
  \begin{equation*}
   T_n\coloneqq\sqrt{r_n}\abs{\mathcal{E}_n}^{-\tfrac{1}{k}}T^{\max,\mathcal{F}^1(\mathcal{E}_n)}_3. 
  \end{equation*}
  Assume that $(\vY_n,\vX_n)$ satisfies $\P^{(\vY_n,\vX_n)}\in\HO$,
  then Proposition~\ref{thm:asymptotic_dist_h0_tot} implies that for all
  $\epsilon>0$ there exists $M_{\epsilon}>0$ such that for all
  $n\in\N$ it holds that $\P\left(\abs{T_n}>M_{\epsilon}\right)<\epsilon$.
  This in particular implies that
  \begin{equation*}
    F_{T_n}^{-1}(1-\alpha)\leq M_{\alpha}.
  \end{equation*}
  Next, assume $(\vY_n,\vX_n)$ satisfies
  $\P^{(\vY_n,\vX_n)}\in\HA^n(a_n,b_n)$. Then, there exist
  $i,j\in\{1,\dots,L+1\}$ such that
  $a_n=\abs{\sigma_{e_i(\CPt^*_n)}^2-\sigma_{e_j(\CPt^*_n)}^2}$ and
  $b_n=\norm{\beta_{e_i(\CPt^*_n)}-\beta_{e_j(\CPt^*_n)}}_2$. Moreover,
  by \hyperref[item:c2]{(C2)} there exist sequences of
  environments $(f_n)_{n\in\N}$ and $(g_n)_{n\in\N}$ with
  $f_n,g_n\in\mathcal{E}_n$ such that for sufficiently large $n$ it
  holds that $f_n\subseteq e_i(\CPt^*_n)$ and
  $g_n\subseteq e_j(\CPt^*_n)$. Additionally, we have that
  $$\omega_n\coloneqq\frac{\abs{\mathcal{E}_n}^{\tfrac{1}{k}}}{\sqrt{r_n}}$$
  satisfies $\tfrac{1}{\sqrt{r_n}}=\landauO(\omega_n)$ and by
  assumption also that at least one of the following two conditions
  are satisfied, $\omega_n=\landauo(a_n)$ or
  $\omega_n=\landauo(b_n^2)$. Thus we can apply
  Proposition~\ref{thm:asymptotic_dist_ha} to get that
  \begin{align*}
    \lim_{n\rightarrow\infty}\P\left(\varphi^*_{T^{\max,\mathcal{F}^1(\mathcal{E}_n)}_3}(\vY_n,\vX_n)=1\right)
    &=\lim_{n\rightarrow\infty}\P\left(\abs{T^{\max,\mathcal{F}^1(\mathcal{E}_n)}_3}>F_{T^{\max,\mathcal{F}^1(\mathcal{E}_n)}_3}^{-1}(1-\alpha)\right)\\
    &\overset{\mathclap{\eqref{eq:prop_distfun}}}=\lim_{n\rightarrow\infty}\P\left(\abs{T_n}>F_{T_n}^{-1}(1-\alpha)\right)\\
    &\geq\lim_{n\rightarrow\infty}\P\left(\sqrt{r_n}\abs{\mathcal{E}_n}^{-\tfrac{1}{k}}\abs{T_{f_n,g_n}^3}>M_{\alpha}\right)=1,
  \end{align*}
  which completes the proof of Theorem~\ref{thm:consistent_test}.
\end{proof}

\subsubsection{Intermediate results}

\begin{lemma}[representation of $T_{e_1,e_2}^3$ for true change points]
  \label{thm:expansion_T}
  Let $e_1,e_2\in\mathcal{E}_n(\CPt_n^*)$ then it holds that
  \begin{align*}
    T_{e_1,e_2}^3(\sresid_n)&=\frac{\frac{1}{\abs{e_1}}\norm[\big]{\vepsilon_{e_1}+\vX_{e_1}(\beta_{e_1}-\hat{\beta}_{e_2})}_{2}^2}{\frac{1}{\abs{e_2}}\norm[\big]{\vepsilon_{e_2}+\vX_{e_2}(\beta_{e_2}-\hat{\beta}_{e_2})}_{2}^2}-1\\
    &=\frac{\frac{1}{\abs{e_1}}\left(\vepsilon_{e_1}^{\top}\vepsilon_{e_1}+2\vepsilon_{e_1}^{\top}\vX_{e_1}(\beta_{e_1}-\hbeta_{e_2})+(\beta_{e_1}-\hbeta_{e_2})^{\top}\vX_{e_1}^{\top}\vX_{e_1}(\beta_{e_1}-\hbeta_{e_2})\right)}{\frac{1}{\abs{e_2}}\left(\vepsilon_{e_2}^{\top}\vepsilon_{e_2}+2\vepsilon_{e_2}^{\top}\vX_{e_2}(\beta_{e_2}-\hbeta_{e_2})+(\beta_{e_2}-\hbeta_{e_2})^{\top}\vX_{e_2}^{\top}\vX_{e_2}(\beta_{e_2}-\hbeta_{e_2})\right)}-1
  \end{align*}
\end{lemma}
A proof of this result is given in
Appendix~\ref{sec:intermed_proofs}.

The following theorem gives the asymptotic distribution of
our test statistic under the null hypothesis $\HO$.

\begin{proposition}[asymptotic distribution under $\HO$]
  \label{thm:asymptotic_dist_h0_tot}
  Let $(\vY_{n},\vX_{n})_{n\in\N}$ satisfy
  Assumption~\ref{assumption:asymptotic_cpmodel} and for all $n\in\N$
  satisfy $\P^{(\vY_{n},\vX_{n})}\in\HO^n$, let $\sresid_n$ be the
  scaled residuals defined in \eqref{eq:scaled_resid} corresponding to
  $(\vY_n,\vX_n)$, let
  $\mathcal{E}_n\subseteq\mathcal{P}(\{1,\dots,n\})$ be a sequence of
  collections of pairwise disjoint environments satisfying conditions
  \hyperref[item:c1]{(C1)} and \hyperref[item:c1]{(C3,k)}. Then, it
  holds for all $e_n,f_n\in\mathcal{E}_n$ that
  \begin{equation*}
    T_{e_n,f_n}^3(\sresid_n)=\landauOp\left(\frac{1}{\sqrt{r_n}}\right)
    \quad\text{and}\quad
    T^{\max,\mathcal{F}^1(\mathcal{E}_n)}_3(\sresid_n)
    =\landauOp\left(\frac{\abs{\mathcal{E}_n}^{\frac{1}{k}}}{\sqrt{r_n}}\right),
    \quad\text{as }n\rightarrow\infty.
  \end{equation*}
\end{proposition}
A proof of this result is given in Appendix~\ref{sec:intermed_proofs}.

Next, we give the corresponding theorem for the asymptotic distribution of
our test statistics under the alternative hypothesis $\HA$.

\begin{proposition}[asymptotic distribution under $\HA$]
  \label{thm:asymptotic_dist_ha}
  Let $(\vY_{n},\vX_{n})_{n\in\N}$ satisfy
  Assumption~\ref{assumption:asymptotic_cpmodel} and for all $n\in\N$
  satisfy $\P^{(\vY_{n},\vX_{n})}\in\HA^{n}(a_n,b_n)$, let $\sresid_n$
  be the scaled residuals defined in \eqref{eq:scaled_resid}
  corresponding to $(\vY_n,\vX_n)$. Additionally, let
  $i,j\in\{1,\dots,L+1\}$ such that
  $a_n=\abs{\sigma^2_{e_i(\CPt^*)}-\sigma^2_{e_j(\CPt^*)}}$ and
  $b_n=\norm{\beta_{e_i(\CPt^*)}-\beta_{e_j(\CPt^*)}}_2$. Then,
  assume that $f_n\subseteq e_{i}(\CPt^*_n)$ and
  $g_n\subseteq e_{j}(\CPt^*_n)$ are sequences satisfying that for
  $e_n\in\{f_n,g_n\}$ the sequences $(\sigma_{e_n}^2)_{n\in\N}$ and
  $(\beta_{e_n})_{n\in\N}$ are convergent and the limit of
  $\sigma_{e_n}^2$ is strictly positive and the sequence
  $(\{f_n,g_n\})_{n\in\N}$ satisfies assumptions
  \hyperref[item:c1]{(C1)} and \hyperref[item:c3]{(C3,k)}.
  Then it holds that
  \begin{equation}
    \label{eq:upperboundpart}
    T_{f_n,g_n}^3(\sresid_n)=\landauOp\left(a_n\right)+\landauOp\left(b_n^2\right)+\landauOp\left(\frac{1}{\sqrt{r_n}}\right),\quad\text{as
    }n\rightarrow\infty.
  \end{equation}
  Moreover, let $(\omega_n)_{n\in\N}$ be a sequence which satisfies
  $\tfrac{1}{\sqrt{r_n}}=\landauO(\omega_n)$ and additionally at
  least one of the two conditions $\omega_n=\landauo(a_n)$ or
  $\omega_n=\landauo(b_n^2)$. Then it also holds for all $t\geq 0$ that
  \begin{equation}
    \label{eq:lowerboundpart}
    \lim_{n\rightarrow\infty}\P\left(\tfrac{1}{\omega_n}\abs[\big]{T_{f_n,g_n}^3(\sresid_n)}\leq t\right)=0.
  \end{equation}
\end{proposition}
A proof of this result is given in Appendix~\ref{sec:intermed_proofs}.

\subsubsection{Proofs of intermediate results}\label{sec:intermed_proofs}

\begin{proof}[Lemma~\ref{thm:expansion_T}]
  The result is given by the following straight forward calculation,
  \begin{align*}
    T_{e_1,e_2}^3(\sresid_n)+1
    &=\frac{(\sresid_{e_1}-\vX_{e_1}\hgamma_{e_2})^{\top}(\sresid_{e_1}-\vX_{e_1}\hgamma_{e_2})}{\hat{s}^2_{e_2}\abs{e_1}}\\
    &=\frac{\frac{1}{\abs{e_1}}\norm{\sresid_{e_1}-\vX_{e_1}\hgamma_{e_2}}_2^2}{\frac{1}{\abs{e_2}}\norm{\sresid_{e_2}-\vX_{e_2}\hgamma_{e_2}}_2^2}\\
    &=\frac{\frac{1}{\abs{e_1}}\norm{\resid_{e_1}-\vX_{e_1}(\vX_{e_2}^{\top}\vX_{e_2})^{-1}\vX_{e_2}^{\top}\resid_{e_2}}_2^2}{\frac{1}{\abs{e_2}}\norm{\resid_{e_2}-\vX_{e_2}(\vX_{e_2}^{\top}\vX_{e_2})^{-1}\vX_{e_2}^{\top}\resid_{e_2}}_2^2}\\
    &=\frac{\frac{1}{\abs{e_1}}\norm{\vY_{e_1}-\vX_{e_1}\hbeta-\vX_{e_1}(\vX_{e_2}^{\top}\vX_{e_2})^{-1}\vX_{e_2}^{\top}\vY_{e_2}+\vX_{e_1}(\vX_{e_2}^{\top}\vX_{e_2})^{-1}\vX_{e_2}^{\top}\vX_{e_2}\hbeta}_2^2}{\frac{1}{\abs{e_2}}\norm{\vY_{e_2}-\vX_{e_2}\hbeta-\vX_{e_2}(\vX_{e_2}^{\top}\vX_{e_2})^{-1}\vX_{e_2}^{\top}\vY_{e_2}+\vX_{e_2}(\vX_{e_2}^{\top}\vX_{e_2})^{-1}\vX_{e_2}^{\top}\vX_{e_2}\hbeta}_2^2}\\
    &=\frac{\frac{1}{\abs{e_1}}\norm{\vY_{e_1}-\vX_{e_1}\hbeta_{e_2}}_2^2}{\frac{1}{\abs{e_2}}\norm{\vY_{e_2}-\vX_{e_2}\hbeta_{e_2}}_2^2}\\
    &=\frac{\frac{1}{\abs{e_1}}\norm{\vX_{e_1}\beta_{e_1}+\vepsilon_{e_1}-\vX_{e_1}\hbeta_{e_2}}_2^2}{\frac{1}{\abs{e_2}}\norm{\vX_{e_2}\beta_{e_2}+\vepsilon_{e_2}-\vX_{e_2}\hbeta_{e_2}}_2^2}\\
    &=\frac{\frac{1}{\abs{e_1}}\norm{\vepsilon_{e_1}+\vX_{e_1}(\beta_{e_1}-\hbeta_{e_2})}_2^2}{\frac{1}{\abs{e_2}}\norm{\vepsilon_{e_2}+\vX_{e_2}(\beta_{e_2}-\hbeta_{e_2})}_2^2}\\
    &=\frac{\frac{1}{\abs{e_1}}\left(\vepsilon_{e_1}^{\top}\vepsilon_{e_1}+2\vepsilon_{e_1}^{\top}\vX_{e_1}(\beta_{e_1}-\hbeta_{e_2})+(\beta_{e_1}-\hbeta_{e_2})^{\top}\vX_{e_1}^{\top}\vX_{e_1}(\beta_{e_1}-\hbeta_{e_2})\right)}{\frac{1}{\abs{e_2}}\left(\vepsilon_{e_2}^{\top}\vepsilon_{e_2}+2\vepsilon_{e_2}^{\top}\vX_{e_2}(\beta_{e_2}-\hbeta_{e_2})+(\beta_{e_2}-\hbeta_{e_2})^{\top}\vX_{e_2}^{\top}\vX_{e_2}(\beta_{e_2}-\hbeta_{e_2})\right)},
  \end{align*}
  which completes the proof of Lemma~\ref{thm:expansion_T}.
\end{proof}

\begin{proof}[Proposition~\ref{thm:asymptotic_dist_h0_tot}]
  Under $\HOS$ for all $n\in\N$ there exist fixed
  $\beta_n\in\R^{d\times 1}$ and $\sigma_n^2\in\R_{>0}$ such that for
  all $e_n\in\mathcal{E}_n$ it holds that $\beta_{e_n}=\beta_n$ and
  $\sigma_{e_n}^2=\sigma_n^2$. The main idea in the first part of the
  proof is to use the representation of $T_{e_1,e_2}^3$ given in
  Lemma~\ref{thm:expansion_T}, analyze the convergence of all terms
  individually and finally conclude by combining the convergences.
  
  We begin by proving for all $e_n,f_n\in\mathcal{E}_n$ the following
  estimates
  \begin{enumerate}[(a)]
  \item $\E\left(\left(\frac{1}{\abs{e_{n}}}\vepsilon_{e_n}^{\top}\vepsilon_{e_n}-\sigma_n^2\right)^{2k}\right)\leq \frac{C_1}{\abs{e_{n}}^k}$\label{item:term1}
  \item $\E\left(\left(\frac{1}{\abs{e_{n}}}\vepsilon_{e_n}^{\top}\vX_{e_n}(\beta_{n}-\hat{\beta}_{f_n})\right)^{2k}\right)\leq\frac{C_2}{\abs{f_{n}}^k}$\label{item:term2}
  \item $\E\left(\left(\frac{1}{\abs{e_{n}}}(\beta_{n}-\hat{\beta}_{f_n})^{\top}\vX_{e_n}^{\top}\vX_{e_n}(\beta_{n}-\hat{\beta}_{f_n})\right)^{2k}\right)\leq\frac{C_3}{\abs{f_{n}}^{2k}}$\label{item:term3}
  \end{enumerate}
  
  To prove \eqref{item:term1} consider the following calculation,
  \begin{align*}
    \E\left(\left(\tfrac{1}{\abs{e_{n}}}\vepsilon_{e_n}^{\top}\vepsilon_{e_n}-\sigma_{n}^2\right)^{2k}\right)
    &=\E\left(\left(\tfrac{1}{\abs{e_{n}}}\sum_{i\in
      e_{n}}(\epsilon_i^2-\sigma_{n}^2)\right)^{2k}\right)\\
    &=\frac{1}{\abs{e_{n}}^{2k}}\sum_{i_1,\dots,i_{2k}\in
      e_{n}}\E\left((\epsilon_{i_1}^2-\sigma_{n}^2)\cdots(\epsilon_{i_{2k}}^2-\sigma_{n}^2)\right)\\
    &=\frac{1}{\abs{e_{n}}^{2k}}\sum_{i_1,\dots,i_{k}\in
      e_{n}}\E\left((\epsilon_{i_1}^2-\sigma_{n}^2)^2\cdots(\epsilon_{i_{k}}^2-\sigma_{n}^2)^2\right)\\
    &\leq\frac{C^k}{\abs{e_{n}}^k},
  \end{align*}
  where $C>0$ is a constant satisfying that for all $n\in\N$ and $\ell\in\{1,\dots,k\}$ that
  $\E\left((\epsilon_{i}^2-\sigma_{n}^2)^{2\ell}\right)\leq C.$
  
  Next, we prove \eqref{item:term2}. An application of the Cauchy-Schwarz inequality leads
  to the following inequality
  \begin{equation}
    \label{eq:termsnonind}
    \left(\tfrac{1}{\abs{e_{n}}}\abs[\big]{\vepsilon_{e_n}^{\top}\vX_{e_n}(\beta_{n}-\hat{\beta}_{f_n})}\right)^{2k}\leq\left(\tfrac{1}{\abs{e_n}}\norm{\vepsilon_{e_n}}_{\R^{\abs{e_n}}}^2\right)^{k}\left(\tfrac{1}{\abs{e_n}}\norm[\big]{\vX_{e_n}(\beta_{n}-\hat{\beta}_{f_n})}_{\R^{\abs{e_n}}}^2\right)^{k}.
  \end{equation}
  We now distinguish between the two cases $e_n=f_n$ and $e_n\cap
  f_n=\varnothing$. Begin with the case $e_n\cap f_n=\varnothing$, then $\vepsilon_{e_n}$
  and $\vX_{e_n}$ are both independent of $\hat{\beta}_{f_n}$ and
  hence \eqref{eq:termsnonind} together with a spectral inequality imply that
  \begin{align*}
    &\E\left(\left(\tfrac{1}{\abs{e_{n}}}\abs[\big]{\vepsilon_{e_n}^{\top}\vX_{e_n}(\beta_{n}-\hat{\beta}_{f_n})}\right)^{2k}\right)\\
    &\quad\leq\E\left(\left(\tfrac{1}{\abs{e_n}}\norm{\vepsilon_{e_n}}_{\R^{\abs{e_n}}}^2\right)^{k}\left(\tfrac{1}{\abs{e_n}}\norm[\big]{\vX_{e_n}(\beta_{n}-\hat{\beta}_{f_n})}_{\R^{\abs{e_n}}}^2\right)^{k}\right)\\
    &\quad\leq\E\left(\left(\tfrac{1}{\abs{e_n}}\norm{\vepsilon_{e_n}}_{\R^{\abs{e_n}}}^2\right)^{k}\right)
      \E\left(\left(\lambda_{\max}\left(\tfrac{1}{\abs{e_n}}\vX_{e_n}^{\top}\vX_{e_n}\right)\right)^{k}\right)
      \E\left(\norm[\big]{\beta_{n}-\hat{\beta}_{f_n}}_{\R^d}^{2k}\right).
  \end{align*}
  So together with Theorem~\ref{thm:ols_convergence_strong}, the
  moment bounds on
  $\lambda_{\max}\left(\tfrac{1}{\abs{e_n}}\vX_{e_n}^{\top}\vX_{e_n}\right)$
  and the fact that $\vepsilon_n$ is Gaussian distributed this implies that
  \begin{equation*}
    \E\left(\left(\tfrac{1}{\abs{e_{n}}}\abs[\big]{\vepsilon_{e_n}^{\top}\vX_{e_n}(\beta_{n}-\hat{\beta}_{f_n})}\right)^{2k}\right)\leq \frac{C}{\abs{f_n}^k}.
  \end{equation*}
  Next, assume that $e_n=f_n$. Then, defining the projection matrix
  $\vP_{e_n}\coloneqq\vX_{e_n}(\vX_{e_n}^{\top}\vX_{e_n})\vX_{e_n}^{\top}$
  we get that
  \begin{align*}
    \E\left(\left(\tfrac{1}{\abs{e_{n}}}\abs[\big]{\vepsilon_{e_n}^{\top}\vX_{e_n}(\beta_{n}-\hat{\beta}_{e_n})}\right)^{2k}\right)
    &=\E\left(\left(\tfrac{1}{\abs{e_{n}}}\abs[\big]{\vepsilon_{e_n}^{\top}\vP_{e_n}\vepsilon_{e_n}}\right)^{2k}\right)\\
    &=\frac{1}{\abs{e_{n}}^{2k}}\E\left(\norm[\big]{\vP_{e_n}\vepsilon_{e_n}}^{4k}_{\R^{\abs{e_n}}}\right)\leq\frac{C}{\abs{e_n}^{2k}},
  \end{align*}
  where in the last step we used the same argument as in
  \eqref{eq:projection}. This completes the proof of
  \eqref{item:term2}.

  Finally, in order to prove \eqref{item:term3} we again distinguish
  between the two cases. Begin by assuming that $e_n=f_n$, then
  together with Theorem~\ref{thm:ols_convergence_strong} we get that
  \begin{align*}
    &\E\left(\left(\tfrac{1}{\abs{e_{n}}}(\beta_{n}-\hat{\beta}_{e_n})^{\top}\vX_{e_n}^{\top}\vX_{e_n}(\beta_{n}-\hat{\beta}_{e_n})\right)^{2k}\right)\\
    &\quad=\frac{1}{\abs{e_{n}}^{2k}}\E\left(\norm[\big]{\vX_{e_n}(\beta_{n}-\hat{\beta}_{e_n})}^{4k}\right)\leq\frac{C}{\abs{e_{n}}^{2k}}.
  \end{align*}
  Furthermore, assume $e_n\cap f_n=\varnothing$ then, using a spectral
  inequality together with Theorem~\ref{thm:ols_convergence_strong} we
  get that
    \begin{align*}
    &\E\left(\left(\tfrac{1}{\abs{e_{n}}}(\beta_{n}-\hat{\beta}_{f_n})^{\top}\vX_{e_n}^{\top}\vX_{e_n}(\beta_{n}-\hat{\beta}_{f_n})\right)^{2k}\right)\\
    &\quad\leq\E\left(\left(\lambda_{\max}\left(\tfrac{1}{\abs{e_n}}\vX_{e_n}^{\top}\vX_{e_n}\right)\norm[\big]{\beta_{n}-\hat{\beta}_{f_n}}^2\right)^{2k}\right)\\
    &\quad=\E\left(\lambda_{\max}\left(\tfrac{1}{\abs{e_n}}\vX_{e_n}^{\top}\vX_{e_n}\right)^{2k}\right)\E\left(\norm[\big]{\beta_{n}-\hat{\beta}_{f_n}}^{4k}\right)\leq\frac{C}{\abs{f_n}^{2k}},
  \end{align*}
  which completes the proof of \eqref{item:term3}.

  We can now combine these results to analyze the convergence
  properties of the test statistic $T^{\max,\mathcal{F}^1(\mathcal{E}_n)}_3$. As an intermediate
  step consider the following two statistics,
  \begin{equation}
    \label{eq:top_part}
    U_{e_n,f_n}\coloneqq \frac{1}{\abs{e_n}}\left(\vepsilon_{e_n}^{\top}\vepsilon_{e_n}+2\vepsilon_{e_n}^{\top}\vX_{e_n}(\beta_{n}-\hbeta_{f_n})+(\beta_{n}-\hbeta_{f_n})^{\top}\vX_{e_n}^{\top}\vX_{e_n}(\beta_{n}-\hbeta_{f_n})\right)
  \end{equation}
  and
  \begin{equation}
    \label{eq:bottom_part}
    V_{f_n}\coloneqq \frac{1}{\abs{f_n}}\left(\vepsilon_{f_n}^{\top}\vepsilon_{f_n}+2\vepsilon_{f_n}^{\top}\vX_{f_n}(\beta_{n}-\hbeta_{f_n})+(\beta_{n}-\hbeta_{f_n})^{\top}\vX_{f_n}^{\top}\vX_{f_n}(\beta_{n}-\hbeta_{f_n})\right).
  \end{equation}
  Using \eqref{item:term1}, \eqref{item:term2} and \eqref{item:term3}
  together with the Minkowski-inequality we get that
  \begin{align*}
    &\norm[\big]{U_{e_n,f_n}-\sigma_n^2}_{\operatorname{L}^{2k}}\\
    &\quad\leq\norm[\big]{\tfrac{1}{\abs{e_n}}\vepsilon_{e_n}^{\top}\vepsilon_{e_n}-\sigma_n^2}_{\operatorname{L}^{2k}}
      +\norm[\big]{\tfrac{2}{\abs{e_n}}\vepsilon_{e_n}^{\top}\vX_{e_n}(\beta_n-\hat{\beta}_{f_n})}_{\operatorname{L}^{2k}}
      +\norm[\big]{\tfrac{1}{\abs{e_n}}\vX_{e_n}^{\top}(\beta_n-\hat{\beta}_{f_n})}_{\operatorname{L}^{2k}}\\
    &\quad\leq
      \frac{C_1}{\sqrt{\abs{e_n}}}+\frac{C_2}{\sqrt{\abs{f_n}}}+\frac{C_3}{\abs{f_n}}\\
    &\quad\leq \frac{C}{\sqrt{r_n}}.
  \end{align*}
  Hence, with a union bound and Chebyshev's inequality we get for all
  constants $M>0$ that
  \begin{align*}
    \P\left(\left(\abs{\mathcal{E}_n}^{-\tfrac{1}{k}}r_n^{\tfrac{1}{2}}\right)\max_{e_n,f_n\in\mathcal{E}_n}\abs{U_{e_n,f_n}-\sigma_n^2}>M\right)
    &\leq\sum_{e_n,f_n\in\mathcal{E}_n}\P\left(\left(\abs{\mathcal{E}_n}^{-\tfrac{1}{k}}r_n^{\tfrac{1}{2}}\right)\abs{U_{e_n,f_n}-\sigma_n^2}>M\right)\\
    &\leq\sum_{e_n,f_n\in\mathcal{E}_n}\E\left(\abs{U_{e_n,f_n}-\sigma_n^2}^{2k}\right)\left(\abs{\mathcal{E}_n}^{-\tfrac{1}{k}}r_n^{\tfrac{1}{2}}\right)^{2k}M^{-2k}\\
    &\leq\frac{\abs{\mathcal{E}_n}(\abs{\mathcal{E}_n}-1)}{2}\frac{C^{2k}}{r_n^k}\abs{\mathcal{E}_n}^{-2}r_n^{k}M^{-2k}\\
    &=\landauO\left(\abs{\mathcal{E}_n}^2\right)\landauO\left(r_n^{-k}\right)\landauO\left(\abs{\mathcal{E}_n}^{-2}r_n^{k}\right)=\landauO(1).
  \end{align*}
  This implies,
  \begin{equation}
    \label{eq:unconv}
    \max_{e_n,f_n\in\mathcal{E}_n}U_{e_n,f_n}=\sigma_n^2+\landauOp\left(\abs{\mathcal{E}_n}^{\tfrac{1}{k}}r_n^{-\tfrac{1}{2}}\right)
    \quad\text{as } n\rightarrow\infty.
  \end{equation}
  Similarly we for the $V_{f_n}$ in \eqref{eq:bottom_part} we get
  \begin{align*}
    \norm[\big]{V_{f_n}-\sigma_n^2}_{\operatorname{L}^{2k}}
    &\leq\norm[\big]{\tfrac{1}{\abs{f_n}}\vepsilon_{f_n}^{\top}\vepsilon_{f_n}-\sigma_n^2}_{\operatorname{L}^{2k}}
      +\norm[\big]{\tfrac{1}{\abs{f_n}}\vepsilon_{f_n}^{\top}\vX_{f_n}(\beta_n-\hat{\beta}_{f_n})}_{\operatorname{L}^{2k}}
      +\norm[\big]{\tfrac{1}{\abs{f_n}}\vX_{f_n}^{\top}(\beta_n-\hat{\beta}_{f_n})}_{\operatorname{L}^{2k}}\\
    &\leq
      \frac{C_1}{\sqrt{\abs{f_n}}}+\frac{C_2}{\sqrt{\abs{f_n}}}+\frac{C_3}{\abs{f_n}}\\
    &\leq \frac{C}{\sqrt{r_n}}.
  \end{align*}
  And again, using the Chebyshev's inequality we get for all
  constants $M>0$ that
  \begin{align*}
    \P\left(\sqrt{r_n}\min_{f_n\in\mathcal{E}_n}\abs{V_{f_n}-\sigma_n^2}>M\right)
    &\leq\P\left(\sqrt{r_n}\abs{V_{f_n}-\sigma_n^2}>M\right)\\
    &\leq\E\left(\abs{V_{f_n}-\sigma_n^2}^{2k}\right)\left(\sqrt{r_n}\right)^{2k}M^{-2k}\\
    &\leq\frac{C^{2k}}{r_n^k}r_n^{k}M^{-2k}=\landauO(1),
  \end{align*}
  which implies that
  \begin{equation}
    \label{eq:vnconv}
    \min_{f_n\in\mathcal{E}_n}V_{f_n}=\sigma_n^2+\landauOp\left(r_n^{-\tfrac{1}{2}}\right)
    \quad\text{as } n\rightarrow\infty.
  \end{equation}
  Hence, \eqref{eq:unconv} and \eqref{eq:vnconv} together with
  Lemma~\ref{thm:convergence_rate_fractions} imply that
  \begin{equation*}
    \frac{\max_{e_n,f_n\in\mathcal{E}_n}U_{e_n,f_n}}{\min_{f_n\in\mathcal{E}_n}V_{f_n}}-1=\landauOp\left(\abs{\mathcal{E}_n}^{\tfrac{1}{k}}r_n^{-\tfrac{1}{2}}\right)
    \quad\text{as } n\rightarrow\infty.
  \end{equation*}
  Since for any $M\geq1$ it holds that
  \begin{equation*}
    \P\left(\abs[\big]{T^{\max,\mathcal{F}^1(\mathcal{E}_n)}_3}>M\right)
    \leq\P\left(\abs[\bigg]{\frac{\max_{e_n,f_n\in\mathcal{E}_n}U_{e_n,f_n}}{\min_{f_n\in\mathcal{E}_n}V_{f_n}}-1}>M\right),
  \end{equation*}
  we also have that
  \begin{equation*}
    T^{\max,\mathcal{F}^1(\mathcal{E}_n)}_3=\landauOp\left(\abs{\mathcal{E}_n}^{\tfrac{1}{k}}r_n^{-\tfrac{1}{2}}\right)
    \quad\text{as } n\rightarrow\infty.
  \end{equation*}
  This completes the proof of
  Proposition~\ref{thm:asymptotic_dist_h0_tot}.
\end{proof}

\begin{proof}[Proposition~\ref{thm:asymptotic_dist_ha}]
  We divide the proof into two parts. In the first part we prove
  \eqref{eq:upperboundpart} and then in the second part we prove
  \eqref{eq:lowerboundpart} using some results from the first
  part.
  
  \textbf{Part 1:}\\
  Similar to the proof of Proposition~\ref{thm:asymptotic_dist_h0_tot},
  the main idea in the proof is to use the representation of
  $T_{e_1,e_2}^3$ given in Lemma~\ref{thm:expansion_T}, analyze the
  convergence of all terms individually and finally conclude by
  combining the convergences.
  
  We begin by proving the following inequalities
  \begin{enumerate}[(a)]
  \item for $e_n\in\{f_n,g_n\}$: $\norm[\big]{\frac{1}{\abs{e_{n}}}\vepsilon_{e_n}^{\top}\vepsilon_{e_n}-\sigma_{e_n}^2}_{L^2}\leq \frac{C}{\sqrt{\abs{e_{n}}}}$\label{item:term1ha}
  \item
    $\norm[\big]{\frac{1}{\abs{f_{n}}}\vepsilon_{f_n}^{\top}\vX_{f_n}(\beta_{f_n}-\hat{\beta}_{g_n})}_{L^2}\leq\frac{C_1}{\sqrt{\abs{f_{n}}\abs{g_n}}}+\frac{C_2}{\sqrt{\abs{f_{n}}}}b_n$\label{item:term2ha}
  \item $\norm[\big]{\frac{1}{\abs{g_{n}}}\vepsilon_{g_n}^{\top}\vX_{g_n}(\beta_{g_n}-\hat{\beta}_{g_n})}_{L^2}\leq\frac{C}{\sqrt{\abs{g_{n}}}}$\label{item:term3ha}
  \item
    $\norm[\big]{\frac{1}{\abs{f_{n}}}(\beta_{f_n}-\hat{\beta}_{g_n})^{\top}\vX_{f_n}^{\top}\vX_{f_n}(\beta_{f_n}-\hat{\beta}_{g_n})}_{L^2}\leq\frac{C_1}{\abs{g_{n}}}+C_2b_n^2$\label{item:term4ha}
  \item $\norm[\big]{\frac{1}{\abs{g_{n}}}(\beta_{g_n}-\hat{\beta}_{g_n})^{\top}\vX_{g_n}^{\top}\vX_{g_n}(\beta_{n}-\hat{\beta}_{g_n})}_{L^2}\leq\frac{C}{\abs{g_{n}}}$\label{item:term5ha}
  \end{enumerate}
  
  Let $e_n\in\{f_n,g_n\}$, we first show \eqref{item:term1ha},
  \begin{align*}
    \E\left(\left(\frac{1}{\abs{e_n}}\vepsilon_{e_n}^{\top}\vepsilon_{e_n}-\sigma_{e_n}^2\right)^2\right)
    &=\E\left(\left(\frac{1}{\abs{e_{n}}}\sum_{i\in
      e_{n}}(\epsilon_i^2-\sigma_{e_n}^2)\right)^2\right)\\
    &=\frac{1}{\abs{e_{n}}^2}\sum_{i,j\in
      e_{n}}\E\left((\epsilon_i^2-\sigma_{e_n}^2)(\epsilon_j^2-\sigma_{e_n}^2)\right)\\
    &=\frac{1}{\abs{e_{n}}^2}\sum_{i\in
      e_{n}}\E\left(\left(\epsilon_i^2-\sigma_{e_n}^2\right)^2\right)\\
    &=\frac{1}{\abs{e_{n}}^2}\sum_{i\in
      e_{n}}\var\left(\epsilon_i^2\right)\\
    &=\frac{2\sigma_{e_n}^4}{\abs{e_{n}}}.
  \end{align*}
  Since the sequence $(\sigma_{e_n})_{n\in\N}$ is assumed to be
  convergent this proves \eqref{item:term1ha}.
  In order to prove \eqref{item:term2ha} we use the independence of
  $\vX$ and $\vepsilon$ and the Cauchy-Schwarz inequality to get
  \begin{align}
    \E\left(\left(\frac{1}{\abs{f_n}}\vepsilon_{f_n}^{\top}\vX_{f_n}(\beta_{f_n}-\beta_{g_n})\right)^2\right)
    &=\frac{1}{\abs{f_n}^2}\sum_{i_1,i_2\in
      f_n}\E\left(\epsilon_{i_1}\vX_{i_1}(\beta_{f_n}-\beta_{g_n})\epsilon_{i_2}\vX_{i_2}(\beta_{f_n}-\beta_{g_n})\right)\nonumber\\
    &=\frac{1}{\abs{f_n}^2}\sum_{i\in
      f_n}\E\left(\epsilon_{i}^2\right)\E\left(\left(\vX_{i}(\beta_{f_n}-\beta_{g_n})\right)^2\right)\nonumber\\
    &\leq\frac{\sigma_{f_n}^2}{\abs{f_n}}\E\left(\norm{\vX_i}_{\R^d}^2\right)\norm{\beta_{f_n}-\beta_{g_n}}_{\R^d}^2\nonumber\\
    &\leq\frac{C}{\abs{f_n}}\norm{\beta_{f_n}-\beta_{g_n}}_{\R^d}^2,\label{eq:triangleoneCC}
  \end{align}
  where in the last step we used that the sequence $(\sigma_{e_n})_{n\in\N}$ is convergent
  and $\E(\norm{\vX_i}^2)$ (for $i\in f_n$) is bounded. A similar inequality,
  additionally using Theorem~\ref{thm:ols_convergence_strong}, leads to
  \begin{align}
    \E\left(\left(\frac{1}{\abs{f_n}}\vepsilon_{f_n}^{\top}\vX_{f_n}(\beta_{g_n}-\hat{\beta}_{g_n})\right)^2\right)
    &=\frac{1}{\abs{f_n}^2}\sum_{i_1,i_2\in
      f_n}\E\left(\epsilon_{i_1}\vX_{i_1}(\beta_{g_n}-\hat{\beta}_{g_n})\epsilon_{i_2}\vX_{i_2}(\beta_{g_n}-\hat{\beta}_{g_n})\right)\nonumber\\
    &=\frac{1}{\abs{f_n}^2}\sum_{i\in
      f_n}\E\left(\epsilon_{i}^2\right)\E\left(\left(\vX_{i}(\beta_{g_n}-\hat{\beta}_{g_n})\right)^2\right)\nonumber\\
    &\leq\frac{\sigma_{f_n}^2}{\abs{f_n}}\E\left(\norm{\vX_i}_{\R^d}^2\right)\E\left(\norm{\beta_{g_n}-\hat{\beta}_{g_n}}_{\R^d}^2\right)\nonumber\\
    &\leq\frac{C}{\abs{f_n}\abs{g_n}}.\label{eq:triangletwoCC}
  \end{align}
  Finally, combining \eqref{eq:triangleoneCC} and
  \eqref{eq:triangletwoCC} with the Minkowski inequality proves
  \eqref{item:term2ha}.

  In order to prove \eqref{item:term3ha} define the projection matrix
  $\vP_{g_n}\coloneqq\vX_{g_n}(\vX_{g_n}^{\top}\vX_{g_n})\vX_{g_n}^{\top}$
  we get that
  \begin{align*}
    \E\left(\left(\tfrac{1}{\abs{g_{n}}}\abs[\big]{\vepsilon_{g_n}^{\top}\vX_{g_n}(\beta_{g_n}-\hat{\beta}_{g_n})}\right)^{2}\right)
    &=\E\left(\left(\tfrac{1}{\abs{g_{n}}}\abs[\big]{\vepsilon_{g_n}^{\top}\vP_{g_n}\vepsilon_{g_n}}\right)^{2}\right)\\
    &=\frac{1}{\abs{g_{n}}^{2}}\E\left(\norm[\big]{\vP_{g_n}\vepsilon_{g_n}}^{4}_{\R^{\abs{g_n}}}\right)\\
    &\leq\frac{C}{\abs{g_n}^{2}},
  \end{align*}
  where in the last step we used the same argument as in
  \eqref{eq:projection}.
  
  Next, we prove \eqref{item:term4ha}. This is again a straight
  forward estimate using Minkowski's inequality, $(a+b)^2\leq 2a^2+2b^2$ and
  Theorem~\ref{thm:ols_convergence_strong},
  \begin{align*}
    &\E\left(\left(\tfrac{1}{\abs{f_{n}}}(\beta_{f_n}-\hat{\beta}_{g_n})^{\top}\vX_{f_n}^{\top}\vX_{f_n}(\beta_{f_n}-\hat{\beta}_{g_n})\right)^2\right)\\
    &\quad\leq\E\left(\lambda_{\max}\left(\tfrac{1}{\abs{f_n}}\vX_{f_n}^{\top}\vX_{f_n}\right)^2\norm[\big]{\beta_{f_n}-\hat{\beta}_{g_n}}_{\R^d}^4\right)\\
    &\quad\leq\E\left(\lambda_{\max}\left(\tfrac{1}{\abs{f_n}}\vX_{f_n}^{\top}\vX_{f_n}\right)^2\right)\E\left(\left(2\norm[\big]{\beta_{f_n}-\beta_{g_n}}_{\R^d}^2+2\norm[\big]{\beta_{g_n}-\hat{\beta}_{g_n}}_{\R^d}^2\right)^2\right)\\
    &\quad\leq
      C_1\E\left(4\norm[\big]{\beta_{f_n}-\beta_{g_n}}_{\R^d}^4+4\norm[\big]{\beta_{f_n}-\beta_{g_n}}_{\R^d}^2\norm[\big]{\beta_{g_n}-\hat{\beta}_{g_n}}_{\R^d}^2+4\norm[\big]{\beta_{g_n}-\hat{\beta}_{g_n}}_{\R^d}^4\right)\\
    &\quad\leq C_1\left(4\norm[\big]{\beta_{f_n}-\beta_{g_n}}_{\R^d}^4+4\norm[\big]{\beta_{f_n}-\beta_{g_n}}_{\R^d}^2\frac{C_2}{\abs{g_n}}+4\frac{C_2^2}{\abs{g_n}^2}\right)\\
    &\quad= C_1\left(2\norm[\big]{\beta_{f_n}-\beta_{g_n}}_{\R^d}^2+\frac{2C_2}{\abs{g_n}}\right)^2.
  \end{align*}

  Finally, \eqref{item:term5ha} is an immediate consequence of Theorem~\ref{thm:ols_convergence_strong},
  \begin{equation*}
    \E\left(\left(\tfrac{1}{\abs{g_{n}}}(\beta_{g_n}-\hat{\beta}_{g_n})^{\top}\vX_{g_n}^{\top}\vX_{g_n}(\beta_{g_n}-\hat{\beta}_{g_n})\right)^{2}\right)
    =\frac{1}{\abs{g_{n}}^{2}}\E\left(\norm[\big]{\vX_{g_n}(\beta_{g_n}-\hat{\beta}_{g_n})}^{4}\right)
    \leq\frac{C}{\abs{g_{n}}^{2}}.
  \end{equation*}

  As in the proof of Proposition~\ref{thm:asymptotic_dist_h0_tot} consider
  the following two statistics,
  \begin{equation}
    \label{eq:top_partha}
    U_{f_n,g_n}\coloneqq \frac{1}{\abs{f_n}}\left(\vepsilon_{f_n}^{\top}\vepsilon_{f_n}+2\vepsilon_{f_n}^{\top}\vX_{f_n}(\beta_{f_n}-\hbeta_{g_n})+(\beta_{f_n}-\hbeta_{g_n})^{\top}\vX_{f_n}^{\top}\vX_{f_n}(\beta_{f_n}-\hbeta_{g_n})\right)
  \end{equation}
  and
  \begin{equation}
    \label{eq:bottom_partha}
    V_{g_n}\coloneqq \frac{1}{\abs{g_n}}\left(\vepsilon_{g_n}^{\top}\vepsilon_{g_n}+2\vepsilon_{g_n}^{\top}\vX_{g_n}(\beta_{g_n}-\hbeta_{g_n})+(\beta_{g_n}-\hbeta_{g_n})^{\top}\vX_{g_n}^{\top}\vX_{g_n}(\beta_{g_n}-\hbeta_{g_n})\right).
  \end{equation}
  Using \eqref{item:term1ha}, \eqref{item:term2ha} and \eqref{item:term4ha}
  together with the Minkowski-inequality we get that
  \begin{align}
    \norm[\big]{U_{f_n,g_n}-\sigma_{g_n}^2}_{\operatorname{L}^{2}}&\leq\norm[\big]{U_{f_n,g_n}-\sigma_{f_n}^2}_{\operatorname{L}^{2}}+\abs[\big]{\sigma_{f_n}^2-\sigma_{g_n}^2}\nonumber\\
    &\leq\norm[\big]{\tfrac{1}{\abs{f_n}}\vepsilon_{f_n}^{\top}\vepsilon_{f_n}-\sigma_{f_n}^2}_{\operatorname{L}^{2}}
      +\norm[\big]{\tfrac{2}{\abs{f_n}}\vepsilon_{f_n}^{\top}\vX_{f_n}(\beta_{f_n}-\hat{\beta}_{g_n})}_{\operatorname{L}^{2}}\nonumber\\
    &\qquad+\norm[\big]{\tfrac{1}{\abs{f_n}}(\beta_{f_n}-\hat{\beta}_{g_n})\vX_{f_n}^{\top}\vX_{f_n}(\beta_{f_n}-\hat{\beta}_{g_n})}_{\operatorname{L}^{2}}+a_n\nonumber\\
    &\leq
      \frac{C_1}{\sqrt{\abs{f_n}}}+\frac{C_2}{\sqrt{\abs{f_{n}}\abs{g_n}}}+\frac{C_3}{\sqrt{\abs{f_{n}}}}b_n+\frac{C_4}{\abs{g_{n}}}+C_5b_n^2+a_n\nonumber\\
    &\leq C\left(\frac{1}{\sqrt{r_n}}+b_n^2+a_n\right).\label{eq:upartasr}
  \end{align}
  Similarly, using \eqref{item:term1ha}, \eqref{item:term3ha} and
  \eqref{item:term5ha} we get for $V_{g_n}$ in \eqref{eq:bottom_partha} that
  \begin{align}
    \norm[\big]{V_{g_n}-\sigma_{g_n}^2}_{\operatorname{L}^{2}}
    &\leq\norm[\big]{\tfrac{1}{\abs{g_n}}\vepsilon_{g_n}^{\top}\vepsilon_{g_n}-\sigma_{g_n}^2}_{\operatorname{L}^{2}}
      +\norm[\big]{\tfrac{2}{\abs{g_n}}\vepsilon_{g_n}^{\top}\vX_{g_n}(\beta_{g_n}-\hat{\beta}_{g_n})}_{\operatorname{L}^{2}}
      +\norm[\big]{\tfrac{1}{\abs{g_n}}\vX_{g_n}^{\top}(\beta_{g_n}-\hat{\beta}_{g_n})}_{\operatorname{L}^{2}}\nonumber\\
    &\leq
      \frac{C_1}{\sqrt{\abs{g_n}}}+\frac{C_2}{\sqrt{\abs{g_n}}}+\frac{C_3}{\abs{g_n}}\nonumber\\
    &\leq \frac{C}{\sqrt{r_n}}.\label{eq:vpartasr}
  \end{align}
  Since $L^2$ convergence implies convergence in probability
  \eqref{eq:upartasr} and \eqref{eq:vpartasr} imply that
  \begin{equation*}
    U_{f_n,g_n}=\sigma_{g_n}^2+\landauOp\left(\frac{1}{\sqrt{r_n}}+b_n^2+a_n\right)
    \quad\text{and}\quad
    V_{g_n}=\sigma_{g_n}^2+\landauOp\left(\frac{1}{\sqrt{r_n}}\right)
    \quad\text{as } n\rightarrow\infty.
  \end{equation*}
  Hence, using Lemma~\ref{thm:convergence_rate_fractions} it holds that 
  \begin{equation*}
    T_{f_n,g_n}^3(\sresid)=\frac{U_{f_n,g_n}}{V_{g_n}}-1=\landauOp\left(\frac{1}{\sqrt{r_n}}+b_n^2+a_n\right)=\landauOp\left(a_n\right)+\landauOp\left(b_n^2\right)+\landauOp\left(\frac{1}{\sqrt{r_n}}\right)
  \end{equation*}
  as $n\rightarrow\infty$. This completes the first part of the proof.

  \textbf{Part 2:}\\
  Next, we prove \eqref{eq:lowerboundpart}. Since $\sigma_{g_n}^2$
  converges to a positive constant and $\omega_n$ converges to zero,
  there exists $c,C,\delta>0$ and $n_0\in\N$ such that for all
  $n\in\{n_0,n_0+1,\dots\}$ it holds
  that $$0<c+\delta<\sigma^2_{g_n}<C-\delta\quad\text{and}\quad \omega_n<1.$$ Let
  $n\in\{n_0,n_0+1,\dots\}$, then on the event
  $\{\abs{V_{g_n}-\sigma_{g_n}^2}\leq\delta\omega_n\}$ it holds that
  \begin{equation*}
    \abs{T_{f_n,g_n}^3}
    =\frac{\abs{U_{f_n,g_n}-V_{g_n}}}{\abs{V_{g_n}}}
    \geq\frac{\abs{U_{f_n,g_n}-\sigma^2_{g_n}}}{\abs{V_{g_n}}}-\frac{\abs{V_{g_n}-\sigma^2_{g_n}}}{\abs{V_{g_n}}}
    \geq\frac{\abs{U_{f_n,g_n}-\sigma^2_{g_n}}}{C}-\frac{\delta\omega_n}{c},
  \end{equation*}
  which in particular implies for all $t>0$ that
  \begin{equation*}
    \P\left(\tfrac{1}{\omega_n}\abs{T_{f_n,g_n}^3}\leq t\right)
    \leq\P\left(\tfrac{1}{\omega_n}\abs{U_{f_n,g_n}-\sigma_{g_n}^2}\leq
      Ct+\tfrac{\delta C}{c}\right)
    +\P\left(\abs{V_{g_n}-\sigma_{g_n}^2}>\delta\omega_n\right).
  \end{equation*}
  Therefore, in order to prove \eqref{eq:lowerboundpart} it is
  sufficient to show that
  $\abs{V_{g_n}-\sigma_{g_n}^2}=\landauop(\omega_n)$ and that for all
  $t\in\R$ it holds that
  \begin{equation}
    \label{eq:missingpartU}
    \lim_{n\rightarrow\infty}\P\left(\tfrac{1}{\omega_n}\abs{U_{f_n,g_n}-\sigma_{g_n}^2}\leq
      t\right)=0.
  \end{equation}
  However, by \eqref{eq:vpartasr} and the fact that $L^2$ convergence
  implies convergence in probability we already have shown that
  $\abs{V_{g_n}-\sigma^2_{g_n}}=\landauop(\omega_n)$. It thus
  remains to prove \eqref{eq:missingpartU}. To simplify the notation,
  we make the following definitions
  \begin{equation*}
    \begin{aligned}[c]
      &A_n\coloneqq
      \tfrac{1}{\abs{f_n}}\vepsilon^{\top}_{f_n}\vepsilon_{f_n}-\sigma_{g_n}^2\\
      &\tilde{A}_n\coloneqq
      \tfrac{1}{\abs{f_n}}\vepsilon^{\top}_{f_n}\vepsilon_{f_n}-\sigma_{f_n}^2
    \end{aligned}
    \qquad\qquad
    \begin{aligned}[c]
      &B_n\coloneqq
      \tfrac{2}{\abs{f_n}}\vepsilon^{\top}_{f_n}\vX_{f_n}(\beta_{f_n}-\hat{\beta}_{g_n})\\
      &C_n\coloneqq \tfrac{1}{\abs{f_n}}(\beta_{f_n}-\hat{\beta}_{g_n})^{\top}\vX_{f_n}^{\top}\vX_{f_n}(\beta_{f_n}-\hat{\beta}_{g_n}).
    \end{aligned}
  \end{equation*}
  For the proof we require two more intermediate results. Denote by
  $Z_n^A$ and $Z_n^C$ are sequences satisfying that
  $\tfrac{1}{a_n}Z_n^A\overset{\P}{\rightarrow}0$ and
  $\tfrac{1}{b_n^2}Z_n^C\overset{\P}{\rightarrow}0$ as
  $n\rightarrow\infty$. We want to show that if $\omega_n=\landauo(a_n)$
  it holds that
  \begin{equation}
    \label{eq:AnCnconvergence1}
    \lim_{n\rightarrow\infty}\P\left(\abs{A_n}\leq
      Z_n^A\right)=0
  \end{equation}
  and if $\omega_n=\landauo(b_n^2)$
  it holds that
  \begin{equation}
    \label{eq:AnCnconvergence2}
    \lim_{n\rightarrow\infty}\P\left(\abs{C_n}\leq Z_n^C\right)=0.
  \end{equation}
  Define
  $\tilde{b}_n\coloneqq\E\left(\norm{X_{f_n}(\beta_{f_n}-\beta_{g_n})}_{\R^{\abs{f_n}}}^2\right)$,
  the proof of these two results relies on the Paley-Zygmund
  inequality \citep[e.g.][Section 8.3]{weber2009} and the following four inequalities,
  \begin{multicols}{4}\raggedcolumns
    \begin{enumerate}[(1)]
    \item $\abs{\E(A_n)}=a_n$,\label{item:est1}
    \item $\E(\tilde{A}_n^2)\leq \frac{c}{r_n}$,\label{item:est2}
    \item $\E\left(C_n\right)\geq\tilde{b}_n^2-\frac{cb_n}{\sqrt{r_n}}$,\label{item:est3}
    \item $\E(C_n^2)\leq\tilde{b}_n^2+\frac{c}{\sqrt{r_n}}$.\label{item:est4}
    \end{enumerate}
  \end{multicols}
  
  We begin by proving \eqref{eq:AnCnconvergence1}. Hence, assume that
  $\omega_n=\landauo(a_n)$, for any $\theta\in[0,1]$ we can use
  \eqref{item:est1} together with Jensen's inequality and the
  Paley-Zygmund inequality to get that
  \begin{equation*}
    \P\left(\abs{A_n}>a_n\theta\right)
    \geq\P\left(\abs{A_n}>\E(\abs{A_n})\cdot\theta\right)
    \geq\frac{\E(\abs{A_n})^2}{\E(A_n^2)}(1-\theta)^2.
  \end{equation*}
  Applying Jensen's inequality once more together
  with \eqref{item:est1} and \eqref{item:est2} leads to
  \begin{align*}
    \P\left(\abs{A_n}>a_n\theta\right)
    &\geq\frac{\E(A_n)^2}{\E((\tilde{A}_n+\sigma_{f_n}^2-\sigma_{g_n}^2)^2)}(1-\theta)^2\\
    &\geq\frac{a_n^2}{\frac{c}{r_n}+a_n^2}(1-\theta)^2\\
    &=\frac{1}{1+\frac{c}{a_n^2r_n}}(1-\theta)^2.
  \end{align*}
  This implies for all $\theta\in(0,1]$ that
  \begin{align*}
    \P\left(\abs{A_n}\leq Z_n^A\right)
    &=\P\left(\abs{A_n}\leq Z_n^A,
      \tfrac{1}{a_n}Z_n^A\leq\theta\right)+\P\left(\abs{A_n}\leq Z_n^A,
      \tfrac{1}{a_n}Z_n^A>\theta\right)\\
    &\leq\P\left(\abs{A_n}\leq
      a_n\theta\right)+\P\left(\tfrac{1}{a_n}Z_n^A>\theta\right)\\
    &\leq 1-\frac{1}{1+\frac{c}{a_n^2r_n}}(1-\theta)^2+\P\left(\tfrac{1}{a_n}Z_n^A>\theta\right).\\
  \end{align*}
  Furthermore, since we assumed that
  $\tfrac{1}{a_n}Z_n^A\overset{\P}{\rightarrow}0$ and that $\omega_n=\landauo(a_n)$ this in particular implies for all
  $\theta\in(0,1]$ that
  \begin{equation*}
    \lim_{n\rightarrow\infty}\P\left(\abs{A_n}\leq Z_n^A\right)\leq 1-(1-\theta)^2,
  \end{equation*}
  and since $\theta$ can be chosen independently of $n$ we have proved that
  \begin{equation}
    \label{eq:convAndist}
    \lim_{n\rightarrow\infty}\P\left(\abs{A_n}\leq Z_n^A\right)=0.
  \end{equation}
  
  Next, we use a similar reasoning to prove
  \eqref{eq:AnCnconvergence2}. Hence, assume that
  $\omega_n=\landauo(b_n^2)$, for any $\theta\in[0,1]$ we can
  use \eqref{item:est3} together with the Paley-Zygmund inequality to get that
  \begin{equation*}
    \P\left(C_n>\left(\tilde{b}_n^2-c\tfrac{b_n}{\sqrt{r_n}}\right)\theta\right)
    \geq\P\left(C_n>\E(C_n)\cdot\theta\right)
    \geq\frac{\E(C_n)^2}{\E(C_n^2)}(1-\theta)^2.
  \end{equation*}
  Making use of \eqref{item:est3} and \eqref{item:est4} leads to
  \begin{align*}
    \P\left(C_n>\left(\tilde{b}_n^2-c\tfrac{b_n}{\sqrt{r_n}}\right)\theta\right)
    &\geq\frac{\left(\tilde{b}_n^2-c\tfrac{b_n}{\sqrt{r_n}}\right)^2}{\left(\tilde{b}_n^2+\tfrac{c}{\sqrt{r_n}}\right)^2}(1-\theta)^2\\
    &=\frac{\tilde{b}_n^4-2c\tilde{b}_n^2\tfrac{b_n}{\sqrt{r_n}}+c^2\tfrac{b_n^2}{r_n}}{\tilde{b}_n^4+2c\tfrac{\tilde{b}_n^2}{\sqrt{r_n}}+\tfrac{c^2}{r_n}}(1-\theta)^2\\
    &=\frac{1-2c\tilde{b}_n^{-2}\tfrac{b_n}{\sqrt{r_n}}+c^2\tfrac{b_n^2}{\tilde{b}_n^4r_n}}{1+2\tfrac{c}{\tilde{b}_n^2\sqrt{r_n}}+\tfrac{c^2}{\tilde{b}_n^4r_n}}(1-\theta)^2\\
    &=\frac{1-c_1\tfrac{1}{b_n\sqrt{r_n}}+c_2\tfrac{1}{b_n^2r_n}}{1+c_3\tfrac{1}{b_n^2\sqrt{r_n}}+c_4\tfrac{1}{b_n^4r_n}}(1-\theta)^2,
  \end{align*}
  where in the last step we used that there exist $c,C>0$ such that
  $c\cdot b_n\leq\tilde{b}_n\leq C\cdot b_n$.
  This implies for all $\theta\in(0,1]$ that
  \begin{align*}
    \P\left(C_n\leq Z_n^C\right)
    &=\P\left(C_n\leq Z_n^C,
      \left(\tilde{b}_n^2-c\tfrac{b_n}{\sqrt{r_n}}\right)^{-1}Z_n^C\leq\theta\right)+\P\left(C_n\leq Z_n^C,
      \left(\tilde{b}_n^2-c\tfrac{b_n}{\sqrt{r_n}}\right)^{-1}Z_n^C>\theta\right)\\
    &\leq\P\left(C_n\leq
      \left(\tilde{b}_n^2-c\tfrac{b_n}{\sqrt{r_n}}\right)\theta\right)+\P\left(\left(\tilde{b}_n^2-c\tfrac{b_n}{\sqrt{r_n}}\right)^{-1}Z_n^C>\theta\right)\\
    &\leq 1-\frac{1-c_1\tfrac{1}{b_n\sqrt{r_n}}+c_2\tfrac{1}{b_n^2r_n}}{1+c_3\tfrac{1}{b_n^2\sqrt{r_n}}+c_4\tfrac{1}{b_n^4r_n}}(1-\theta)^2+\P\left(\left(\tilde{b}_n^2-c\tfrac{b_n}{\sqrt{r_n}}\right)^{-1}Z_n^C>\theta\right).\\
  \end{align*}
  Furthermore, since we assumed that
  $\tfrac{1}{b_n^2}Z_n^C\overset{\P}{\rightarrow}0$ and that
  $\omega_n=\landauo(b_n^2)$ this in particular implies for all
  $\theta\in(0,1]$ that
  \begin{equation*}
    \lim_{n\rightarrow\infty}\P\left(C_n\leq Z_n^C\right)\leq 1-(1-\theta)^2,
  \end{equation*}
  and since $\theta$ can be chosen independently of $n$ we have proved that
  \begin{equation}
    \label{eq:convCndist}
    \lim_{n\rightarrow\infty}\P\left(C_n\leq Z_n^C\right)=0.
  \end{equation}

  Finally, we are ready to prove \eqref{eq:missingpartU}. We begin by
  observing that,
  \begin{align}
    \P\left(\tfrac{1}{\omega_n}\abs{U_{f_n,g_n}-\sigma_{g_n}^2}\leq
    t\right)
    &=\P\left(\abs{A_n+B_n+C_n}\leq t\omega_n\right)\nonumber\\
    &\leq\P\left(-t\omega_n-\abs{B_n}\leq
      A_n+C_n\leq t\omega_n+\abs{B_n}\right).\label{eq:firstestimate}
  \end{align}
  Since $\omega_n$ by definition has a slower (or equal)
  convergence rate as $\frac{1}{\sqrt{r_n}}$, we can interpret it as
  the fastest rate at which the alternatives can converge without
  loosing detectability. Keeping this intuition in mind, we
  distinguish the following 3 cases,
  \begin{enumerate}[(\text{case} 1)]
  \item variance and regression shifts are detectable, i.e. $\omega_n=\landauo(a_n)$ and
    $\omega_n=\landauo(b_n^2)$,\label{item:case1}
  \item only variance shifts are detectable, i.e. $\omega_n=\landauo(a_n)$ and
    $b_n^2=\landauO(\omega_n)$,\label{item:case2}
  \item only regression shifts are detectable, i.e. $\omega_n=\landauo(b_n^2)$ and
    $a_n=\landauO(\omega_n)$.\label{item:case3}
  \end{enumerate}

  \textbullet\, (case 1): Define
  $Z_n\coloneqq t\omega_n+\abs{B_n}$, then using
  $\omega_n=\landauo(a_n)$ and
  $\omega_n=\landauo(b_n^2)$ together with
  \eqref{item:term2ha} it holds that
  $\tfrac{1}{a_n}Z_n\overset{\P}{\rightarrow}0$ and
  $\tfrac{1}{b_n^2}Z_n\overset{\P}{\rightarrow}0$.  Hence,
  \eqref{eq:firstestimate} together with a union bound and
  \eqref{eq:convAndist} and \eqref{eq:convCndist} leads to
  \begin{align*}
    \lim_{n\rightarrow\infty}\P\left(\tfrac{1}{\omega_n}\abs{U_{f_n,g_n}-\sigma_{g_n}^2}\leq
    t\right)
    &\leq\lim_{n\rightarrow\infty}\P\left(-Z_n\leq
      A_n+C_n\leq Z_n\right)\\
    &\leq\lim_{n\rightarrow\infty}\P\left(\abs{A_n}\leq Z_n\right)+\lim_{n\rightarrow\infty}\P\left(C_n\leq Z_n\right)=0.
  \end{align*}

  \textbullet\, (case 2): Define $Z_n\coloneqq t\omega_n+\abs{B_n}+C_n$, then using
  $\omega_n=\landauo(a_n)$ and
  $b_n^2=\landauO(\omega_n)$ together with
  \eqref{item:term2ha} and \eqref{item:term4ha} it holds that
  $\tfrac{1}{a_n}Z_n\overset{\P}{\rightarrow}0$.  Hence,
  \eqref{eq:firstestimate} together with \eqref{eq:convAndist} leads
  to
  \begin{align*}
    \lim_{n\rightarrow\infty}\P\left(\tfrac{1}{\omega_n}\abs{U_{f_n,g_n}-\sigma_{g_n}^2}\leq
    t\right)
    &\leq\lim_{n\rightarrow\infty}\P\left(-Z_n\leq
      A_n\leq Z_n\right)\\
    &\leq\lim_{n\rightarrow\infty}\P\left(\abs{A_n}\leq Z_n\right)=0.
  \end{align*}

  \textbullet\, (case 3): Define $Z_n\coloneqq t\omega_n+\abs{B_n}+\abs{A_n}$, then using
  $\omega_n=\landauo(b_n^2)$ and
  $a_n=\landauO(\omega_n)$ together with
  \eqref{item:term1ha} and \eqref{item:term4ha} it holds that
  $\tfrac{1}{b_n^2}Z_n\overset{\P}{\rightarrow}0$.  Hence,
  \eqref{eq:firstestimate} together with \eqref{eq:convCndist} leads
  to
  \begin{align*}
    \lim_{n\rightarrow\infty}\P\left(\tfrac{1}{\omega_n}\abs{U_{f_n,g_n}-\sigma_{g_n}^2}\leq
    t\right)
    &\leq\lim_{n\rightarrow\infty}\P\left(-Z_n\leq
      C_n\leq Z_n\right)\\
    &\leq\lim_{n\rightarrow\infty}\P\left(C_n\leq Z_n\right)=0.
  \end{align*}
  Thus we have proved \eqref{eq:missingpartU}, which completes the
  proof of Proposition~\ref{thm:asymptotic_dist_ha}.  
\end{proof}

\subsection{Theorem~\ref{thm:consistent_test_bonferroni}}\label{sec:proof_consistent2}

The proof of Theorem~\ref{thm:consistent_test_bonferroni} is very
similar to the proof of
Theorem~\ref{thm:consistent_test}. Essentially, we use the same
methods to show that the test statistic
$T^{\max,\mathcal{F}^1(\mathcal{E}_n)}_1$ (see
\eqref{eq:teststat_beta}) is capable of detecting changes in the
regression coefficients with a rate of $b_n$ and that
$T^{\max,\mathcal{F}^1(\mathcal{E}_n)}_2$ (see
\eqref{eq:teststat_sigma}) is capable of detecting changes in the
residual variance with a rate of $a_n$. Combining both results using a
Bonferroni adjustment preserves these rates. In order to not repeat
all the details we will therefore often refer to the proof in
Section~\ref{sec:proof_consistent}.

\begin{proof}[Theorem~\ref{thm:consistent_test_bonferroni}]
  Using the same arguments and notation as in the proof of Theorem~\ref{thm:consistent_test} we can use
  Proposition~\ref{thm:asymptotic_dist_h0_tot2} and Proposition~\ref{thm:asymptotic_dist_ha2} to show that
  \begin{itemize}
  \item for $\omega_n=\landauo(b_n)$ it holds that
    \begin{equation}
      \label{eq:teststat1}
      \lim_{n\rightarrow\infty}\P\left(\varphi^*_{T^{\max,\mathcal{F}^1(\mathcal{E}_n)}_1}(\vY_n,\vX_n)=1\right)=1
    \end{equation}
  \item and for $\omega_n=\landauo(a_n)$ it holds that
    \begin{equation}
      \label{eq:teststat2}
      \lim_{n\rightarrow\infty}\P\left(\varphi^*_{T^{\max,\mathcal{F}^1(\mathcal{E}_n)}_2}(\vY_n,\vX_n)=1\right)=1.
    \end{equation}    
  \end{itemize}
  Combining these test using a Bonferroni adjustment preserves the
  consistency properties of each of the individual tests, which
  completes the proof of Theorem~\ref{thm:consistent_test_bonferroni}.
\end{proof}

\subsubsection{Intermediate results}

\begin{lemma}[representation of $T_{e_1,e_2}^1$ and $T_{e_1,e_2}^2$ for true change points]
  \label{thm:expansion_T2}
  Let $e_1,e_2\in\mathcal{E}_n(\CPt_n^*)$ then it holds that
  \begin{align*}
    T_{e_1,e_2}^1(\sresid_n)&=\frac{\norm[\big]{\hat{\beta}_{e_1}-\hat{\beta}_{e_2}}_{2}}{\norm{\sresid_n}_{2}}\quad\text{and}\\
    T_{e_1,e_2}^2(\sresid_n)&=\frac{\frac{1}{\abs{e_1}}\left(\vepsilon_{e_1}^{\top}\vepsilon_{e_1}+2\vepsilon_{e_1}^{\top}\vX_{e_1}(\beta_{e_1}-\hbeta_{e_1})+(\beta_{e_1}-\hbeta_{e_1})^{\top}\vX_{e_1}^{\top}\vX_{e_1}(\beta_{e_1}-\hbeta_{e_1})\right)}{\frac{1}{\abs{e_2}}\left(\vepsilon_{e_2}^{\top}\vepsilon_{e_2}+2\vepsilon_{e_2}^{\top}\vX_{e_2}(\beta_{e_2}-\hbeta_{e_2})+(\beta_{e_2}-\hbeta_{e_2})^{\top}\vX_{e_2}^{\top}\vX_{e_2}(\beta_{e_2}-\hbeta_{e_2})\right)}-1.
  \end{align*}
\end{lemma}
The proof of this result is immediate using the same transformation as
in the proof of Lemma~\ref{thm:expansion_T}.

The following theorem gives the asymptotic distribution of
our test statistics under the null hypothesis $\HO$.

\begin{proposition}[asymptotic distribution under $\HO$]
  \label{thm:asymptotic_dist_h0_tot2}
  Let $(\vY_{n},\vX_{n})_{n\in\N}$ satisfy
  Assumption~\ref{assumption:asymptotic_cpmodel} and for all $n\in\N$
  satisfy $\P^{(\vY_{n},\vX_{n})}\in\HO^n$, let $\sresid_n$ be the
  scaled residuals defined in \eqref{eq:scaled_resid} corresponding to
  $(\vY_n,\vX_n)$, let
  $\mathcal{E}_n\subseteq\mathcal{P}(\{1,\dots,n\})$ be a sequence of
  collections of pairwise disjoint environments satisfying conditions
  \hyperref[item:c1]{(C1)} and \hyperref[item:c1]{(C3,k)}. Then, it
  holds for all $e_n,f_n\in\mathcal{E}_n$ that
  \begin{equation*}
    T_{e_n,f_n}^1(\sresid_n)=\landauOp\left(\frac{1}{\sqrt{r_n}}\right)
    \quad\text{and}\quad
    T^{\max,\mathcal{F}^1(\mathcal{E}_n)}_1(\sresid_n)
    =\landauOp\left(\frac{\abs{\mathcal{E}_n}^{\frac{1}{k}}}{\sqrt{r_n}}\right),
    \quad\text{as }n\rightarrow\infty,
  \end{equation*}
  as well as
  \begin{equation*}
    T_{e_n,f_n}^2(\sresid_n)=\landauOp\left(\frac{1}{\sqrt{r_n}}\right)
    \quad\text{and}\quad
    T^{\max,\mathcal{F}^1(\mathcal{E}_n)}_2(\sresid_n)
    =\landauOp\left(\frac{\abs{\mathcal{E}_n}^{\frac{1}{k}}}{\sqrt{r_n}}\right),
    \quad\text{as }n\rightarrow\infty.
  \end{equation*}
\end{proposition}
A proof of this result is given in Appendix~\ref{sec:intermed_proofs2}.

Next, we give the corresponding theorem for the asymptotic distribution of
our test statistics under the alternative hypothesis $\HA$.

\begin{proposition}[asymptotic distribution under $\HA$]
  \label{thm:asymptotic_dist_ha2}
  Let $(\vY_{n},\vX_{n})_{n\in\N}$ satisfy
  Assumption~\ref{assumption:asymptotic_cpmodel} and for all $n\in\N$
  satisfy $\P^{(\vY_{n},\vX_{n})}\in\HA^{n}(a_n,b_n)$, let $\sresid_n$
  be the scaled residuals defined in \eqref{eq:scaled_resid}
  corresponding to $(\vY_n,\vX_n)$. Additionally, let
  $i,j\in\{1,\dots,L+1\}$ such that
  $a_n=\abs{\sigma^2_{e_i(\CPt^*)}-\sigma^2_{e_j(\CPt^*)}}$ and
  $b_n=\norm{\beta_{e_i(\CPt^*)}-\beta_{e_j(\CPt^*)}}_2$. Then,
  assume that $f_n\subseteq e_{i}(\CPt^*_n)$ and
  $g_n\subseteq e_{j}(\CPt^*_n)$ are sequences satisfying that for
  $e_n\in\{f_n,g_n\}$ the sequences $(\sigma_{e_n}^2)_{n\in\N}$ and
  $(\beta_{e_n})_{n\in\N}$ are convergent and the limit of
  $\sigma_{e_n}^2$ is strictly positive and the sequence
  $(\{f_n,g_n\})_{n\in\N}$ satisfies assumptions
  \hyperref[item:c1]{(C1)} and \hyperref[item:c3]{(C3,k)}.  Moreover,
  let $(\omega_n)_{n\in\N}$ be a sequence which satisfies
  $\tfrac{1}{\sqrt{r_n}}=\landauO(\omega_n)$ then if $\omega_n=\landauo(b_n)$
  it holds for all $t\geq 0$ that
  \begin{equation}
    \label{eq:lowerboundpart2a}
    \lim_{n\rightarrow\infty}\P\left(\tfrac{1}{\omega_n}\abs[\big]{T_{f_n,g_n}^1(\sresid_n)}\leq t\right)=0.
  \end{equation}
  and if $\omega_n=\landauo(a_n)$ it holds for all $t\geq 0$ that
  \begin{equation}
    \label{eq:lowerboundpart2b}
    \lim_{n\rightarrow\infty}\P\left(\tfrac{1}{\omega_n}\abs[\big]{T_{f_n,g_n}^2(\sresid_n)}\leq t\right)=0.
  \end{equation}
\end{proposition}
A proof of this result is given in Appendix~\ref{sec:intermed_proofs2}.

\subsubsection{Proofs of intermediate results}\label{sec:intermed_proofs2}

\begin{proof}[Proposition~\ref{thm:asymptotic_dist_h0_tot2}]
  We begin with the results for the test statistic
  $T^2$. Using the notation from the proof of
  Proposition~\ref{thm:asymptotic_dist_h0_tot} and Lemma~\ref{thm:expansion_T2} it holds that
  \begin{equation*}
    T_{f_n,g_n}^2=\frac{V_{f_n}}{V_{g_n}}-1
    \quad\text{and}\quad
    \abs[\big]{T^{\max,\mathcal{F}^1(\mathcal{E}_n)}_2}\leq\abs[\bigg]{\frac{\max_{f_n\in\mathcal{E}_n}V_{f_n}}{\min_{g_n\in\mathcal{E}_n}V_{g_n}}-1}.
  \end{equation*}
  Moreover, similar computations show that
  \begin{align*}
    V_{f_n}=\sigma_n^2+\landauOp\left(r_n^{-\frac{1}{2}}\right),
    \quad
    \min_{f_n\in\mathcal{E}_n}V_{f_n}=\sigma_n^2+\landauOp\left(r_n^{-\frac{1}{2}}\right)
    \quad\text{and}\quad
    \max_{f_n\in\mathcal{E}_n}V_{f_n}=\sigma_n^2+\landauOp\left(r_n^{-\frac{1}{2}}\abs{\mathcal{E}_n}^{\frac{1}{k}}\right).
  \end{align*}
  Applying Lemma~\ref{thm:convergence_rate_fractions} proves the
  desired results.

  Next, we consider the test statistic $T^1$. Using the expansion from
  Lemma~\ref{thm:expansion_T2} together with the convergence of the
  OLS estimator (see Theorem~\ref{thm:ols_convergence_strong}) it
  holds that
  \begin{align*}
    \E\left(\left(T^1_{f_n,g_n}\right)^{2k}\right)
    &\leq\E\left(\tfrac{1}{\norm{\sresid_n}_2^{2k}}\left(\norm{\hat{\beta}_{f_n}-\beta_n}_2+\norm{\hat{\beta}_{g_n}-\beta_n}_2\right)^{2k}\right)\\
    &\leq\E\left(C_1\norm{\hat{\beta}_{f_n}-\beta_n}_2^{2k}+C_1\norm{\hat{\beta}_{g_n}-\beta_n}_2^{2k}\right)\\
    &\leq \frac{C_2}{\abs{f_n}^k}+\frac{C_2}{\abs{g_n}^k}\leq \frac{C_3}{r_n^k}.
  \end{align*}
  As in the proof of Proposition~\ref{thm:asymptotic_dist_h0_tot} we
  can now apply Chebyshev's inequality together with a union bound to also
  get that
  \begin{equation*}
    \abs{T^{\max,\mathcal{F}^1(\mathcal{E}_n)}_1}=\landauOp\left(\abs{\mathcal{E}_n}^{\frac{1}{k}}r_n^{-\frac{1}{2}}\right)\quad\text{as
    }n\rightarrow\infty.
  \end{equation*}
  This completes the proof of Proposition~\ref{thm:asymptotic_dist_h0_tot2}.
\end{proof}

\begin{proof}[Proposition~\ref{thm:asymptotic_dist_ha2}]
  We begin by proving \eqref{eq:lowerboundpart2a}. By the
  representation in Lemma~\ref{thm:expansion_T2} and since there
  exists a constant $c>0$ such that $\norm{\sresid_n}_2\geq c$ it holds
  for all $t\geq 0$ that
  \begin{equation*}
    \P\left(\frac{1}{\omega_n}\abs{T^1_{f_n,g_n}(\sresid_n)}\leq t\right)
    \leq \P\left(\norm{\hat{\beta}_{f_n}-\hat{\beta}_{g_n}}_2\leq \frac{t\omega_n}{c}\right).
  \end{equation*}
  Using the inequalities
  $\norm{\hat{\beta}_{f_n}-\hat{\beta}_{g_n}}_2\geq b_n-\norm{\hat{\beta}_{f_n}-\beta_{f_n}}_2-\norm{\hat{\beta}_{g_n}-\beta_{g_n}}_2$ and
  $\norm{\hat{\beta}_{f_n}-\hat{\beta}_{g_n}}_2\leq
  \norm{\hat{\beta}_{f_n}-\beta_{f_n}}_2+\norm{\hat{\beta}_{g_n}-\beta_{g_n}}_2+b_n$
  we can derive the following two inequalities
  \begin{equation*}
    \E\left(\norm{\hat{\beta}_{f_n}-\hat{\beta}_{g_n}}_2\right)\geq
    b_n-\frac{C_1}{\sqrt{r_n}}
    \quad\text{and}\quad
    \E\left(\norm{\hat{\beta}_{f_n}-\hat{\beta}_{g_n}}_2^2\right)\leq \frac{C_2}{r_n}+\frac{C_3b_n}{\sqrt{r_n}}+b_n^2.
  \end{equation*}
  As in the proof of Proposition~\ref{thm:asymptotic_dist_ha} we can
  apply the Paley-Zygmund inequality and use that
  $\omega_n=\landauo(a_n)$ to show that
  \begin{equation*}
    \lim_{n\rightarrow\infty}\P\left(\norm{\hat{\beta}_{f_n}-\hat{\beta}_{g_n}}_2\leq \frac{t\omega_n}{c}\right)=0.
  \end{equation*}
  This completes the proof of \eqref{eq:lowerboundpart2a}.

  Next, we prove \eqref{eq:lowerboundpart2b}. Since $\sigma_{g_n}^2$
  converges to a positive constant and $\omega_n$ converges to zero,
  there exists $c,C,\delta>0$ and $n_0\in\N$ such that for all
  $n\in\{n_0,n_0+1,\dots\}$ it holds
  that $$0<c+\delta<\sigma^2_{g_n}<C-\delta\quad\text{and}\quad \omega_n<1.$$ Let
  $n\in\{n_0,n_0+1,\dots\}$, then on the event
  $\{\abs{V_{f_n}-\sigma_{f_n}^2}\leq\delta\omega_n\}\cup\{\abs{V_{g_n}-\sigma_{g_n}^2}\leq\delta\omega_n\}$ it holds that
  \begin{equation*}
    \abs{T_{f_n,g_n}^2}
    =\frac{\abs{V_{f_n}-V_{g_n}}}{\abs{V_{g_n}}}
    \geq\frac{a_n-\abs{V_{f_n}-\sigma^2_{f_n}}-\abs{V_{g_n}-\sigma^2_{g_n}}}{\abs{V_{g_n}}}
    \geq\frac{a_n}{C}-\frac{2\delta\omega_n}{c},
  \end{equation*}
  which in particular implies for all $t\geq 0$ that
  \begin{equation*}
    \P\left(\tfrac{1}{\omega_n}\abs{T_{f_n,g_n}^2}\leq t\right)
    \leq\mathds{1}_{\left\{\frac{a_n}{\omega_n
          C}-\frac{2\delta}{c}\leq t\right\}}
    +\P\left(\abs{V_{f_n}-\sigma_{f_n}^2}\leq\delta\omega_n\right)
    +\P\left(\abs{V_{g_n}-\sigma_{g_n}^2}\leq\delta\omega_n\right).
  \end{equation*}
  In the proof of Proposition~\ref{thm:asymptotic_dist_ha} (see
  \eqref{eq:vpartasr}) we showed that $\abs{V_{f_n}-\sigma_{f_n}^2}=\landauop(\omega_n)$ and
  $\abs{V_{g_n}-\sigma_{g_n}^2}=\landauop(\omega_n)$. Therefore, using
  the assumption $\omega_n=\landauo(a_n)$ this implies that
  \begin{equation*}
    \lim_{n\rightarrow\infty}\P\left(\tfrac{1}{\omega_n}\abs{T_{f_n,g_n}^2}\leq t\right)=0,
  \end{equation*}
  which completes the proof of Proposition~\ref{thm:asymptotic_dist_ha2}.
\end{proof}

\subsection{Corollary~\ref{thm:consistent_method}}\label{sec:proof_corollary}

\begin{proof}[Corollary~\ref{thm:consistent_method}]
  Based on the empirical coverage property given in
  Proposition~\ref{thm:empirical_coverage} it holds that
  \begin{align}
    \P\left(\hat{S}(\varphi_{n,B})=\tilde{S}\right)
    &=1-\P\left(\hat{S}(\varphi_{n,B})\supsetneq\tilde{S}\right)-\P\left(\hat{S}(\varphi_{n,B})\subsetneq\tilde{S}\right)\nonumber\\
    &=\P\left(\hat{S}(\varphi_{n,B})\subseteq\tilde{S}\right)-\P\left(\hat{S}(\varphi_{n,B})\subsetneq\tilde{S}\right)\nonumber\\
    &\geq 1-\alpha-\P\left(\hat{S}(\varphi_{n,B})\subsetneq\tilde{S}\right).\label{eq:firstestimateaa}
  \end{align}
  Moreover, using the union bound we get that
  \begin{align}
    \P\left(\hat{S}(\varphi_{n,B})\subsetneq\tilde{S}\right)
    &=\P\left(\textstyle{\bigcup_{\overset{S\subsetneq\tilde{S}:}{\HOS\text{ false}}}}\left\{\varphi_{n,B}^S=0\right\}\right)\nonumber\\
    &\leq\sum_{\overset{S\subsetneq\tilde{S}:}{\HOS\text{
      false}}}\P\left(\varphi_{n,B}^S=0\right)\nonumber\\
    &\leq 2^d\cdot\max_{\overset{S\subsetneq\tilde{S}:}{\HOS\text{ false}}}\P\left(\varphi_{n,B}^S=0\right).\label{eq:secondestimatebb}
  \end{align}
  Finally, using Theorem~\ref{thm:consistent_test_bonferroni} and combining
  \eqref{eq:firstestimateaa} with \eqref{eq:secondestimatebb} it holds
  that
  \begin{equation*}
    \lim_{n\rightarrow\infty}\lim_{B\rightarrow\infty}\P\left(\hat{S}(\varphi_{n,B})=\tilde{S}\right)
    \geq 1-\alpha-\lim_{n\rightarrow\infty}\lim_{B\rightarrow\infty}\left[2^d\cdot\max_{\overset{S\subsetneq\tilde{S}:}{\HOS\text{
          false}}}\P\left(\varphi_{n,B}^S=0\right)\right]
    = 1-\alpha,
  \end{equation*}
  which completes the proof of Corollary~\ref{thm:consistent_method}.
\end{proof}

\subsection{Extension to uniform consistency}\label{sec:uniform_consistency}

As discussed in Remark~\ref{rmk:uniform_consistency} our asymptotic
consistency results can be extended to hold uniformly. The precise
statement is given in the following theorem.

\begin{theorem}[uniform rate consistency]
  \label{thm:uniform_consistent_test}
  Assume Assumption~\ref{assumption:asymptotic_cpmodel}
  and~\ref{assumption:normal}, let $S\subseteq\{1,\dots,d\}$ and let
  $(\mathcal{E}_n)_{n\in\N}$ be a sequence of collections of pairwise
  disjoint non-empty environments with the properties
  \hyperref[item:c1]{(C1)}, \hyperref[item:c2]{(C2)} and
  \hyperref[item:c3]{(C3,k)} where condition \hyperref[item:c2]{(C2)}
  is extended to ensure that the variances are uniformly bounded,
  i.e.,
  \begin{equation*}
    0<c\leq\inf_{(\vY_n,\vX_n)\in\bar{H}^{n}_{A,S}(\bar{a}_n,\bar{b}_n)}\lim_{n\rightarrow\infty}\sigma^2_{e_n}
    \leq\sup_{(\vY_n,\vX_n)\in\bar{H}^{n}_{A,S}(\bar{a}_n,\bar{b}_n)}\lim_{n\rightarrow\infty}\sigma^2_{e_n}\leq
    C<\infty
  \end{equation*}
  (the bounds in condition \hyperref[item:c3]{(C3,k)} is already
  uniform across $\bar{H}^{n}_{A,S}(\bar{a}_n,\bar{b}_n)$). Moreover,
  assume that for all $n\in\N$ it holds that
  $(\vY_n,\vX_n)\sim P_n\in\bar{H}_{A,S}^n(\bar{a}_n,\bar{b}_n)$, where
  $\bar{a}_n$ and $\bar{b}_n$ satisfy the following condition
  \begin{equation*}
    \frac{\abs{\mathcal{E}_n}^{\tfrac{1}{k}}}{\sqrt{r_n}}=\landauo(\bar{a}_n)
    \quad\text{or}\quad
    \frac{\abs{\mathcal{E}_n}^{\tfrac{1}{k}}}{\sqrt{r_n}}=\landauo(\bar{b}_n^2).
  \end{equation*}
  Then it holds that
  \begin{equation*}
    \lim_{n\rightarrow\infty}\lim_{B\rightarrow\infty}\inf_{P_n\in\bar{H}^{n}_{A,S}(\bar{a}_n,\bar{b}_n)}\P_{P_n}\left(\varphi_{B}(\vY_n,\vX_n^S)=1\right)=1,
  \end{equation*}
  where $\varphi$ is either the combined or the decoupled test.
\end{theorem}
The proof is a straight forward extension of the proofs given in
Sections~\ref{sec:proof_consistent} and~\ref{sec:proof_consistent2}.
We illustrate the changes only for the combined test. Similar
arguments can be applied to the decoupled test. In order to prove the
result for the combined test we can use the exact same proof as for
Theorem~\ref{thm:consistent_test} with the exception that we need to
strength the result from Proposition~\ref{thm:asymptotic_dist_ha}. In
particular we need to extend equation \eqref{eq:lowerboundpart} in the
following way,
\begin{equation*}
  \lim_{n\rightarrow\infty}\sup_{P_n\in\bar{H}^{n}_{A,S}(\bar{a}_n,\bar{b}_n)}\P\left(\tfrac{1}{\omega_n}\abs[\big]{T_{f_n,g_n}^3(\sresid_n)}\leq t\right)=0.
\end{equation*}
This can be accomplished by ensuring that all constants appearing in
the proof of Proposition~\ref{thm:asymptotic_dist_ha} hold uniformly
across all potential alternatives in
$\bar{H}^{n}_{A,S}(\bar{a}_n,\bar{b}_n)$. In particular, we need
to make sure this holds for all constants appearing in the inequalities
\eqref{item:term1ha}, \eqref{item:term2ha}, \eqref{item:term3ha},
\eqref{item:term4ha} and \eqref{item:term5ha}.  This is, however,
immediate given the additional assumption
\begin{equation*}
  0<c\leq\inf_{P_n\in\bar{H}^{n}_{A,S}(\bar{a}_n,\bar{b}_n)}\lim_{n\rightarrow\infty}\sigma^2_{e_n}
  \leq\sup_{P_n\in\bar{H}^{n}_{A,S}(\bar{a}_n,\bar{b}_n)}\lim_{n\rightarrow\infty}\sigma^2_{e_n}\leq
  C<\infty.
\end{equation*}

\subsection{Proposition~\ref{thm:srtest_level2}}\label{sec:tslevelproof}

\begin{proof}
  Throughout the proof we use the notation
  $\vZ\coloneqq(Z_{p+1},Z_{p+2},\dots,Z_n)^{\top}\in\R^{(n-p)\times
    (d+p(d+1))}$ where
  $Z_t=(X_t,Y_{t-1},X_{t-1}\dots,Y_{t-p},X_{t-p})$. For a fixed,
  significance level $\alpha\in(0,1)$ we construct our test for the
  time series setting as follows. Let
  $\snoise_1,\snoise_2,\dots\iid\mathcal{N}(\mathbf{0},\vI_{n-p})$,
  then given that $(\vY,\vX)$ satisfies $\P^{(\vY,\vX)}\in\tildeHOS$
  it holds for all $\vz\in\R^{(n-p)\times (d+p(d+1))}$ and for all
  $i\in\N$ that the random variables defined by
  \begin{equation*}
    \sresid_i^{S,p,\vz}\coloneqq\frac{(\vI-\vP^{S,p}_{\vz})\snoise_i}{\norm{(\vI-\vP^{S,p}_{\vz})\snoise_i}_2},
  \end{equation*}
  are i.i.d. copies of $\sresid^{S,p}\mid\vZ=\vz$, where
  $\sresid^{S,p}$ are the scaled residuals defined in
  \eqref{eq:scaled_residZ}. To see this, use the properties of the
  projection matrix $\vP^{S,p}_{\vZ}$ and the fact that
  $\vY=\vZ^{S,p}\eta+\noise$ for some
  $\eta\in\R^{(\abs{S}+p(d+1))\times 1}$.

  Next, for all $B\in\N$ define the cut-off functions
  $c_{T,B}:\R^{(n-p)\times (d+p(d+1))}\rightarrow\R$ given for all
  $\vz\in\R^{(n-p)\times (d+p(d+1))}$ by
  \begin{equation*}
    c_{T,B}(\vz)\coloneqq\lceil B(1-\alpha)\rceil\text{-largest value
      of } \abs{T(\sresid_1^{S,p,\vz})},\dots,\abs{T(\sresid_B^{S,p,\vz})}.
  \end{equation*}
  Then, our hypothesis tests $(\varphi^{S,p}_{T,B})_{B\in\N}$ are defined for all
  $B\in\N$ by
  \begin{equation*}
    \varphi^{S,p}_{T,B}(\vY,\vX)\coloneqq\mathds{1}_{\{\abs{T(\sresid^{S,p})}>c_{T,B}(\vZ)\}}.
  \end{equation*}
  Using the well established fact \citep[see e.g.][Example
  11.2.13]{lehmann2005} that the quantiles of the empirical
  distribution converge to the quantiles of the true distribution, we
  get $\P$-a.s. for all $\vz\in\R^{(n-p)\times (d+p(d+1))}$ that
  \begin{equation*}
    \lim_{B\rightarrow\infty}c_{T,B}(\vz)=F_{T\left(\sresid^{S,p,\vz}_1\right)}^{-1}(1-\alpha),
  \end{equation*}
  where $F_{T(\sresid^{S,p,\vz}_1)}^{-1}$ is the quantile function of
  the random variable $T(\sresid^{S,p,\vz}_1)$.  Hence, conditioning on
  $\vZ$ leads to
  \begin{align*}
    \lim_{B\rightarrow\infty}\P\left(\varphi^{S,p}_{T,B}(\vY,\vX)=1\right)
    &=\E\left(\E\left(\lim_{B\rightarrow\infty}\varphi^{S,p}_{T,B}(\vY,\vX)\big\rvert\vZ\right)\right)\\
    &=\E\left(\E\left(\mathds{1}_{\{\abs{T(\sresid^{S,p})}>F_{T(\sresid^{S,p,\vZ}_1)}^{-1}(1-\alpha)\}}\Big\rvert\vZ\right)\right)\\
    &=\alpha,
  \end{align*}
  which completes the proof of Proposition~\ref{thm:srtest_level2}.
\end{proof}

\pagebreak
\section{Auxiliary results}\label{sec:auxiliary}

\begin{theorem}[convergence of OLS-estimates (random design)]
  \label{thm:ols_convergence_strong}
  Let
  $(Y_{n,i},X_{n,i})_{i\in\{1,\dots,n\},n\in\N}\subseteq\R\times\R^d$,
  $(\epsilon_{n,i})_{i\in\{1,\dots,n\},n\in\N}\subseteq\R$,
  $(\sigma_n)_{n\in\N}\subseteq\R_{>0}$ and
  $(\beta_n)_{n\in\N}\subseteq\R$ satisfy for all $n\in\N$ and for all
  $i\in\{1,\dots,n\}$ that
  \begin{equation*}
    Y_{n,i}=\beta_n X_{n,i}+\epsilon_{n,i}\quad\text{and}\quad
    \epsilon_{n,i}\independent X_{n,i},
  \end{equation*}
  where $\vepsilon_n\sim\mathcal{N}(\vNull,\sigma_n^2\vI)$ and $\frac{1}{n}\vX_n^{\top}\vX_n$
  is $\P$-a.s. invertible for $n$ sufficiently large. Additionally,
  let $k\in\N$ and assume there exists a constant $c>0$ such that for
  all $n\in\N$ it holds that
  \begin{equation*}
    c\leq\E\left(\abs[\big]{\lambda_{\min}(\tfrac{1}{n}\vX_n^{\top}\vX_n^{\top})}^k\right)<\infty.
  \end{equation*}
  Moreover, let $\hbeta_n$ denote the OLS-estimator of
  $\beta_n$ based on
  $(Y_{n,1},X_{n,1}),\dots,(Y_{n,n},X_{n,n})$. Then, for all $n\in\N$
  there exists a constant $C_1,C_2>0$ (depending on $k$) such that
  \begin{equation}
    \label{eq:twoinequalities}
    \E\left(\norm{\hbeta_n-\beta_n}_{\R^d}^{2k}\right)\leq
    \frac{C_2}{n^k}.    
    \quad\text{and}\quad
    \E\left(\norm{\vX_n(\hbeta_n-\beta_n)}_{\R^n}^{2k}\right)\leq C_1
  \end{equation}
  In particular, it holds for all $p\in\{1,\dots,2k\}$ that
  \begin{equation*}
    \norm{\hbeta_n-\beta_n}_{\operatorname{L}^p}=\landauOLp\left(\tfrac{1}{\sqrt{n}}\right)\quad\text{as
    } n\rightarrow\infty.
  \end{equation*}
\end{theorem}

\begin{proof}
  First, note that given the above assumptions the OLS-estimator can
  be expressed as
  \begin{equation}
    \label{eq:hbetambeta}
    \hbeta_n
    =(\vX_n^{\top}\vX_n)^{-1}\vX_n^{\top}\vY_n
    =\beta_n+\left(\vX_n^{\top}\vX_n\right)^{-1}\vX_n^{\top}\vepsilon_n.
  \end{equation}
  Using this combined with a spectral estimate we get that,
  \begin{align}
    \E\left(\norm{\hbeta_n-\beta_n}^{2k}_{\R^d}\right)
    &=\E\left(\norm[\big]{\left(\vX_n^{\top}\vX_n\right)^{-1}\vX_n^{\top}\vepsilon_n}^{2k}_{\R^d}\right)\nonumber\\
    &=\E\left(\left(\vepsilon_n^{\top}\vX_n\left(\vX_n^{\top}\vX_n\right)^{-1}\left(\vX_n^{\top}\vX_n\right)^{-1}\vX_n^{\top}\vepsilon_n\right)^k\right)\nonumber\\
    &\leq\E\left(\left(\lambda_{\max}\left(\left(\vX_n^{\top}\vX_n\right)^{-1}\right)\vepsilon_n^{\top}\vX_n\left(\vX_n^{\top}\vX_n\right)^{-1}\vX_n^{\top}\vepsilon_n\right)^k\right)\nonumber\\
    &=\E\left(\lambda_{\max}\left(\left(\vX_n^{\top}\vX_n\right)^{-1}\right)^k\norm[\big]{\vP_n\vepsilon_n}_{\R^n}^{2k}\right),\label{eq:expansion}
  \end{align}
  where
  $\vP_n\coloneqq\vX_n\left(\vX_n^{\top}\vX_n\right)^{-1}\vX_n^{\top}$
  is the projection matrix onto the column space of $\vX_n$. Let
  $\vV_n$ be a matrix where all columns together form an orthonormal basis of
  $\R^n$ and the first $d$ columns form an orthonormal basis of
  the column space of $\vX_n$. Then it in particular holds that
  \begin{equation*}
    \vP_n=\vV_n\vD_n\vV_n^{\top},
    \quad\text{with}\quad
    \vD_n=
    \begin{pmatrix}
      \vI_{d} &\vNull\\
      \vNull &\vNull
    \end{pmatrix}
    \in\R^{n\times n}.
  \end{equation*}
  Moreover, the orthogonality of $\vV_n$ implies that the vector $\vepsilon^*_n\coloneqq
  \vV_n^{\top}\vepsilon_n$ is again $\mathcal{N}(\vNull,\sigma_n^2\vI)$
  distributed. We therefore get that
  \begin{align}
    \E\left(\left.\norm[\big]{\vP_n\vepsilon_n}_{\R^n}^{2k}\right\rvert
    \vX_n\right)
    &=\E\left(\left.\left[\vepsilon_n^{\top}\vV_n\vD_n\vV_n^{\top}\vepsilon_n\right]^k\right\rvert
      \vX_n\right)\nonumber\\
    &=\E\left(\left.\left[\sum_{i=1}^d(\vV_n^{\top}\vepsilon_n)^2_i\right]^k\right\rvert
      \vX_n\right)\nonumber\\
    &=\sum_{i_1,\dots,i_k=1}^d\E\left(\left.(\vepsilon_{i_1}^{*})^2\cdots(\vepsilon_{i_k}^{*})^2\right\rvert
      \vX_n\right)\nonumber\\
    &\leq d^kC_k,\label{eq:projection}
  \end{align}
  where in the last step we used that all moments of a normal distribution up to a fixed order
  $2k$ can be bounded by a constant $C_k$. Hence, combining
  \eqref{eq:expansion} and \eqref{eq:projection} and using the
  properties of the conditional expectation we get that
  \begin{align*}
    \E\left(\norm{\hbeta_n-\beta_n}^{2k}_{\R^d}\right)
    &=\E\left(\lambda_{\max}\left(\left(\vX_n^{\top}\vX_n\right)^{-1}\right)^k\E\left(\left.\norm[\big]{\vP_n\vepsilon_n}_{\R^n}^{2k}\right\rvert
      \vX_n\right)\right)\\
    &\leq\E\left(\lambda_{\max}\left(\left(\vX_n^{\top}\vX_n\right)^{-1}\right)^kd^kC_k\right)\\
    &=\E\left(\lambda_{\max}\left(\left(\tfrac{1}{n}\vX_n^{\top}\vX_n\right)^{-1}\right)^k\right)\frac{d^kC_k}{n^k}\\
    &\leq\frac{d^kC_k}{cn^k}.
  \end{align*}
  This proves the first inequality in \eqref{eq:twoinequalities}.

  Next, observe that again by \eqref{eq:hbetambeta} it holds that
  \begin{align}
    \E\left(\norm{\vX_n(\hbeta_n-\beta_n)}^{2k}_{\R^d}\right)
    &=\E\left(\norm[\big]{\vX_n\left(\vX_n^{\top}\vX_n\right)^{-1}\vX_n^{\top}\vepsilon_n}^{2k}_{\R^d}\right)\nonumber\\
    &=\E\left(\norm[\big]{\vP_n\vepsilon_n}_{\R^n}^{2k}\right).\label{eq:expansion2}
  \end{align}
  Thus, combining \eqref{eq:expansion2} with \eqref{eq:projection}
  proves the second inequality in \eqref{eq:twoinequalities}.
  
  The last part of the theorem is then an immediate consequence of Jensen's
  inequality, which thus completes the proof of
  Theorem~\ref{thm:ols_convergence_strong}.
\end{proof}

\begin{lemma}[convergence rate of fractions]
  \label{thm:convergence_rate_fractions}
  Let $(X_n)_{n\in\N},(Y_n)_{n\in\N}\subseteq\R$ be two sequences of
  random variables which satisfy that
  \begin{equation*}
    X_n=x_n+\landauOp(a_n)
    \quad\text{and}\quad
    Y_n=y_n+\landauOp(a_n),
  \end{equation*}
  where $(x_n)_{n\in\N}$, $(y_n)_{n\in\N}$ and $(a_n)_{n\in\N}$ are
  strictly positive, deterministic and convergent sequences satisfying
  that $\lim_{n\rightarrow\infty}x_n>0$,
  $\lim_{n\rightarrow\infty}y_n>0$ and $\lim_{n\rightarrow\infty}a_n=0$.
  Then, it holds that
  \begin{equation*}
    \frac{X_n}{Y_n}=\frac{x_n}{y_n}+\landauOp(a_n).
  \end{equation*} 
\end{lemma}

\begin{proof}
  Fix $\epsilon>0$, by assumption there exist $M_1,M_2>0$ and
  $n_0\in\N$ such that
  for all $n\in\{n_0,n_0+1,\dots\}$ it holds that
  \begin{equation}
    \label{eq:xnynbounded}
    \P\left(\tfrac{1}{a_n}\abs{X_n-x_n}>M_1\right)<\frac{\epsilon}{3}
    \quad\text{and}\quad
    \P\left(\tfrac{1}{a_n}\abs{Y_n-y_n}>M_2\right)<\frac{\epsilon}{3}.
  \end{equation}
  Using the convergence of the sequences $(x_n)_{n\in\N}$ and
  $(y_n)_{n\in\N}$, there exists $x^*,y^*>0$ and
  $n_1\in\{n_0,n_0+1\dots\}$ such that for all
  $n\in\{n_1,n_1+1,\dots\}$ it holds that $y^*<y_n$ and $x_n<x^*$.
  Moreover, fix $0<\delta<y^*$, then since
  $\lim_{n\rightarrow\infty}a_n=0$ there exists
  $n_2\in\{n_1,n_1+1,\dots\}$ such that for all
  $n\in\{n_2,n_2+1,\dots\}$ it holds that $a_nM_2<\delta$ and thus in
  particular
  \begin{equation}
    \label{eq:ynconvergent}
    \P\left(\abs{Y_n-y_n}>\delta\right)<\frac{\epsilon}{3}.
  \end{equation}
  Next, observe that for any $M>0$ it holds that
  \begin{align}
    &\P\left(\frac{1}{a_n}\abs[\bigg]{\frac{X_n}{Y_n}-\frac{x_n}{y_n}}>M\right)\nonumber\\
    &\quad=\P\left(\frac{1}{a_n}\abs[\bigg]{\frac{X_ny_n-Y_nx_n}{Y_ny_n}}>M,\abs{Y_n-y_n}\leq\delta\right)
      +\P\left(\frac{1}{a_n}\abs[\bigg]{\frac{X_n}{Y_n}-\frac{x_n}{y_n}}>M,\abs{Y_n-y_n}>\delta\right)\nonumber\\
    &\quad\leq\P\left(\frac{1}{a_n}\frac{\abs{X_ny_n-Y_nx_n}}{(y_n-\delta)y_n}>M\right)
      +\P\left(\abs{Y_n-y_n}>\delta\right)\nonumber\\
    &\quad=\P\left(\tfrac{1}{a_n}\abs{X_n-Y_n\tfrac{x_n}{y_n}}>M(y_n-\delta)\right)
      +\P\left(\abs{Y_n-y_n}>\delta\right)\nonumber\\
    &\quad\leq\P\left(\tfrac{1}{a_n}\abs{X_n-x_n}+\tfrac{1}{a_n}\abs{Y_n\tfrac{x_n}{y_n}-x_n}>M(y_n-\delta)\right)
      +\P\left(\abs{Y_n-y_n}>\delta\right)\nonumber\\
    &\quad\leq\P\left(\tfrac{1}{a_n}\abs{X_n-x_n}>\tfrac{M}{2}(y_n-\delta)\right)+\P\left(\tfrac{1}{a_n}\abs{Y_n-y_n}>\tfrac{My_n}{2x_n}(y_n-\delta)\right)
      +\P\left(\abs{Y_n-y_n}>\delta\right).\label{eq:xnynestimate}
  \end{align}
  Therefore, combining \eqref{eq:xnynbounded}, \eqref{eq:ynconvergent}
  and \eqref{eq:xnynestimate} with
  $M>\max\{\frac{2M_1}{y^*-\delta},\frac{2M_2x^*}{y^*(y^*-\delta)}\}$
  we get for all $n\in\{n_2,n_2+1,\dots\}$ that
  \begin{equation*}
    \P\left(\frac{1}{a_n}\abs[\bigg]{\frac{X_n}{Y_n}-\frac{x_n}{y_n}}>M\right)
    <\frac{\epsilon}{3}+\frac{\epsilon}{3}+\frac{\epsilon}{3}=\epsilon,
  \end{equation*}
  which completes the proof of Lemma~\ref{thm:convergence_rate_fractions}.
\end{proof}

\pagebreak
\section{Monetary policy data set}\label{sec:dataset}

The data set used in Section~\ref{sec:expreal} and described in
Table~\ref{table:variables} has been gathered from three
different sources, as follows
\begin{itemize}
\item quarterly GDP data for Switzerland from \citet{fedreserve1}
\item quarterly GDP data for Euro states from \citet{fedreserve2}
\item monthly business confidence index (BCI) for Switzerland from \citet{OECD1}
\item monthly consumer price index (CPI) for Switzerland from \citet{OECD2}
\item monthly balance sheet data SNB from \citet{snb}
\item monthly call money rate SNB from \citet{snb}
\item monthly average exchange rates CHF from \citet{snb}.
\end{itemize}
From each of these we took data from January 1999 to January 2017 and
performed the transformation described in Table~\ref{table:variables}.

\fi

\end{document}